\documentclass[pdflatex,sn-mathphys-num]{sn-jnl}% Math and Physical Sciences Numbered Reference Style
\usepackage{graphicx}%
\usepackage{multirow}%
\usepackage{amsmath,amssymb,amsfonts}%
\usepackage{amsthm}%
\usepackage{mathrsfs}%
\usepackage[title]{appendix}%
\usepackage{xcolor}%
\usepackage{textcomp}%
\usepackage{manyfoot}%
\usepackage{booktabs}%
\usepackage{algorithm}%
\usepackage{algorithmicx}%
\usepackage{algpseudocode}%
\usepackage{listings}%
\usepackage{diagbox}
\usepackage{multirow}
\numberwithin{equation}{section}

\theoremstyle{thmstyleone}%
\newtheorem{theorem}{Theorem}%  meant for continuous numbers
\newtheorem{proposition}[theorem]{Proposition}% 
\newtheorem{corollary}[theorem]{Corollary}
\newtheorem{lemma}[theorem]{Lemma}

\theoremstyle{thmstyletwo}%
\newtheorem{remark}{Remark}%
\newtheorem{assumption}{Assumption}

\theoremstyle{thmstylethree}%

\begin{document}

\title[SSAV for stochastic CH equation]{A scalar auxiliary variable-based semi-implicit scheme for stochastic Cahn--Hilliard equation}


\author[1]{\fnm{Jianbo} \sur{Cui}}\email{jianbo.cui@polyu.edu.hk}

\author[2]{\fnm{Jie} \sur{Shen}}\email{jshen@eitech.edu.cn}
\equalcont{These authors contributed equally to this work.}

\author[1]{\fnm{Derui} \sur{Sheng}}\email{sdr@lsec.cc.ac.cn}

\author[1]{\fnm{Yahong} \sur{Xiang}}\email{xiangyahong2025@outlook.com}
\equalcont{These authors contributed equally to this work.}

\affil[1]{\orgdiv{Department of Applied Mathematics}, \orgname{The Hong Kong Polytechnic University}, \orgaddress{\street{Hung Hom}, \city{Kowloon}, \country{Hong Kong}}}

\affil[2]{\orgdiv{School of Mathematical Sciences}, \orgname{Eastern Institute of Technology}, \orgaddress{ \city{Ningbo}, \postcode{315200}, \state{Zhejiang}, \country{People's Republic of China}}}

%
%
%\affil[3]{\orgdiv{Department}, \orgname{Organization}, \orgaddress{\street{Street}, \city{City}, \postcode{610101}, \state{State}, \country{Country}}}


\abstract{
	In this paper, we present a novel semi-implicit numerical scheme for the stochastic Cahn--Hilliard equation driven by multiplicative noise. By reformulating the original equation into an equivalent  stochastic scalar auxiliary variable (SSAV) system, our method enables an efficient and stable treatment of polynomial nonlinearities in a semi-implicit fashion. In order to accurately capture the impact of stochastic perturbations, we carefully incorporate It\^o correction terms into the SSAV approximation. Leveraging the smoothing properties of the underlying semigroup  and the $H^{-1}$-dissipative structure of the nonlinear term, we establish  the optimal strong convergence order of one-half for the proposed scheme in the trace-class noise case. Moreover, we show that the modified SAV energy asymptotically preserves the energy evolution law. 
	Finally, numerical experiments are provided to validate the theoretical results and to explore the influence of noise near the sharp-interface limit.}

\keywords{Stochastic scalar auxiliary variable,  stochastic Cahn--Hilliard equation, strong convergence,  energy evolution law}



\maketitle

\section{Introduction}
The Cahn--Hilliard equation was initially introduced  to describe phase separation dynamics in binary metallic alloys \cite{cahn1958free}. 
By incorporating stochastic perturbations representing thermal fluctuations and random solute vibrations \cite{SL18}, the stochastic Cahn--Hilliard equation provides a more realistic description of microstructure evolution. {This stochastic formulation has been applied broadly, including to nucleation dynamics \cite{BSW16}, spinodal decomposition \cite{BMW01,MSW00}, coarsening phenomena \cite{DD16}, cell proliferation and adhesion \cite{KS08}, bubble motion \cite{BS20}, and related processes.}

In this paper,
we consider the following stochastic Cahn--Hilliard equation
\begin{equation}\label{model}
	\begin{aligned}
		d\phi(t) &= \Delta \mu(t)\,dt + g(\phi(t))\,dW(t), \qquad &&\text{in } \mathcal{O}\times(0,T],\\
		\mu(t) &:= \frac{\delta E(\phi(t))}{\delta \phi}
		= -\Delta \phi(t) + F'(\phi(t)), \qquad &&\text{in } \mathcal{O}\times(0,T],
	\end{aligned}
\end{equation}
%{\b add the t variable for this kind equation in the rest of the paper. this line still exceeds the margin limit}
subject to homogeneous Dirichlet boundary conditions $\phi = \mu = 0$ on $\partial\mathcal{O}\times(0,T]$ and the non-random initial condition $\phi(0)=\phi^0$.  The spatial domain $\mathcal{O}\subset \mathbb{R}^d$, $d\in\{1,2,3\}$, is a bounded domain with smooth boundary  or a bounded convex domain with polygonal boundary. The Dirichlet problem is physically relevant because it governs the propagation of a solidification front into an ambient medium at rest relative to the front \cite{DN91}; see also \cite{CH20,EL92}.
The unknown $\phi$ denotes the order parameter {for phase transitions}, and the chemical potential $\mu$ is defined as the functional derivative of the Ginzburg--Landau free energy functional
\begin{align}\label{energy_def}
	E(\phi(t)) = \int_{\mathcal{O}}\Big[\frac12|\nabla \phi(t,x)|^2+f(\phi(t,x))\Big]dx
\end{align}
where
$f$ is a quartic polynomial, typically chosen as the double-well potential $f(\phi)=\frac{1}{4}(\phi^2-1)^2$. In the sequel, $F'$ and $g$ denote the Nemitskii operators associated with $f'$ and a bounded function $\sigma:\mathbb R\to\mathbb R$, respectively, and $W$ is a spatially homogeneous Wiener process of trace class (see Assumptions \ref{assum2} and \ref{assum1}). For the results on well-posedness and {regularity} of \eqref{model}, we refer to 
\cite{cui2023wellposedness,da2004irregular,DN91} and references therein. In these studies, the energy evolution law serves as a fundamental tool (see, e.g., \cite{elezovic1991stochastic}), and has also been used to investigate the long-time dynamics, such as the existence of invariant measures \cite{da1996stochastic}, {ergodicity \cite{GM15}, and the existence} and strict positivity of the solution densities \cite{cardon2001cahn,CC02,CH20}, for the underlying system.

In recent years, the stochastic Cahn--Hilliard equation has become a focal point of numerical research, as the absence of closed-form analytical solutions necessitates the development of robust computational methods. A central challenge in numerical  discretizations for stochastic evolution equations lies in the treatment of time discretization \cite{PJ01}. In particular, the stochastic stability and convergence behave fundamentally differently from the deterministic case. For the stochastic Cahn--Hilliard equation with polynomial nonlinearities, a common approach to address stability is to use fully implicit time-stepping schemes; see, for example, \cite{CHS21,DN91,FKLL18,QW20} for additive noise and \cite{CH20,FLZ20,hong2024density,ZL22} for multiplicative noise.  Recently, using a truncation strategy, \cite{GoldysSoenjayaTran2026} developed a fully discrete scheme for a class of fourth-order stochastic partial differential equations driven by spatially smooth multiplicative noise, based on a partially implicit time-stepping combined with a mixed finite element method.
However, these schemes require solving a large nonlinear stochastic system at each step, which is computationally expensive in high dimensions or on fine meshes. Moreover, a unified error-analysis framework that combines numerical discretization and iteration errors is still lacking, leaving the overall simulation error unclear.

Another popular numerical approach to solve the stochastic Cahn--Hilliard equation is via explicit or semi-implicit schemes, which aims to improve efficiency while maintaining stability, often through structure-preserving or tamed strategies. For example, tamed-exponential Euler methods have been proposed for \eqref{model} in the additive-noise setting \cite{brehier2022weak}, but they may fail to capture  the energy evolution law. Splitting-based structure-preserving schemes offer another option and are effective for the stochastic Allen--Cahn equation \cite{BCH19,BG19}; however, they are ill-suited to the stochastic Cahn--Hilliard equation because  the corresponding subsystems may not be explicitly solvable. A further alternative is the stochastic scalar auxiliary variable (SSAV) approach, inspired by scalar auxiliary variable (SAV) methods for deterministic gradient flows \cite{shen2018convergence,shen2018scalar}. SSAV techniques have been successfully adapted to stochastic models, including stochastic nonlinear Klein--Gordon \cite{cui2025stochastic} and Allen--Cahn equations \cite{metzger2024convergent,metzger2025strong}. For the stochastic Cahn--Hilliard equation with dynamic boundary conditions, \cite{metzger2023convergent} proposed an SAV-based finite element scheme and proved convergence in distribution to suitable weak solutions, but without establishing strong convergence. 

It is known that the strong convergence order of numerical scheme is a key measure of pathwise accuracy and is crucial for the complexity analysis of Monte Carlo and multilevel Monte Carlo methods \cite{CM08}. Nonetheless, to the best of our knowledge, no explicit or semi-implicit scheme has been shown to achieve strong convergence while preserving the energy evolution law for the stochastic Cahn--Hilliard equation with non-globally Lipschitz nonlinearities. This gap motivates the present work.

The  numerical analysis of the SAV-based scheme for \eqref{model}  faces several challenges. First, unlike the deterministic case, the energy evolution is governed by the complex interplay between random noise and the nonlinear potential, which complicates the construction of the SSAV framework. Capturing the energy evolution
demands careful control over the approximations of both the solution and its gradient. Second, the low temporal regularity of the Wiener process prevents the direct application of standard SAV formulations to stochastic systems, necessitating appropriate adaptations. Third, in contrast to the stochastic Allen--Cahn equation, the strong convergence analysis for \eqref{model} is further complicated by the presence of an unbounded linear operator acting on the nonlinearity, which introduces additional technical difficulties in establishing  error estimates.
% Firstly, the evolution of energy is governed by the interplay between random noise and a nonlinear potential and it does not necessarily exhibit the monotonic dissipation observed in deterministic systems. Preserving the energy evolution law, either exactly or asymptotically, is non-trivial, as it requires accounting for the weak approximation of not only the solution itself but also its gradient.
% Second, due to the low temporal regularity of the Wiener process, the standard SAV scheme does not directly apply to the stochastic system in general.  
% Last, compared with the case of stochastic Allen--Cahn equations, significant difficulties arise in the strong convergence analysis of numerical scheme for \eqref{model} due to the presence of an unbounded linear operator in front of  the nonlinearity.

To address these challenges, we  introduce an SSAV
$
r(t): = \sqrt{E_{\textup{p}}(\phi(t))},
$ where $E_{\textup{p}}$ is the potential functional (see \eqref{def:ssav}), and reformulate  \eqref{model} into the following SSAV system 
\begin{equation}\label{SAVmodel-intro}
	d\phi(t) = -A^2\phi(t) dt  +\frac{r(t)}{\sqrt{E_{\textup{p}}({\phi(t)})}}A {F'}(\phi(t))  dt+g(\phi(t))d W(t), %\label{SAVmodel_1-intro}  
	%\\   
	% \label{eq:r(tn+1)-intro}
	% d r(t)  &= \frac{1}{2\sqrt{E_{\textup{p}}({\phi(t)})}} \left\langle{F'}(\phi(t)), d\phi(t) \right\rangle \\
	% &\quad - \frac{1}{8E_{\textup{p}}(\phi(t))^{3/2}}
	% \sum_{k=1}^{\infty} \left\langle {F'}(\phi(t)),g(\phi(t))Q^{\frac 12}e_k\right\rangle ^2
	% dt \notag\\
	% &\quad + \frac{1}{4\sqrt{E_{\textup{p}}(\phi(t)) }} \sum_{k=1}^{\infty}\left\langle F''(\phi(t))g(\phi(t))Q^{\frac 12}e_k,g(\phi(t))Q^{\frac 12}e_k\,\right\rangle\  dt\notag
\end{equation}
where $A$ is the Laplacian equipped with homogeneous  Dirichlet boundary conditions. 
By freezing the drift coefficient $\frac{r(t)}{\sqrt{E_{\textup{p}}({\phi(t)})}}{F'}(\phi(t))$ of \eqref{SAVmodel-intro} over each  subinterval $[t_n,t_{n+1}]$ by  a modified nonlinearity $\tilde{f}^n$ (to be specified later), we  propose the following exponential Euler SSAV scheme \begin{align}\label{schme_phi-intro}
	X^{n+1} = e^{-A^2\tau}X^n
	+ (I - e^{-A^2\tau})A^{-1}\tilde{f}^n +e^{-A^2\tau}g(X^n)\delta W^n.
\end{align}
Here, $\{t_n=n\tau\}_{n=0}^N$ denotes the temporal grid with the step size $\tau=T/N$, $N\in\mathbb{N}^+$, and $\delta W^n=W(t_{n+1})-W(t_n)$ is the Wiener increment.

The modified nonlinearity $\tilde f^n$ has to be carefully designed, as it involves the numerical approximation of the SSAV $r(t)$ whose evolution is intrinsically different from that of SAV in the deterministic setting (see, e.g., \cite{shen2018scalar}).
This can be observed in the following asymptotic expansion 
% Let $\{t_n=n\tau\}_{n=0}^N$ denote the temporal grid with the step size $\tau=T/N$, $N\in\mathbb{N}^+$.} Then
% by discarding terms of order $o(\tau)$, the SSAV satisfies that for any $n=0,1,\cdots,N-1$,
\begin{align}\label{SSAVr_TE-intro}
r(t)  &=r(s)+\frac{1}{2\sqrt{E_{\textup{p}}({\phi(s)})}}\left\langle F'(\phi(s)), \phi(t)-\phi(s) \right\rangle \\
&\quad - \frac{1}{8(E_{\textup{p}}(\phi(s)))^{3/2}}
\left\langle F'(\phi(s)),\phi(t)-\phi(s) \,\right\rangle^2
\notag\\\notag
&\quad + \frac{1}{4\sqrt{E_{\textup{p}}(\phi(s)) }} \left\langle F''(\phi(s))(\phi(t)-\phi(s)),\phi(t)-\phi(s) \,\right\rangle+o(|t-s|)
\end{align}
due to the It\^o--Taylor expansion {\cite[Section 5.5]{KP92} and low regularity structure of $W$.  Since \eqref{SSAVr_TE-intro} contains two additional It\^o correction terms,
the direct usage of It\^o--Taylor method for \eqref{SSAVr_TE-intro} destroys the linear structure of the SSAV system, making the resulting scheme cannot be solved explicitly.} %{\r Are you sure that this statement is rigorous?}
Inspired by \cite{metzger2023convergent}, we design the discrete SSAV as follows
%to involve both the increment of the numerical solution and the Wiener increment 
\begin{align}\label{scheme_r_sto-intro}
r^{n+1} &=r^n + \frac{1}{2\sqrt{E_{\textup{p}}(X^n) }} \left\langle F'(X^n),X^{n+1} - X^n\,\right\rangle\\
&\quad- \frac{1}{8(E_{\textup{p}}(X^n) )^{3/2}}
\left\langle F'(X^n),g(X^n)\delta W^n\,\right\rangle 
\left\langle F'(X^n),X^{n+1} - X^n\,\right\rangle \notag\\
&\quad + \frac{1}{4\sqrt{E_{\textup{p}}(X^n) }} \left\langle F''(X^n)(X^{n+1} - X^n),
g(X^n)\delta W^n\right\rangle\notag,
\end{align}
which maintains the scheme's linear structure while ensuring consistency with the It\^o formula. Accordingly, the nonlinearity $\tilde{f}^n$ is then determined via (see \eqref{fn})
\begin{align}\label{lemma1_scheme_r_sto_before-intro}
{ 2r^{n+1} ( r^{n+1}-r^{n})}
&=\langle \tilde{f}^n, X^{n+1}-X^{n}\rangle.
\end{align}
Although replacing $X^{n+1}-X^n$ by $g(X^n)\delta W^n$ in \eqref{scheme_r_sto-intro} also gives an explicit update for the discrete SSAV, it fails to fulfill the SAV structure \eqref{lemma1_scheme_r_sto_before-intro} (see Remark \ref{modify_reason}).
We would also like to mention that since the proposed SAV formulation is fundamentally energy-based, it may not be suitable for the stochastic evolution equation driven by space--time white noise (see, e.g., \cite{CC02,CHS21}). %{\r \cite{} some ref for SCH driven by space--time white noise}.

To overcome the difficulties arising from the superlinear nonlinearity, two key components are essential: obtaining regularity estimates of the numerical solution $X^n$, and a careful quantification of  the difference between the square root of the discrete potential energy 
$\sqrt{E_{\textup{p}}(X^n)}$ and the discrete SSAV $r^n$. These elements are crucial for establishing the sharp strong convergence rate of the proposed scheme  \eqref{schme_phi-intro}. On the one hand, 
the special algebraic structure \eqref{lemma1_scheme_r_sto_before-intro} enables us to control the polynomial growing drift term $\tilde{f}^n$, which in turn allows us to prove that the scheme \eqref{schme_phi-intro} is unconditionally stable in $H^1(\mathcal{O})$ (see {Lemma \ref{lemma_stability-pre}}). 
As a consequence, we can show the averaged energy evolution law of the exponential Euler SSAV scheme \eqref{schme_phi-intro}
(see Theorem~\ref{Sec:energy:pro}), which captures the effects of noise at a discrete level and asymptotically preserves the averaged energy evolution law of \eqref{model}.
Then, using the semigroup approach, we lift the regularity of the numerical solution to the 
$H^{\beta}(\mathcal{O})$ 
space for some $\beta>\frac{d}{2}$, yielding the $L^\infty(\mathcal{O})$-stability of the numerical solution.

On the other hand,  by utilizing the temporal H\"older continuity estimates of the numerical solution in both the 
$L^2(\mathcal{O})$ and $H^1(\mathcal{O})$ norms,
we prove that  (see Lemma \ref{lemma3})
\begin{equation*}
\mathbb{E}\left[|r^{n}-\sqrt{E_{\textup{p}}( X^{n})}|^p\right]\leq C(p)\tau^{\frac{p}{2}},\qquad p\ge1.
\end{equation*} 
As a comparison,   \cite[Lemma~6.4]{MS24} presented a convergence order of $1/8$ for this SAV quantity under dynamic boundary conditions.  
Furthermore, by exploiting the one-sided Lipschitz continuity of $-F'$ in $L^{2}(\mathcal{O})$ and the local Lipschitz continuity of $F'$ in $H^1(\mathcal{O})$, %{\r this means that we have $H^3$ regularity?},
we establish the following sharp strong convergence rate  (see Theorem~\ref{Thm3_main})
\begin{align*}
\mathbb{E}\left[\|\phi(t_n)-X^{n}\|^2\right]\leq C \tau,
\end{align*}
based on a coupling of variational and semigroup approaches. {The strong convergence order $1/2$ is optimal in the sense that it coincides with the temporal H\"older continuity exponent of the exact solution to \eqref{model}.}

In summary, the main contributions of this work are as follows:
\begin{enumerate}
	\item We develop a novel semi-implicit time discretization framework for the stochastic Cahn--Hilliard equation with multiplicative noise by integrating the exponential Euler method with the SAV approach. The resulting scheme is iteration-free and unconditionally stable.
	
	\item We establish the optimal strong convergence rate of the proposed scheme through a coupling of variational and semigroup approaches. This hybrid analytical framework not only overcomes the technical challenges posed by non-globally Lipschitz nonlinearity and multiplicative noise, but also can be extended to related models, such as the stochastic Allen--Cahn equation.
	
	\item We prove that the proposed scheme asymptotically preserves the averaged energy evolution law of the underlying continuous system, thereby revealing a deeper consistency between the discrete and continuous dynamics. Numerical experiments validate this theoretical finding and further demonstrate the divergence of the averaged energy evolution law for the standard SAV scheme.
	\end{enumerate}

	The remainder of this paper is organized as follows. In Section~\ref{section2}, we present the preliminaries and main results. Section~\ref{sec:Con} introduces the construction of the exponential Euler SSAV scheme \eqref{schme_phi-intro} and demonstrates its explicit solvability. In Section~\ref{sec3:regularity}, we establish regularity estimates for the numerical solution, which form the basis for the proofs of the main results given in Sections~\ref{section_converge} and~\ref{Sec:energy}. Finally, Section~\ref{section5} presents several numerical experiments that confirm both the accuracy and  efficiency of the proposed scheme,  as well as its ability to capture the sharp-interface dynamics. For the sake of completeness, the appendix provides proofs of several auxiliary results used in the paper.

	\section{Preliminaries and main results}\label{section2}

	In this section, we present the preliminaries and state the main result of this paper. In Subsection~\ref{subsec:Pre}, we specify the main assumptions on the coefficients \(F'\) and \(g\), as well as on the driving Wiener process \(W\), under which the stochastic Cahn--Hilliard equation \eqref{model} admits a unique mild solution. In Subsection~\ref{subsec:MR}, we present the optimal strong convergence order and the energy evolution law of the proposed scheme~\eqref{schme_phi-intro}.

	%
	%first introduce the main assumptions on the coefficients \( F' \) and \( g \), as well as on the driving Wiener process \( W \), under which the stochastic Cahn--Hilliard equation \eqref{model} admits a unique mild solution. Then we present the main results in subsection \ref{}.
	%
	%
	%to effectively simulate the mild solution to \eqref{model}, {\color{blue}we propose a semi-implicit, unconditionally stable scheme for \eqref{model} based on the SSAV approach, and establish the explicit solvability, 
%	
%	
%	optimal strong convergence order, and the energy evolution law of the numerical solution.}
%

We begin by introducing some useful notations. 
%{\r Let $a\vee b:=\max\{a,b\}$ for $a,b\in\mathbb{R}$.} 
For $p\in[1,\infty)$, denote 
by \(L^p(\mathcal{O}) \) the space of $p$th integrable functions  defined on $\mathcal{O}$, equipped with the norm $\|v\|_{L^p} :=(\int_{\mathcal{O}} |v(x)|^p \, dx)^{1/p}$ for $p\in[1,\infty)$. For $p=\infty$,  $L^\infty$ consists of measurable functions $v$ such that $\|v\|_{L^\infty}:= \operatorname{ess\,sup}_{x \in \mathcal{O}} |v(x)|<\infty$.
When \( p = 2 \), we write \( H := L^2(\mathcal{O}) \) and \( \|\cdot\| := \|\cdot\|_{L^2} \) for brevity. The symbol \( \langle \cdot, \cdot \rangle \)  represents the inner product in \( H \). Let \( \mathcal{L}(H) \) be the space of bounded linear operators from the Hilbert space \( H \) to itself. Given two Hilbert spaces $U$ and $V $, we define \( \mathcal{L}_2(U, V) \) as the space of Hilbert--Schmidt operators from \( U \) to \( V \), equipped with the norm
%\[
$\|\Psi\|_{\mathcal{L}_2(U,V)} := \left( \sum_{k=1}^{\infty} \|\Psi  f_k\|_{V}^2 \right)^{1/2}$ for $\Psi\in \mathcal{L}_2(U,V)$,
%\]
where \( \{f_k\}_{k \in \mathbb{N}^{+}} \) is a complete orthonormal basis of \( U \). 
Throughout the paper, \( C \) denotes a generic positive constant that may vary from one occurrence to another and is independent of the discretization parameter. This constant may depend on the terminal time \( T \), the initial value \( \phi^0 \), the covariance operator of $W$, and the coefficients \( F' \) and \( g \), but such dependence is not made explicit for simplicity.

\subsection{Preliminaries}\label{subsec:Pre}In this subsection, we introduce the main assumptions,  mild formulation, and averaged energy evolution law for the stochastic Cahn--Hilliard equation \eqref{model}.
% {\b at somewhere (maybe after the main result) we should mention why we choose this domain? My thinking: 1. for the Lipschitz domain, our temporal scheme still works for trace class noise case. 2. disadvantage, spectral Galerkin method may be not good choice, FDM or FEM may be better, the analysis needed to be modifed corresponding.}

% {\b the below is strange, why the domain is different from $[0,1]^d$?}

Since the spatial domain $\mathcal{O}\subset \mathbb{R}^d$, $d\in\{1,2,3\}$, is a bounded domain with smooth boundary $\partial\mathcal{O}$ or a bounded convex domain with polygonal boundary,  by \cite[Section 2.3]{KR14}, the Dirichlet Laplacian $A:D(A)\subset H\to H$ is densely defined, self-adjoint and positive definite with compact inverse. Hereafter, $D(\cdot)$ denotes the domain of an operator.
Then
%{\r Recall that \(A\) denotes} the Laplace operator on \(\mathcal O=(0,1)^d\), $d\in\{1,2,3\}$, with homogeneous Dirichlet boundary conditions.
there exists a non-decreasing sequence $\{\lambda_j\}_{j\ge1}\subset (0,\infty)$ and a complete orthonormal basis  \(\{e_j\}_{j\ge1}\subset D(A)\) of \(H\) such that 
$-Ae_j=\lambda_j e_j$  and $\lim_{j\to\infty}\lambda_j=\infty$.
For \(\alpha\in\mathbb{R}\), define the fractional power $(-A)^{\alpha}$ of the operator $-A$ by
$
(-A)^{\alpha}v=\sum_{j=1}^{\infty}\lambda_j^{\alpha}\,\langle v,e_j\rangle\,e_j
$
with the domain $D((-A)^{\alpha}):=\{v=\sum_{j=1}^\infty \langle v,e_j\rangle e_j:\sum_{j=1}^{\infty}\lambda_j^{2\alpha}\,\langle v,e_j\rangle^2<\infty\}$.
%For simplicity, we take $\mathcal{O}=(0,1)^d$ the unit cube in $\mathbb R^d$.
It is well known that \(-A^{2}\) generates an  \(C_{0}\)-semigroup \(\{S(t)=e^{-tA^{2}},t\ge 0\}\) of contractions on \(H\) (see, e.g., \cite{CH20}), i.e.,
$\|S(t)\|_{\mathcal{L}(H)}\leq 1$ for all $t\ge0$.
Furthermore, $\{S(t), t>0\}$ has the following smoothing effect (see, e.g., \cite[Lemma B.9]{KR14})
\begin{align}
\left\|(-A)^{\gamma}S(t)\right\|_{\mathcal{L}(H)}&\leq C(\gamma)t^{-\frac{\gamma}{2}},\qquad \gamma\geq 0,\label{assum1_eq1}\\
\left\|(-A)^{-\nu}(I-S(t))\right\|_{\mathcal{L}(H)}&\leq C(\nu)t^{\frac{\nu}{2}},\qquad\nu\in[0,2].\label{assum1_eq2}
\end{align}

For $s>0$, let $H^s(\mathcal{O})$ denote the standard Sobolev space with the norm $\|\cdot\|_{H^s}$, and for $s>\frac12$, $\dot{H}^s(\mathcal{O}):=\{v\in H^s(\mathcal{O}): v=0 \text{ on }\partial\mathcal{O}\}$.
According to \cite[Theorems 16.12 \& 16.13]{yagi2009abstract}, for $s\in(0,\tfrac12)$ the Sobolev norm $\|\cdot\|_{H^s}$ is equivalent to 
$\|(-A)^{\frac{s}{2}}\cdot\|$ on $H^s(\mathcal{O})$, 
and for $s\in(\tfrac12,\tfrac32)\cup(\tfrac32,2)$ the Sobolev norm $\|\cdot\|_{H^s}$ is equivalent to 
$\|(-A)^{\frac{s}{2}}\cdot\|$ on $\dot{H}^s(\mathcal{O})$. %{\color{red} question: what is your notation for fractional interpolation sobolev space? do we need it?}%{\color{blue}We use $D((-A)^{\frac s2})$ as the fractional interpolation sobolev space.}
We will also frequently use the 
equivalence between $\|\cdot\|_{H^1}$ and $\|\nabla \cdot\|$ on $\dot{H}^1(\mathcal{O})$, namely, there exists a constant $C>0$ such that
$C^{-1}\|\nabla u\|\le  \|u\|_{H^1}\le C  \|\nabla u\|$ for all $u\in \dot{H}^1(\mathcal{O})$, due to the Poincar\'e inequality. 
We note that the numerical analysis of the proposed scheme also extends to \eqref{model} with homogeneous Neumann boundary conditions by a slight modification. In this setting, one can first extract the constant mode to account for the kernel of the Neumann Laplacian, and then carry out the analysis on the mean-zero subspace (see, e.g., \cite{CHS21,QW20}).

Next, we specify the main assumptions on the drift and diffusion coefficients, as well as on the driving Wiener process.

%Our analysis relies on the following assumption for the drift term $F'$, with the typical example $F'(\phi)=\phi^3-\phi$ in mind.

\begin{assumption}\label{assum2}
Let 
$f(\xi)=c_1\xi^4+c_2\xi^3+c_3\xi^2$ with $c_1>0$ and $c_2,c_3\in\mathbb{R}$.
Assume that $F':L^{6}(\mathcal{O})\to H$ is the Nemytskii operator associated with $f':\mathbb{R}\to\mathbb{R}$, i.e.,
$F'(v)(x)= f'(v(x)), x\in\mathcal{O}$ for $v\in L^6(\mathcal{O})$.
\end{assumption}
%\begin{remark}\label{rem-coer} {\r better to remove the remark environment}

In the sequel, let $\phi^0\in \dot{H}^1(\mathcal{O})$ be non-random, under which
\begin{align}\label{pro:eq:phi0}
E(\phi^0) =\frac12\|\nabla\phi^0\|^{2}+\int_{\mathcal{O}}f(\phi^0(x))dx\le \frac12\|\nabla\phi^0\|^2+C\left(\|\phi^0\|_{L^4}^{4}+1\right),
\end{align}
in view of Assumption \ref{assum2} and the Sobolev embedding $H^1(\mathcal{O})\hookrightarrow L^4(\mathcal{O}) $.
Under Assumption \ref{assum2}, there exists a positive constant $L_f$ such that $-f^{\prime\prime}\le L_f$, which implies 
\begin{align}
-\langle F'(u)-F'(v),\,u-v\rangle &\le \,L_f\,\|u-v\|^{2}, \quad {u,v\in  L^{6}(\mathcal{O})}.\label{assum2_eq1}
\end{align}
It follows from the quadratic growth of $f^{\prime\prime}$ that
\begin{equation}\label{assum2_eq2}
\|F'(u)-F'(v)\| \le C\bigl(1+\|u\|_{L^\infty}^{2}+\|v\|_{L^\infty}^{2}\bigr)\,\|u-v\|, 
\quad {u,v\in L^\infty(\mathcal{O})}.
\end{equation}
Moreover, the potential energy functional 
$\int_{\mathcal{O}} f(v(x))dx$,  is coercive in $L^4(\mathcal{O})$; that is, there are positive constants $c$ and $\Theta$ such that
\begin{equation}\label{eq-coer}
\int_{\mathcal{O}} f(v(x))dx +\Theta\ge c\int_{\mathcal{O}}(|v(x)|^4+1)dx,\qquad v\in L^{4}(\mathcal{O}).
\end{equation}
%{\r why not $\int_{\mathcal{O}} f(v(x))dx +\Theta\ge c\int_{\mathcal{O}}|v(x)|^4dx$ for all $v\in L^{4}(\mathcal{O})$?}
%\end{remark}

Let \( W = \{W(t)\}_{t \geq 0} \)  in \eqref{model} be a \( Q \)-Wiener process defined on a  complete filtered probability space \( (\Omega, \mathcal{F}, \{\mathcal{F}_t\}_{t \geq 0}, \mathbb{P}) \), which admits the Karhunen--Lo\`eve expansion
$
W(t) = \sum_{k=1}^\infty Q^{1/2} e_k \beta_k(t)$ for $t\geq 0$.
Here \( \{\beta_k\}_{k \in \mathbb{N}^{+}} \) is a sequence of independent real-valued standard Brownian motions. Assume that $Q\in\mathcal{L}(H)$ satisfies $Qe_k=\mathsf{q}_ke_k$ for $k\ge1$, where $\mathsf{q}_k\ge0$. Denote by $Q^{1/2} H$ the image of $Q^{1/2}$ on $H$ endowed with the inner product $(v,w)_0:=\langle Q^{-\frac12}v ,Q^{-\frac12}w\rangle$ for $v,w\in Q^{1/2} H$, where $Q^{-\frac12}$ is the pseudo inverse of $Q^{\frac12}$.
%{\r pseduo inverse or inverse here?}
Then $\{Q^{1/2}e_k\}_{k=1}^\infty$ forms a complete orthonormal basis of $Q^{1/2} H$. 
We make the following assumption on the diffusion term of \eqref{model}.

\begin{assumption}\label{assum1}
The mapping \( g: H \to {\mathcal{L}_2^0}:={\mathcal{L}_2( Q^{1/2} H,H)} \) is defined by
%\begin{align}\label{gphi}
$(g(v)w)(x) := \sigma(v(x)) w(x)$ for $v\in H$ and $w \in Q^{1/2}H$,
%\end{align}
where \( \sigma: \mathbb{R} \to \mathbb{R} \) is a bounded and continuously differentiable function {with bounded derivative}. Moreover, 
$
\sum_{k=1}^{\infty} \|Q^{\frac{1}{2}} e_k\|^2_{L^{\infty}} + \sum_{k=1}^{\infty} \|\nabla (Q^{\frac{1}{2}} e_k)\|^2 \le C
$ for some $C>0$.
\end{assumption}

%{\r do we need use remark environment here?}
%\begin{remark}
The boundedness of $\sigma'$ implies that the mapping $g:H\to{\mathcal{L}_2^0}$ is Lipschitz continuous, i.e.,
\begin{align}\label{assume_Lip}
\|{g}({v})-{g}({w})\| _{\mathcal{L}_2^0} \leq C\|{v}-{w}\|, \qquad v,w\in H. 
\end{align}
In addition, under Assumption \ref{assum1}, 
$g:\dot{H}^1(\mathcal{O})\to \mathcal{L}_2(Q^{1/2}H,\dot{H}^1(\mathcal{O}))$ exhibits linear growth. %{\r what is $\dot H^1$?}
Indeed, by the boundedness of $\sigma$ and $\sigma'$, as well as the chain rule, for any $v\in \dot{H}^1(\mathcal{O})$,
\begin{align}\label{lema1_eq_R6-new}
&\|g(v)\|^2_{\mathcal{L}_2(Q^{\frac12}H,\dot{H}^1(\mathcal{O}))}=\sum_{k=1}^\infty\|(-A)^{\frac12}(g(v)Q^{\frac12}e_k)\|^2 \le C\sum_{k=1}^\infty\|\nabla(g(v)Q^{\frac12}e_k)\|^2\\\notag
&\le C\sum_{k=1}^\infty\|\sigma(v)\|_{L^\infty}^2\|\nabla(Q^{\frac12}e_k)\|^2+C\sum_{k=1}^\infty\|\sigma'(v)\|_{L^\infty}^2\|\nabla v\|^2\|Q^{\frac12}e_k\|_{L^\infty}^2\\\notag
&\le C(1+\|\nabla v\|^2).
\end{align}
%	Here in the second identity, we have also used $\|(-A)^{\frac12}\phi\|^2=\int_{\mathcal{O}}|\nabla\phi|^2dx$.
%	
%	
%	{\r where you use it?}{\color{blue}$\|(-A)^{\frac12}\phi\|^2=\langle -A\phi,\phi\rangle=\langle -\Delta\phi,\phi\rangle=\int_{\mathcal{O}}|\nabla\phi|^2dx$},
This property will play a crucial role in establishing the unconditional stability 
%and energy evolution law  %{\b why is important in the evolution law? it is inequality?}
of the numerical solution in $H^1(\mathcal{O})$ (see the proof of {Lemma \ref{lemma_stability-pre}} for details).
%\end{remark}

Under Assumptions \ref{assum2} and \ref{assum1}, the stochastic Cahn--Hilliard equation \eqref{model} admits a unique mild solution $\phi=\{\phi(t),t\in[0,T]\}$ given by (see, e.g., \cite{cardon2001cahn,cui2023wellposedness})
\begin{align*}%\label{mid}
\phi(t) = S(t)\phi^0+\int_{0}^{t} S(t-s)A{F'}(\phi(s))ds+\int_{0}^{t} S(t-s)g(\phi(s))dW(s).
\end{align*}

\begin{proposition}\label{App_prop}
Let Assumptions \ref{assum2} and \ref{assum1} hold, and let $\phi^0\in \dot{H}^{\beta}(\mathcal{O})$ for some $\beta\in[1,2)$.
%{\b why not $\beta=[1,2)$?}
Then for any $p\ge 1$, there exists a constant $C:=C(p,\beta)>0$ such that
\begin{align}\label{eq:phi(t)beta}
	\mathbb{E}\bigg[\sup_{t\in[0,T]}\|\phi(t)\|_{H^{\beta}}^{p}\bigg] \le C.
\end{align}

\end{proposition}
{We include the proof of Proposition \ref{App_prop} in Appendix \ref{proof-App_prop}  for the completeness}.
{Based on a standard finite-dimensional approximation argument (see e.g., \cite[section 2.3]{da1996stochastic}) and applying It\^o's formula, one can see that  the stochastic Cahn--Hilliard equation \eqref{model} satisfies the following energy evolution law} 
%{\r do we need say the rigorous proof need use finite dimensional approximation argument or cite some ref.} 
\begin{align}\label{dissi_conti}
dE(\phi(t)) &= -\left\|\nabla\mu(t)\right\|^{2}dt + \left\langle \mu(t), g(\phi(t))\,dW(t)\right\rangle  \\
&\quad + \frac{1}{2}\sum_{k=1}^{\infty} \left\langle \left(-A + F''(\phi(t))\right) g(\phi(t))Q^{\frac{1}{2}}e_k, g(\phi(t))Q^{\frac{1}{2}}e_k \right\rangle dt,\notag
\end{align}
where \(\mu(t)\) is the chemical potential given by \eqref{model}.
In this paper, we omit such standard  finite-dimensional approximation procedures for convenience.

\subsection{Main results}\label{subsec:MR}
%\subsubsection{Numerical scheme}

%    {\color{blue}In this subsection, we present the main results, 
%    It is shown that the proposed scheme is strong convergent with order $\frac12$ and the 
%    discrete energy evolution law asymptotically preserves the continuous energy evolution law \eqref{dissi_conti}.}

%propose a semi-implicit, energy-stable scheme for \eqref{model}, based on the SAV approach.

Our first main result, Theorem \ref{Thm3_main}, shows that the numerical solution associated with the proposed scheme \eqref{schme_phi-intro} is strongly convergent to the mild solution $\phi$ to \eqref{model}. We remark that the strong convergence order $1/2$ in Theorem  \ref{Thm3_main} is optimal in the sense that it coincides with the temporal H\"older continuity exponent of $\phi$.

\begin{theorem}\label{Thm3_main}
Let Assumptions \ref{assum2} and \ref{assum1} hold, and let $\phi^0\in \dot{H}^2(\mathcal{O})$.  Then for any $p \geq 1$, there exists a constant $C>0$ such that for any $n\in\{ 0,1,\cdots,N\}$,
\begin{align*}
	\mathbb{E}\left[\|\phi(t_n)-X^{n}\|^2\right]\leq C \tau.
\end{align*}
\end{theorem}

According to \eqref{dissi_conti}, the averaged energy evolution law of \eqref{model} reads
\begin{align*}
\mathbb{E}[E(\phi(t))] +\Theta&= E(\phi^0) +\Theta-\int_0^t\mathbb{E}[\left\|\nabla\mu(s)\right\|^{2}]ds \\
&\quad + \frac{1}{2}\int_0^t\mathbb{E}\bigg[\sum_{k=1}^{\infty} \left\langle \left(-A + F''(\phi(s))\right) g(\phi(s))Q^{\frac{1}{2}}e_k, g(\phi(s))Q^{\frac{1}{2}}e_k \right\rangle\bigg] ds\notag
\end{align*}
for any $t\in[0,T]$.
In our numerical study,
the original energy $E(\phi(t_n))+\Theta=\frac12\|\nabla \phi(t)\|^2+r(t_n)^2$ %{\r $r(t_n)^2$?}
is approximated by the following modified SAV energy
\begin{equation}\label{eq:Emodify}
E_{\mathrm{mod}}^{\,n}:=E_{\mathrm{mod}}(X^{n},r^{n})\quad \text{with}\quad
E_{\mathrm{mod}}(X,r):=\frac{1}{2}\,\|\nabla X\|^{2}+|r|^{2}.
\end{equation}
Our second main result is Theorem \ref{Sec:energy:pro} on the averaged evolution law of the modified SSAV energy.

\begin{theorem}\label{Sec:energy:pro}
Let Assumptions \ref{assum2} and \ref{assum1} hold, and let $\phi^0\in \dot{H}^{2}(\mathcal{O})$. Then 
for each \( m \in \{1,2, \cdots, N\} \),\begin{align}\notag
	\label{discrete_energy_law}
	\mathbb{E}\left[E^m_{
		\textup{mod}}\right]&=E^0_{
		\textup{mod}} -\frac{1}{2}\sum_{n=0}^{m-1}
	\mathbb{E}\left[\|(I-S^2(\tau))^{\frac{1}{2}}(-A)^{-\frac 12}\tilde{\mu}^n\|^2\right]\\
	&\quad+ \frac12\tau\sum_{n=0}^{m-1}\mathbb{E}\bigg[\sum_{k=1}^{\infty}\left\langle (-A+F''( X^{n}))g( X^n)Q^{\frac{1}{2}}e_k,g( X^n)Q^{\frac{1}{2}}e_k\right\rangle\bigg]+\mathcal{R}_m^{\tau},
\end{align}
where the modified chemical potential 
\begin{equation}\label{eq:tildemun}
	\tilde\mu^n:=-AX^n+\tilde{f}^n-A(g(X^n)\delta W^n),
\end{equation}
and
the remainder term $\mathcal{R}_m^{\tau}$ satisfies $\lim\limits_{\tau\to 0}\mathcal{R}_m^{\tau}=0$.      
\end{theorem}
Formally, the modified chemical potential \( \tilde\mu^n \) serves as a numerical approximation of the chemical potential
\(
\mu(t_n) = -A\phi(t_n) + F'(\phi(t_n)),
\)
which can be proved rigorously when the driving Wiener process possesses suitable spatial regularity.
Consequently, as the time step tends to zero, the averaged discrete energy evolution law \eqref{discrete_energy_law} for the modified SAV energy recovers the averaged energy evolution law of \eqref{model}.
Hence, the proposed scheme \eqref{schme_phi-intro} asymptotically preserves the averaged energy evolution law of the stochastic Cahn--Hilliard equation \eqref{model}.

%  {\r put main result here since this paper is too long}
\section{Exponential Euler SSAV scheme}\label{sec:Con}
In this section, we present the construction of the exponential Euler SSAV scheme \eqref{schme_phi-intro} and show that it is explicitly solvable.
Following \cite{cui2025stochastic,shen2018convergence,shen2018scalar}, the nonlinear term is treated explicitly through introducing an SSAV $r : [0,T]\times\Omega \to \mathbb{R}$, defined as
$
r(t,\omega) := \sqrt{E_{\textup{p}}(\phi(t,\omega))},
$
where $E_{\textup{p}}$ is the potential energy functional
\begin{equation}\label{def:ssav}
E_{\textup{p}}(v) := \int_{\mathcal{O}} f(v(x))\,dx + \Theta,
\qquad v \in L^{4}(\mathcal{O}).
\end{equation}
Here,
the constant $\Theta>0$ is the same as in \eqref{eq-coer} so that 
\begin{align}\label{lema1_eq_R1_insert_hey}
E_{\textup{p}}(v)\ge c\int_{\mathcal{O}}\left(|v(x)|^4+1\right)dx\ge c,\qquad \forall v\in L^{4}(\mathcal{O}).
\end{align}% {\r you keep $1$ here, the goal is ?}
We remark that the non-negativity of $E_{\textup{p}}$
ensures that the SSAV $r$ is well-defined, whereas the coercivity in \eqref{lema1_eq_R1_insert_hey} is imposed for technical reasons and will be used in the stability and convergence analysis of the proposed scheme. For general nonlinearities not satisfying \eqref{lema1_eq_R1_insert_hey}, one may instead employ the convex splitting technique (see, e.g., \cite{ES93}) to solve \eqref{model}.

In view of \eqref{def:ssav} and the It\^o formula, the original equation \eqref{model} can be recast into the following  SSAV reformulation 
\begin{subequations}\label{SAVmodel}
\begin{align}
	d\phi(t) &= -A^2\phi(t) dt  +\frac{r(t)}{\sqrt{E_{\textup{p}}({\phi(t)})}}A {F'}(\phi(t))  dt+g(\phi(t))d W(t), \label{SAVmodel_1}  
	\\   
	\label{eq:r(tn+1)}
	d r(t)  &= \frac{1}{2\sqrt{E_{\textup{p}}({\phi(t)})}} \left\langle{F'}(\phi(t)), d\phi(t) \right\rangle \\
	&\quad - \frac{1}{8E_{\textup{p}}(\phi(t))^{3/2}}
	\sum_{k=1}^{\infty} \left\langle {F'}(\phi(t)),g(\phi(t))Q^{\frac 12}e_k\right\rangle ^2
	dt \notag\\
	&\quad + \frac{1}{4\sqrt{E_{\textup{p}}(\phi(t)) }} \sum_{k=1}^{\infty}\left\langle F''(\phi(t))g(\phi(t))Q^{\frac 12}e_k,g(\phi(t))Q^{\frac 12}e_k\,\right\rangle\  dt\notag
\end{align}
\end{subequations}
for $t\in(0,T]$ subject to the initial values $\phi(0)=\phi^0$ and $r(0)=\sqrt{E_{\textup{p}}(\phi^0)}$.

To discretize  \eqref{SAVmodel} in time, we partition $[0,T]$ into $N\,(N\in\mathbb{N}^+)$ uniform subintervals with time step size $\tau=T/N$, and denote the time grid by $\{t_n:=n\tau\}_{n=0}^N$. 
Over each subinterval $(t_n,t_{n+1}]$, we freeze the coefficient
$\frac{r(t)}{\sqrt{E_{\textup{p}}(\phi(t))}}\,F'(\phi(t))$ in \eqref{SAVmodel_1} through a suitable approximation, denoted by $\tilde{f}^n$ (to be specified later), and thereby obtain the following approximation of $\phi(t)$:
\begin{align}%\label{SAVmodel}
d\tilde{\phi}(t) &= -A^2\tilde{\phi}(t) dt  +A \tilde{f}^n  dt+g(\tilde{\phi}(t))d W(t), \quad t\in (t_n,t_{n+1}].\label{SAVmodel_1-new}  
\end{align}
Then applying the exponential Euler method to \eqref{SAVmodel_1-new}, we obtain the numerical scheme \eqref{schme_phi-intro} 
%\begin{align}\label{schme_phi}
%	X^{n+1} = S(\tau)X^n
%	+ \left(I - S(\tau)\right)A^{-1}\tilde{f}^n + S(\tau)g(X^n)\delta W^n
%\end{align}
for any $ n\in\{0,1,\cdots,N-1\}$,
with $X^0 = \phi^0$. %where \( \delta W^n := W(t_{n+1}) - W(t_n) \) represents the discrete-time increment of the Wiener process $W$.
%To obtain the numerical scheme {\eqref{schme_phi-intro}},
It remains to define a consistent approximation $\tilde{f}^n$ for the quantity $\frac{r(t)}{\sqrt{E_{\textup{p}}(\phi(t))}}\,F'(\phi(t))$ with $t\in[t_{n},t_{n+1}]$.

\begin{remark}\label{rem-rt}
For the deterministic Cahn--Hilliard equation (i.e., $g\equiv0$), a suitable choice of $\tilde{f}^n$ is 
\begin{align}\label{fnd}
	\tilde{f}_{\textrm{d}}^{\,n}
	\;=\;
	\frac{r^{n+1}_\textrm{d}}{\sqrt{E_{\textup{p}}(X^n)}}\,F'(X^n),
\end{align}
where $r^{n+1}_\textrm{d}$ is the numerical solution of $r(t_{n+1})$ generated iteratively by 
\begin{align}\label{scheme_r_sto-d}
	r^{n+1}_\textrm{d} &=r^n_\textrm{d} + \frac{1}{2\sqrt{E_{\textup{p}}(X^n) }}\langle F'(X^n),
	X^{n+1} - X^n\rangle
\end{align}
for any $ n\in\{0,1,\cdots,N-1\}$,
with the initial value $r^0_\textrm{d}=\sqrt{E_{\textup{p}}(\phi^0)}$ (see, e.g., \cite{shen2018scalar}). However, this construction does not directly extend to the stochastic case. % {\r this remark is very important for the readers: it is better to explain that direct use the deterministic SAV will leads to the convergence to another spde or divergent} 
To illustrate this issue, we apply the Taylor expansion to $\sqrt{E_{\textup{p}}(\phi(t_{n+1}))}$, discarding terms of order three and higher, which formally gives (see also \eqref{SSAVr_TE-intro})
\begin{align}\label{SSAVr_TE}
	r(t_{n+1})  &\approx r(t_{n})+\frac{1}{2\sqrt{E_{\textup{p}}({\phi(t_n)})}}\left\langle F'(\phi(t_n)), \phi(t_{n+1})-\phi(t_n) \right\rangle +R_n,
\end{align}
where $R_n$ is the quadratic term given by 
\begin{align*}
	R_n&:= - \frac{1}{8E_{\textup{p}}(\phi(t_n))^{3/2}}
	\left\langle F'(\phi(t_n)),\phi(t_{n+1})-\phi(t_n) \,\right\rangle^2
	\notag\\
	&\quad + \frac{1}{4\sqrt{E_{\textup{p}}(\phi(t_n)) }} \left\langle F''(\phi(t_n))(\phi(t_{n+1})-\phi(t_n)),\phi(t_{n+1})-\phi(t_n) \,\right\rangle.\notag
\end{align*}
%{\r add roughly speaking  here and in the introduction?}
A comparison of \eqref{SSAVr_TE} and \eqref{scheme_r_sto-d} shows that  the difference $r(t_{N})-r^{N}_\textrm{d}$ contains a telescoping sum 
$\sum_{n=0}^{N-1} R_n$.	{Roughly speaking,}
in the stochastic case, the solution increment 
\(\phi(t_{n+1}) - \phi(t_n)\) is of order \(\mathcal{O}(\tau^{1/2})\). 
{Hence, the telescoping sum
	\(
	\sum_{n=0}^{N-1} R_n
	\)
	does not vanish as \( \tau \to 0 \), which prevents the convergence of
	\( r_{\mathrm{d}}^{N} \) to \( r(T) \) as \( \tau \to 0 \).
	In other words, directly adopting the standard SAV update \eqref {scheme_r_sto-d} together with \eqref{schme_phi-intro} with the choice $\tilde{f}^n=\tilde{f}_{\textrm{d}}^{\,n}$ leads to a numerical scheme that either converges to a different stochastic system or diverges, rather than converging to the target model \eqref{SAVmodel}.}
In fact, the accumulation of the quadratic term $R_n$ over time gives rise to the It\^o correction term in \eqref{eq:r(tn+1)}, and thus
the quadratic terms in \eqref{SSAVr_TE} should be retained in the numerical discretization of 
\(r(t)\), to ensure consistency with the It\^o formula and to accurately capture the 
energy evolution law  \eqref{dissi_conti} of \eqref{model}  {(see Figure \ref{Energy2} for the influence of the modified SSAV)}. %{\r I think that we have numerical test to illustrate this, why not cite them here?}

\end{remark}
Since the quadratic terms in \eqref{SSAVr_TE} involve the second power of $\phi(t_{n+1}) - \phi(t_n)$, a straightforward Euler-type discretization of \eqref{SSAVr_TE} destroys the linear structure of the proposed scheme and fail to produce an explicit or semi-implicit numerical discretization.
Following the strategy in \cite{metzger2023convergent},
% {\r there is one thing missing here and in the introduction: difference and comparison between their method and ours. the reviewer may wonder what is advantage of our method compared to \cite{metzger2023convergent} and his related work } 
we replace one factor $\phi(t_{n+1}) - \phi(t_n)$ in the It\^o correction terms  by $g(X^n)\,\delta W^n$, and update the numerical SSAV $\{r^n\}_{n=0}^{N}$ via \eqref{scheme_r_sto-intro}
%\begin{align}\label{scheme_r_sto}
%	r^{n+1} &=r^n + \frac{1}{2\sqrt{E_{\textup{p}}(X^n) }} \left\langle F'(X^n),X^{n+1} - X^n\,\right\rangle  \\\notag
%	&\quad- \frac{1}{8(E_{\textup{p}}(X^n) )^{3/2}}
%	\left\langle F'(X^n),g(X^n)\delta W^n\,\right\rangle 
%	\left\langle F'(X^n),X^{n+1} - X^n\,\right\rangle \\\notag
%	&\quad + \frac{1}{4\sqrt{E_{\textup{p}}(X^n) }} \left\langle F''(X^n)(X^{n+1} - X^n),g(X^n)\delta W^n\right\rangle
%\end{align}
for any $n\in\{0,1,\cdots,N-1\}$,
with the initial value $r^0=\sqrt{E_{\textup{p}}(\phi^0)}$. 
Instead of using \eqref{fnd},  in the stochastic case, we adopt
\begin{align}\label{fn}
\tilde{f}^{n}=\frac{r^{n+1}}{\sqrt{E_{\textup{p}}(X^n)}}F'(X^n)+\chi^n,
\end{align}
where the modified term 
\begin{align}\label{chi}
\chi^n
&:= -\,\frac{r^{n+1}}{4E_{\textup{p}}(X^n)^{3/2}}
F'(X^n)\!\left\langle F'(X^n),g(X^n)\delta W^n\right\rangle\\\notag
&\quad+
\frac{r^{n+1}}{2\sqrt{E_{\textup{p}}(X^n)}}
F''(X^n)(g(X^n)\delta W^n)
\end{align}% {\r inner product is missing?}
is added to compensate the It\^o correction term on the right hand side of \eqref{scheme_r_sto-intro} (see \eqref{lemma1_scheme_r_sto_before} for more details). The temporal semi-discretization \eqref{schme_phi-intro}, together with \eqref{scheme_r_sto-intro} and \eqref{fn}, constitutes the proposed scheme, which is referred to as the exponential Euler SSAV scheme. In this work, we focus on the time discretization \eqref{schme_phi-intro} for \eqref{model}; the extension to a fully discrete scheme by incorporating suitable spatial discretizations will be addressed in future work.

\begin{remark}\label{modify_reason}
The update of $r^n$ in \eqref{scheme_r_sto-intro} together with the definition of $\tilde f^n$ in \eqref{fn} ensures that
\begin{align}\label{lemma1_scheme_r_sto_before}
	{ 2r^{n+1} ( r^{n+1}-r^{n})}
	&=\langle\tilde{f}^{n}, X^{n+1}-X^{n}\rangle,\qquad n=0,1,\cdots,N-1.
\end{align}
Although applying the Euler method to \eqref{eq:r(tn+1)} or replacing the solution increment $X^{n+1}-X^n$ by $g(X^n)\,\delta W^n$ in the It\^o correction terms of \eqref{scheme_r_sto-intro} also yields an explicit update for $\{r^n\}_{n=0}^N$, neither of these two approaches satisfies \eqref{lemma1_scheme_r_sto_before}. As in the standard SAV scheme for deterministic gradient flows (see \cite{shen2018scalar}), the relation \eqref{lemma1_scheme_r_sto_before} is crucial for establishing the unconditional stability and the averaged energy evolution law of the exponential Euler SSAV scheme \eqref{schme_phi-intro} (see the proofs of Lemma~\ref{lemma_stability} and Theorem~\ref{Sec:energy:pro} for more details). 
%	{\r since this remark is also important, we need also write one-two sentence to highlight in the introduction}
%In addition, we note that when the noise vanishes, the exponential Euler SSAV scheme \eqref{schme_phi-intro} naturally reduces to the standard SAV method for the deterministic Cahn--Hilliard equation.
\end{remark}

To end this section, we point out that the exponential Euler SSAV scheme  \eqref{schme_phi-intro} can be explicitly solved as follows. First, for each \(n= 0,1,\cdots,N-1\), we denote
\begin{align*}%\label{eq:bn}
b^{n}: = \frac{F'(X^{n})}{\sqrt{E_{\textup{p}}(X^{n}) }}
- \frac{F'(X^{n}) \langle F'(X^{n}), g(X^{n}) \delta W^n\rangle}{4 E_{\textup{p}}(X^{n})^{3/2}}
+ \frac{F''(X^{n}) g(X^{n}) \delta W^n}{2\sqrt{E_{\textup{p}}(X^{n}) }}.
\end{align*}
Then, \eqref{scheme_r_sto-intro} and \eqref{fn} can be rewritten as 
\(r^{n+1} =r^n + \frac{1}{2}\langle b^n,X^{n+1} - X^n\rangle\)
and
\(\tilde{f}^n = r^{n+1}b^n = r^nb^n + \frac{1}{2}b^n\langle b^n,X^{n+1} - X^n\rangle.\)
Substituting this expression of $\tilde{f}^n$ into \eqref{schme_phi-intro}, we obtain 
\begin{align}\label{scheme-rewrite}
X^{n+1} - \frac{1}{2}(I - S(\tau)) A^{-1} b^{n} \langle b^{n}, X^{n+1}\rangle = w^n,
\end{align}
where 
$
w^{n} = S(\tau) X^{n} + (I - S(\tau)) A^{-1} b^{n} r^{n} - \frac{1}{2}(I - S(\tau)) A^{-1} b^{n} \langle b^{n}, X^{n}\rangle + S(\tau) g(X^{n}) \delta W^{n}.
$
Taking the inner product on both sides of \eqref{scheme-rewrite} with \( b^{n} \), it follows that
$
\langle X^{n+1}, b^{n}\rangle + \gamma^{n} \langle b^{n}, X^{n+1}\rangle = \langle w^{n}, b^{n}\rangle,
$
where 
$$
\gamma^{n}
=-\frac{1}{2}\left\langle (I - S(\tau)) A^{-1} b^{n}, b^{n} \right\rangle
=\frac12\|(I-S(\tau))^{1/2}(-A)^{-1/2}b^{n}\|^{2}\ge 0.
$$
This yields
$\langle X^{n+1}, b^{n}\rangle = (1 + \gamma^{n})^{-1}\langle b^{n}, w^{n}\rangle.$
Hence, the updated value \( X^{n+1} \) can be computed from \eqref{scheme-rewrite} in a fully explicit manner.

\section{Regularity estimates}\label{sec3:regularity}
In this section, we establish several regularity estimates for the exponential Euler SSAV scheme \eqref{schme_phi-intro}. {These estimations are essential for the proofs of the main results Theorems \ref{Thm3_main} and \ref{Sec:energy:pro}.}

\subsection{Spatial regularity estimate}\label{sec3:regularity_space}
We begin by deriving the $H^1(\mathcal{O})$-spatial regularity estimate of the numerical solution in the following lemma, which implies that the exponential Euler SSAV scheme \eqref{schme_phi-intro} is unconditionally stable.

% {\r we can put some parts in the appendix}

\begin{lemma}\label{lemma_stability-pre}
Let Assumptions \ref{assum2} and \ref{assum1} hold, and let $ \phi^0\in \dot{H}^1(\mathcal{O})$. Then for any $p\geq 1$, there exists a positive constant $C:=C(p)$ such that
%{\r a quick question: why not put $m=N-1$?}	
\begin{align}\label{Lemma1_step1}
	\mathbb{E}\left[\|\nabla  X^{n}\|^{2p}\right]+\mathbb{E}\left[| r^{n}|^{2p}\right] \leq C(p),\qquad\forall~n\in\{0,1,\cdots,N\}.
\end{align}
\end{lemma}
\begin{proof}
Let \( n \in \{0,1,\cdots,N-1\} \). Applying the integration by parts formula and utilizing the commutativity of $S(\tau)$ and $A$,  it follows {from \eqref{schme_phi-intro}} that 
\begin{align}\label{dphi2}
	&\frac{1}{2}\left[\|\nabla X^{n+1}\|^{2}-\|\nabla X^{n}\|^{2}\right]\\\notag
	&=\frac{1}{2}\left\langle(S^{2}(\tau)-I) X^{n},(-A) X^{n}\right\rangle+\frac{1}{2}\big\langle\left(I-S(\tau)\right)^{2}A^{-1}\tilde{f}^{n},-\tilde{f}^{n}\big\rangle\\
	&\quad+\frac{1}{2}\langle S^{2}(\tau)g( X^{n})\delta  W^{n},(-A)g( X^{n})\delta  W^{n}\rangle+\langle S(\tau)(I-S(\tau)) X^{n},-\tilde{f}^{n}\rangle\notag\\
	&\quad+\langle S^{2}(\tau) X^{n},(-A)g( X^{n})\delta  W^{n}\rangle+\langle S(\tau)(I-S(\tau))A^{-1}\tilde{f}^{n},(-A)g( X^{n})\delta  W^{n}\rangle.\notag
\end{align}
From \eqref{lemma1_scheme_r_sto_before} and utilizing \eqref{schme_phi-intro}, we arrive at 
\begin{align}\label{lemma1_scheme_r_sto}
	{ 2r^{n+1} ( r^{n+1}-r^{n})}
	%&=\langle\tilde{f}^{n}, X^{n+1}-X^{n}\rangle\\\notag
	&=\langle(I-S(\tau)) X^{n},-\tilde{f}^{n}\rangle+\langle(S(\tau)-I)A^{-1}\tilde{f}^{n},
	-\tilde{f}^{n}\rangle\\\notag
	&\quad+\langle -S(\tau)A^{-1}\tilde{f}^{n},(-A)g( X^{n})\delta  W^{n}\rangle.\notag
\end{align}
Combining  \eqref{dphi2} and \eqref{lemma1_scheme_r_sto} together, and using
the elementary identity \(2a(a-b)=a^{2}-b^{2}+(a-b)^{2}\) for $a,b\in \mathbb{R}$,
we have
\begin{align}\label{dis_energy_diss}
	&\frac{1}{2}\left[\left\|\nabla   X^{n+1}\right\|^{2}-\left\|\nabla X^{n}\right\|^{2}\right]
	+|r^{n+1}|^{2}-|r^{n}|^{2}+|r^{n+1}-r^n|^{2}\\
	&\quad+ \frac{1}{2}\|(I-S^2(\tau))^{\frac{1}{2}}(-A)^{\frac 12}\tilde{\mu}^n\|^2\notag\\
	&=\frac{1}{2}\left\langle(S^{2}(\tau)-I)( X^{n}-A^{-1}\tilde{f}^{n}+g( X^{n})\delta  W^{n}),(-A)( X^{n}-A^{-1}\tilde{f}^{n}+g( X^{n})\delta  W^{n})\right\rangle\notag\\\nonumber 
	&\quad +\left\langle X^{n}-A^{-1}\tilde{f}^{n},(-A)g( X^{n})\delta  W^{n}\right\rangle +\frac{1}{2}\left\langle g( X^{n})\delta  W^{n},(-A)g( X^{n})\delta  
	W^{n}\right\rangle\\\nonumber
	&\quad +\frac{1}{2}\|(I-S^2(\tau))^{\frac{1}{2}}(-A)^{\frac 12}\tilde{\mu}^n\|^2\notag\\
	&=\left\langle -AX^{n}+\tilde{f}^{n},  g( X^{n})\delta  W^{n}\right\rangle
	+\frac{1}{2}\|(-A)^{\frac12}\left(g( X^{n})\delta  W^{n}\right)\|^2,\notag
\end{align}
where the definition of $\tilde{\mu}^n$ was applied in the last step.
It follows from the formulation \eqref{fn} of $\tilde{f}^n$, as well as Young's inequality, that
\begin{align*}
	&\langle \tilde{f}^{n},g( X^{n})\delta  W^{n}\rangle\\
	&=   \frac{r^n}{\sqrt{E_{\textup{p}}( X^{n})}}\left\langle F'( X^n),g( X^n)\delta  W^n\right\rangle\notag+\frac{r^{n+1}-r^n}{\sqrt{E_{\textup{p}}( X^{n})}}\left\langle F'( X^n),g( X^n)\delta  W^n\right\rangle\notag\\
	&\quad +\frac{r^{n+1}-r^n}{2\sqrt{E_{\textup{p}}( X^{n})}}\left\langle F''( X^n)g( X^n)\delta  W^n,g( X^n)\delta  W^n\right\rangle\notag\\
	&\quad- \frac{r^{n+1}-r^n}{4E_{\textup{p}}( X^{n})^{3/2}}\left\langle F'( X^n),g( X^n)\delta  W^n\right\rangle^2\notag\\
	&\quad+\frac{r^{n}}{2\sqrt{E_{\textup{p}}( X^{n})}}\left\langle F''( X^n)g( X^n)\delta  W^n,g( X^n)\delta  W^n\right\rangle\notag\\
	&\quad- \frac{r^{n}}{4E_{\textup{p}}( X^{n})^{3/2}}\left\langle F'( X^n),g( X^n)\delta  W^n\right\rangle^2\notag\\
	&\le \frac{r^n}{\sqrt{E_{\textup{p}}( X^{n})}}\left\langle F'( X^n),g( X^n)\delta  W^n\right\rangle\notag+\frac{3}{4}|r^{n+1}-r^n|^2+\frac{1}{E_{\textup{p}}( X^{n})}\left\langle F'( X^n),g( X^n)\delta  W^n\right\rangle^2\\
	&\quad+ C \tau |r^{n}|^2 +C(1+\tau^{-1})\frac{1}{E_{\textup{p}}( X^{n})}\left\langle F''( X^n)g( X^n)\delta  W^n,g( X^n)\delta  W^n\right\rangle^2\\
	&\quad+C(1+\tau^{-1})\frac{1}{E_{\textup{p}}( X^{n})^3}\left\langle F'( X^n),g( X^n)\delta  W^n\right\rangle^4.
\end{align*}
Combining this with \eqref{dis_energy_diss}, and then summing the resulting inequality from \( n = 0 \) to \( m \), it holds that  
\begin{align}\label{lema1_eq6_before_before}
	&\|\nabla  X^{m+1}\|^{2}
	+|r^{m+1}|^{2}+\frac14\sum_{n=0}^m|r^{n+1}-r^n|^{2}
	+\frac12\sum_{n=0}^{m}\|(I-S^2(\tau))^{\frac{1}{2}}(-A)^{-\frac 12}\tilde{\mu}^n\|^2\\
	&\leq \|\nabla \phi^{0}\|^{2}+|r^{0}|^{2}
	+C \tau \sum_{n=0}^m|r^{n}|^{2}
	+\sum_{n=0}^m\frac{1}{E_{\textup{p}}(X^{n})}\left\langle F'( X^n),g( X^n)\delta  W^n\right\rangle^2\notag\\
	&\quad+C(1+\tau^{-1})\sum_{n=0}^m\frac{1}{E_{\textup{p}}(X^{n})}\left\langle F''( X^n)g( X^n)\delta  W^n,g( X^n)\delta  W^n\right\rangle^2\notag\\
	&\quad+C(1+\tau^{-1})\sum_{n=0}^m\frac{1}{E_{\textup{p}}(X^{n})^3}\left\langle F'( X^n),g( X^n)\delta  W^n\right\rangle^4
	+\frac{1}{2}\sum_{n=0}^m\| (-A)^{\frac12}(g( X^{n})\delta  W^{n})\|^2\notag\\
	&\quad+\sum_{n=0}^m\Big\langle-A X^n+\frac{r^n}{\sqrt{E_{\textup{p}}(X^{n})}}F'( X^n),g( X^n)\delta  W^n\Big\rangle.\notag
\end{align}
Taking the \( p \)th power on both sides of \eqref{lema1_eq6_before_before}, 
we obtain by H\"older's inequality that
\begin{align}\label{lema1_eq6_before}
	&\|\nabla  X^{m+1}\|^{2p}
	+|r^{m+1}|^{2p}+\bigg(\frac14\sum_{n=0}^m|r^{n+1}-r^n|^{2}\bigg)^p\\
	&\quad{+\bigg(\frac12\sum_{n=0}^{m}\|(I-S^2(\tau))^{\frac{1}{2}}(-A)^{-\frac 12}\tilde{\mu}^n\| ^{2}\bigg)^p}\notag\\
	&\leq C\|\nabla \phi^{0}\|^{2p}+C|r^{0}|^{2p}\notag\\
	&\quad+C \tau \sum_{n=0}^m|r^{n}|^{2p}+C\bigg(\sum_{n=0}^m\frac{1}{E_{\textup{p}}( X^n)}\left|\left\langle F'( X^n),g( X^n)\delta  W^n\right\rangle\right|^2\bigg)^p\notag\\
	&\quad+C(1+\tau^{-1})^p\bigg(\sum_{n=0}^m\frac{1}{E_{\textup{p}}( X^n)}\left\langle F''( X^n)g( X^n)\delta  W^n,g( X^n)\delta  W^n\right\rangle^2\bigg)^p\notag\\
	&\quad+C(1+\tau^{-1})^p\bigg(\sum_{n=0}^m\frac{1}{E_{\textup{p}}(X^{n})^3}\left\langle F'( X^n),g( X^n)\delta  W^n\right\rangle^4\bigg)^p\notag\\
	&\quad+C\bigg(\sum_{n=0}^m\frac{1}{2}\| (-A)^{\frac12}(g( X^{n})\delta  W^{n})\|^2\bigg)^p\notag\\
	&\quad+C\bigg|\sum_{n=0}^m\Big\langle-A X^n+\frac{r^n}{\sqrt{E_{\textup{p}}(X^{n})}}F'( X^n),g( X^n)\delta  W^n\Big\rangle\bigg|^p\notag\\
	&=:C\|\nabla \phi^{0}\|^{2p}+C|r^{0}|^{2p}
	+C \tau \sum_{n=0}^m|r^{n}|^{2p}+R_1^m+R_2^m+R_3^m+R_4^m+R_5^m.\notag
\end{align}
We next estimate the expectations of the terms \( R_1^m, \cdots, R_5^m \), separately.

\textbf{Estimate of \( R_1^m \).} 
According to Assumption~\ref{assum1} and the Sobolev embedding \( L^\infty(\mathcal{O}) \hookrightarrow L^q(\mathcal{O}) \) with $1\le q\le \infty$, it holds that for any $1\le q\le \infty$,
\begin{align}\label{lema1_eq_R1_insert_Q}
	\sum_{k=1}^{\infty}\|g( X^n)Q^{\frac{1}{2}}e_{k}\|_{L^q}^{2}
	%&\le C\sum_{k=1}^{\infty}\|g( X^n)Q^{\frac{1}{2}}e_{k}\|_{L^{\infty}}^{2}\\\notag
	&\le C\|g( X^n)\|^2_{L^{\infty}}\sum_{k=1}^{\infty}\|Q^{\frac{1}{2}}e_{k}\|_{L^{\infty}}^{2}
	\le C.
\end{align}
By Young's inequality, \eqref{lema1_eq_R1_insert_hey}, and Assumption~\ref{assum2}, we derive that for any $q>0$,
\begin{align}\label{lema1_eq_R3_insert} 
	\left\|F'( X^n)\right\|_{L^{4/3}}^{4q}
	\leq C(1+\|X^n\|_{L^4}^{4})^{3q}\le C{E_{\textup{p}}(X^{n})^{3q}}.
\end{align} %{\r comment: I see the reason how $1$ comes from}
In particular, applying \eqref{lema1_eq_R3_insert} gives
\begin{align}\label{lema1_eq_R1_insert}
	\mathbb{E}\bigg[\frac{\left\|F'( X^n)\right\|^{2p}_{L^{4/3}}}{E_{\textup{p}}(X^{n})^p}\bigg]
	&\leq C\mathbb{E}\left[\left\|F'( X^n)\right\|^{\frac23p}_{L^{4/3}}\right]\\\notag
	&\leq  C\mathbb{E}\left[1+\|X^n\|_{L^4}^{2p}\right]\leq C + C\mathbb{E}\left[\|\nabla X^n\|^{2p}\right],
\end{align}
where the last inequality follows from the Sobolev embedding \( H^1(\mathcal{O}) \hookrightarrow L^4(\mathcal{O}) \) and the Dirichlet boundary conditions for the numerical solution.
By H\"older's inequality, the martingale property of stochastic integrals, the Burkholder--Davis--Gundy (BDG) inequality \cite[Chap.~2]{KR14}, \eqref{lema1_eq_R1_insert}, and \eqref{lema1_eq_R1_insert_Q}, we obtain 
\begin{align*}%\label{lema1_eq_R1_final}
	\mathbb{E}[R_1^m]
	& \leq C (m+1)^{p-1}\sum_{n=0}^m \mathbb{E}\bigg[\bigg|\int_{t_n}^{t_{n+1}}\bigg\langle\frac{1}{\sqrt{E_{\textup{p}}(X^{n})}}F'( X^n),g( X^n)d W(s) \bigg\rangle\bigg|^{2p}\bigg]\\
	&\leq C (m+1)^{p-1}\sum_{n=0}^m \tau^p\mathbb{E}\bigg[\frac{1}{E_{\textup{p}}(X^{n})^p}\bigg(\sum_{k=1}^{\infty}\langle F'( X^n),g( X^n)Q^{\frac{1}{2}}e_{k} \rangle^2\bigg)^{p}\bigg]\\
	&\leq  C\sum_{n=0}^m \tau\mathbb{E}\bigg[\frac{\left\|F'( X^n)\right\|^{2p}_{L^{4/3}}}{E_{\textup{p}}(X^{n})^p}\bigg(\sum_{k=1}^{\infty}\|g( X^n)Q^{\frac{1}{2}}e_{k}\|_{L^4}^{2}\bigg)^p\bigg]\\
	&\leq C + C \sum_{n=0}^m \tau \mathbb{E} \left[\|\nabla X^n\|^{2p} \right].
\end{align*}

\textbf{Estimate of \( R_2^m \).} 
By Young's inequality, \eqref{lema1_eq_R1_insert_hey} and with the help of Assumption~\ref{assum2}, one gives
\begin{align}\label{lema1_eq_R2_insert}
	\||F''( X^n)|^{1/2}\|_{L^4}^{4p}=  \|F''( X^n)\|^{2p}
	\leq C(1+\|X^n\|_{L^4}^4)^p\le CE_{\textup{p}}(X^{n})^p.
\end{align}
Applying H\"older's inequality, the BDG inequality  
% {\r How do you apply BDG inequality? unclear for me}
together with \eqref{lema1_eq_R1_insert_Q} and \eqref{lema1_eq_R2_insert},
we infer that
\begin{align*}%\label{lema1_eq_R2}
	\mathbb{E}[R_2^m] 
	&\leq C\tau^{-2p+1}\sum_{n=0}^m\mathbb{E}\bigg[\bigg\|\int_{t_n}^{t_{n+1}}\frac{1}{E_{\textup{p}}(X^{n})^{1/4}}|F''( X^n)|^\frac{1}{2}g( X^n)d  W(s)\bigg\|^{4p}\bigg]\\
	&\leq C\tau\sum_{n=0}^m\mathbb{E}\bigg[\frac{1}{E_{\textup{p}}(X^{n})^p}\bigg(\sum_{k=1}^{\infty}\left\||F''( X^n)|^\frac{1}{2}g( X^n)Q^{\frac{1}{2}}e_k\right\|^{2}\bigg)^{2p}\bigg]\\
	&\leq C\tau\sum_{n=0}^m\mathbb{E}\bigg[\frac{\||F''( X^n)|^\frac{1}{2}\|^{4p}_{L^{4}}}{E_{\textup{p}}(X^{n})^p}\bigg(\sum_{k=1}^\infty \|g( X^n)Q^{\frac{1}{2}}e_k\|^{2}_{L^4}\bigg)^{2p}\bigg]\leq C.
\end{align*}

\textbf{Estimate of \( R_3^m \).}  
Invoking the BDG inequality, \eqref{lema1_eq_R1_insert_Q}, and \eqref{lema1_eq_R3_insert} results in
\begin{align}\label{lema1_eq_R3_insert1} 
	&\mathbb{E}\left[\frac{1}{E_{\textup{p}}(X^{n})^{3p}}\left\langle F'( X^n),g( X^n)\delta   W^n\right\rangle^{4p}\right]\\
	%&=\mathbb{E}\bigg[\bigg|\int_{t_n}^{t_{n+1}}\bigg\langle\frac{1}{E_{\textup{p}}(X^{n})^{3/4}} F'( X^n),g( X^n)d  W(s)\bigg\rangle\bigg|^{4p}\bigg]\notag\\
	&\leq C\mathbb{E}\bigg[\frac{1}{E_{\textup{p}}(X^{n})^{3p}}\bigg(\tau\sum_{k=1}^\infty\left\langle F'( X^n),g( X^n)Q^{\frac{1}{2}}e_{k}\right\rangle^2\bigg)^{2p}\bigg]\notag\\
	&\leq C\tau^{2p}\mathbb{E}\bigg[\frac{\left\|F'( X^n)\right\|^{4p}_{L^{4/3}}}{E_{\textup{p}}(X^{n})^{3p}}\bigg(\sum_{k=1}^{\infty}\|g( X^n)Q^{\frac{1}{2}}e_k\|^2_{L^4}\bigg)^{2p}\bigg]\leq C\tau^{2p}.\notag
\end{align}
H\"older's inequality and \eqref{lema1_eq_R3_insert1} yield
\begin{equation*}%\label{lema1_eq_R3}
	\mathbb{E}[R_3^m] 
	\leq C(1+\tau^{-1})^p (m+1)^{p-1}\sum_{n=0}^m\mathbb{E}\left[\frac{1}{E_{\textup{p}}(X^{n})^{3p}}\left\langle F'( X^n),g( X^n)\delta   W^n\right\rangle^{4p}\right]
	\leq C.
\end{equation*}

\textbf{Estimate of \( R_4^m \).} 
%
%\begin{align}\label{lema1_eq_R6-new}\notag
%    &\sum_{k=1}^{\infty}\left\|(-A)^{\frac12}(g( X^n)Q^{\frac{1}{2}}e_{k})\right\|^2= \sum_{k=1}^{\infty}\left\|\nabla (g( X^n)Q^{\frac{1}{2}}e_{k})\right\|^2\\
%  & \leq C\sum_{k=1}^{\infty}\left\|\sigma'( X^n)\right\|^2_{L^{\infty}}\left\|\nabla  X^n\right\|^2\|Q^{\frac{1}{2}}e_k\|^2_{L^{\infty}}+C\sum_{k=1}^{\infty}\left\|\sigma( X^n)\right\|^2_{L^{\infty}}\|\nabla (Q^{\frac{1}{2}}e_{k})\|^2\leq C + C\left\|\nabla  X^n\right\|^{2}.
%\end{align}
By \eqref{lema1_eq_R6-new},  H\"older's inequality, and the BDG inequality, 
\begin{align}\label{lema1_eq_R6}
	\mathbb{E}[R_4^m] &
	%\le C\mathbb{E}\left[\left(\sum_{n=0}^m\left\|(-A)^{\frac12}\int_{t_n}^{t_{n+1}} g( X^{n})d W(s)\right\|^2\right)^p\right]
	\le C(m+1)^{p-1}\sum_{n=0}^m\mathbb{E}\left[\Big\|(-A)^{\frac12}\int_{t_n}^{t_{n+1}} g( X^{n})d W(s)\Big\|^{2p}\right]\\
	&\leq C(m+1)^{p-1}\sum_{n=0}^m\mathbb{E}\bigg[\bigg(\int_{t_n}^{t_{n+1}}\|g(X^n)\|^2_{\mathcal{L}_2(Q^{\frac12}H,\dot{H}^1(\mathcal{O}))}ds\bigg)^p\bigg]\notag\\
	&\leq C + C\tau\sum_{n=0}^m\mathbb{E}
	\left[\left\|\nabla  X^n\right\|^{2p}\right].\notag
\end{align}

\textbf{Estimate of \( R_5^m \).}  In view of H\"older's inequality and the BDG inequality,
\begin{align}\label{lema1_eq_R7}
	\mathbb{E}[R_5^m]&\leq C\mathbb{E}\bigg[\bigg|\sum_{n=0}^m\bigg\langle-A X^n+\frac{r^n}{\sqrt{E_{\textup{p}}(X^{n})}}F'( X^n),g( X^n)\delta  W^n\bigg\rangle\bigg|^p\bigg]\\
	&\le C\mathbb{E}\left[\bigg(\sum_{n=0}^m\tau\sum_{k=1}^{\infty}\langle -AX^n,g( X^n)Q^{\frac{1}{2}}e_{k}\rangle^2\bigg)^{\frac{p}{2}}\right]\notag\\
	&\quad+C\mathbb{E}\Bigg[\bigg(\sum_{n=0}^m\tau\sum_{k=1}^{\infty}\bigg\langle\frac{r^n}{\sqrt{E_{\textup{p}}(X^{n})}}F'( X^n),g( X^n)Q^{\frac{1}{2}}e_{k}\bigg\rangle^2\bigg)^{\frac{p}{2}}\Bigg]\notag\\
	&=:R_{51}^m+R_{52}^m.\notag
\end{align}
By H\"older's inequality, Young's inequality,
the integration by parts formula, and \eqref{lema1_eq_R6-new}, we derive that 
\begin{align}\label{lema1_eq_R71}
	R_{51}^m
	% &\leq C\mathbb{E}\left[\bigg(\sum_{n=0}^m\tau\sum_{k=1}^{\infty}\langle -AX^n,g( X^n)Q^{\frac{1}{2}}e_{k}\rangle^2\bigg)^{\frac{p}{2}}\right]\notag\\
	&\leq C(m+1)^{\frac{p}{2}-1}\sum_{n=0}^m\tau^{\frac{p}{2}}\mathbb{E}\left[\bigg(\sum_{k=1}^{\infty}\langle\nabla X^n,\nabla(g( X^n)Q^{\frac{1}{2}}e_{k})\rangle^2\bigg)^{\frac{p}{2}}\right]\\
	&\leq C\tau\sum_{n=0}^{m}\mathbb{E}\left[\|\nabla X^n\|^{2p}\right]+C\tau\sum_{n=0}^{m}\mathbb{E}\bigg[\bigg(\sum_{k=1}^{\infty}\|\nabla(g( X^n)Q^{\frac{1}{2}}e_{k})\|^{2}\bigg)^p\bigg]\notag\\
	&\leq C+C\tau\sum_{n=0}^{m}\mathbb{E}\left[\|\nabla X^n\|^{2p}\right].\notag
\end{align}
In a similar manner, by using the BDG inequality,  
H\"older's inequality, Young's inequality, \eqref{lema1_eq_R1_insert_Q}
together with \eqref{lema1_eq_R1_insert}, we conclude that
\begin{align}\label{lema1_eq_R72}
	R_{52}^m
	%&\leq C\mathbb{E}\Bigg[\bigg(\sum_{n=0}^m\tau\sum_{k=1}^{\infty}\bigg\langle\frac{r^n}{\sqrt{E_{\textup{p}}(X^{n})}}F'( X^n),g( X^n)Q^{\frac{1}{2}}e_{k}\bigg\rangle^2\bigg)^{\frac{p}{2}}\Bigg]\notag\\
	&\leq C(m+1)^{\frac{p}{2}-1}\sum_{n=0}^m\tau^{\frac{p}{2}}
	\mathbb{E}\Bigg[\bigg(\sum_{k=1}^{\infty}\bigg\langle\frac{r^n}{\sqrt{E_{\textup{p}}(X^{n})}}F'( X^n),g( X^n)Q^{\frac{1}{2}}e_{k}\bigg\rangle^2\bigg)^{\frac{p}{2}}\Bigg]\\
	&\leq C\tau\sum_{n=0}^{m}\mathbb{E}\left[\left|r^{n}\right|^{2p}\right]
	+C\tau\sum_{n=0}^{m}\mathbb{E}\bigg[\bigg(\frac{\left\|F'( X^n)\right\|^2_{L^{4/3}}}{E_{\textup{p}}(X^{n})}\bigg)^{p}\bigg(\sum_{k=1}^\infty\|g( X^n)Q^{\frac{1}{2}}e_{k}\|^2_{L^4}\bigg)^p\bigg]\notag\\
	&\leq C +C\tau\sum_{n=0}^{m}\mathbb{E}\left[\left|r^{n}\right|^{2p}\right]. \notag
\end{align}
Inserting \eqref{lema1_eq_R71} and \eqref{lema1_eq_R72} into \eqref{lema1_eq_R7} leads to
\begin{equation*}%\label{lema1_eq_R7_final}
	\mathbb{E}[R_5^m] \leq C+C\tau\sum_{n=0}^{m}\mathbb{E}\left[\left|r^{n}\right|^{2p}\right]
	+C\tau\sum_{n=0}^{m}\mathbb{E}\left[\|\nabla X^n\|^{2p}\right].
\end{equation*}

Combining the above estimates for \( R_1^m \) through \( R_5^m \), we deduce by taking expectations on both sides of \eqref{lema1_eq6_before} that for any \( m \in \{0,1, \cdots, N-1\} \),
\begin{align*}%\label{lema1_eq7}
	\mathbb{E}\left[\left\|\nabla  X^{m+1}\right\|^{2p}\right]
	+\mathbb{E}\left[|r^{m+1}|^{2p}\right]
	&\leq C + C\mathbb{E}\left[\left\|\nabla \phi^{0}\right\|^{2p}\right] + C\mathbb{E}\left[|r^{0}|^{2p}\right] 
	\\
	&\quad + C\tau\sum_{n=0}^m\mathbb{E}\left[|r^n|^{2p}\right] + C\tau\sum_{n=0}^m\mathbb{E}\left[\|\nabla X^n\|^{2p}\right]. 
\end{align*}
Hence,
on account of \eqref{lema1_eq_R1_insert_hey}, we can conclude  \eqref{Lemma1_step1}
%Assumption \ref{assum2}, \( \phi^0 \in \dot{H}^1(\mathcal{O}) \), and the Sobolev embedding \( H^1(\mathcal{O}) \hookrightarrow L^4(\mathcal{O}) \), one has
%\begin{equation}\label{eq:r0}
%    r^0 = \sqrt{E_{\textup{p}}(\phi^0) } \le \sqrt{ C\int_{\mathcal{O}} (1+|\phi^0(x)|^4) \,dx}\leq C.
%\end{equation}
%
% Therefore, 
by applying the discrete Gronwall's inequality.
\end{proof}

%{\r better to  say $\sigma$ or $\sigma'$? 
% What is the derivative of the operator $g$?
%}
We  would like to point out that the proof of Lemma~\ref{lemma_stability-pre} crucially relies on the boundedness of $\sigma$ and its derivative $\sigma'$; see in particular the estimate of $R_4^m$ in \eqref{lema1_eq_R6}. Extending the analysis to more general diffusion models, including those with unbounded diffusion coefficients or gradient-type noise, would require substantially different techniques and is therefore left for future work.

%{\r here, some discussion is needed regarding the boundedness assumption on $g'$? why not unbounded? comment on gradient type noise }

Based on Lemma \ref{lemma_stability-pre}, we obtain the following regularity estimate for the numerical solution under a stronger topology; see Appendix \ref{App_Cor1} for the proof.
\begin{corollary}\label{lemma_stability}
Let Assumptions \ref{assum2} and \ref{assum1} hold, and let $ \phi^0\in \dot{H}^1(\mathcal{O})$. Then for any $p\geq 1$, there exists a positive constant $C:=C(p)$ such that
%{\r a quick question: why not put $m=N-1$?}	
\begin{align}\label{lemma_re}\notag
	\mathbb{E}\left[\sup_{0\leq n\leq N-1}\|\nabla  X^{n+1}\|^{2p}\right]&+\mathbb{E}\left[\sup_{0\leq n\leq N-1}| r^{n+1}|^{2p}\right] +\mathbb{E}\bigg[\bigg(\sum_{n=0}^{N-1}|r^{n+1}-r^n|^2\bigg)^p\bigg]\\
	&+\mathbb{E}\bigg[\bigg(\sum_{n=0}^{N-1}\|(I-S^2(\tau))^{\frac{1}{2}}(-A)^{-\frac 12}\tilde{\mu}^n\|^2\bigg)^p\bigg]\leq C,
\end{align}
where $\tilde{\mu}^n $ is the modified chemical  potential defined in \eqref{eq:tildemun}.
%   {\r comment: here the format $\tilde u$ is not same as $\tilde u$, right}
\end{corollary}

The sharp error analysis of the exponential Euler SSAV scheme \eqref{schme_phi-intro} involves the $H^\beta(\mathcal{O})$ ($\beta>1$) spatial regularity estimates for both the exact and numerical solutions.
To facilitate the higher  regularity estimate of the numerical solution,
we introduce its continuous time version $X=\{X(t)\}_{t\in[0,T]}$ defined by
\begin{equation}\label{Xn_contin}
d X(t)  =-A^2 X(t) d t+A \tilde{f}^{\kappa_N(t) / \tau} d t+S(t-\kappa_N(t)) g(X(\kappa_N(t))) d W(t)
\end{equation}
for any $t\in(0,T]$,
with the initial value $X(0) =\phi^0$. Hereafter, for $t\in[0,T]$, we define $\kappa_N(t):=\lfloor t/\tau\rfloor\tau$, {where $\lfloor\cdot\rfloor$ denotes the floor function}. % {\r have we defined the floor function} 
Notice that 
$X(t_n) = X^{n}$ for all $n\in \{0,1,\cdots,N\}$. 

To present the regularity analysis of numerical and exact solutions in a unified manner, we  
consider a stochastic process $\{Z(t)\}_{t\in[0,T]}$ defined via
\begin{align}\label{lemma6_eq3_R1_Zn}
dZ(t)=-A^2 Z(t) d t
+A\tilde{F}(t)dt
+\tilde{G}(t)dW(t),\qquad t\in[0,T],
\end{align}
subject to the initial condition $Z(0)=\phi^0$,
where $\tilde{F}:[0,T]\times\Omega\to H$ and
$\tilde{G}:[0,T]\times\Omega\to {\mathcal{L}_2^0}$ are
$\{\mathcal{F}_t\}_{t\in[0,T]}$-adapted stochastic processes. We focus on the following two cases.
\begin{enumerate}
\item[(a)] If $\tilde{F}(s)=F'(\phi(s))$ and $\tilde{G}(s)=g(\phi(s))$, then
$Z=\phi$ is the mild solution of the original problem \eqref{model}.
\item[(b)] If $\tilde{F}(s)=\tilde{f}^{\,\kappa_N(s)/\tau}$ and
$\tilde{G}(s)=S(s-\kappa_N(s))g(X^{\kappa_N(s)/\tau})$,
then $Z=X$ is the continuous version of the numerical solution $\{X^n\}_{n=0}^N$.
\end{enumerate}

The following lemma identifies sufficient conditions to guarantee the $H^\beta(\mathcal{O})$-spatial regularity of the mild solution $Z$ to \eqref{lemma6_eq3_R1_Zn}. 
\begin{lemma}\label{lemma_tool}
Let  $\tilde{F}\in L^p(0,T;L^p(\Omega;H))$ and $\tilde{G}\in L^p(0,T;L^p(\Omega;\mathcal{L}_2(Q^\frac12 H,H)))$ for any \( p \geq 1 \).
If $\phi^0\in D((-A)^{\frac{\beta}{2}})$ for some  \( \beta \in (0,2) \),  
then
for any \( p \geq 1 \), there exists a constant $C:=C(p,\beta)>0$ such that
\begin{equation*}%\label{lemma_tool_re1}
	\mathbb{E}\bigg[\sup_{t\in[0,T]} \|(-A)^{\frac{\beta}{2}} Z(t)\|^{p} \bigg] \leq C\bigg(\|(-A)^{\frac{\beta}{2}}\phi^0\|^p+\mathbb{E}\int_0^{T} \Big[\|\tilde{F}(s)\|^{p} +\|\tilde {G}(s)\|_{\mathcal{L}_2^0}^{p}\Big] ds\bigg).
\end{equation*}
\end{lemma}

Since the proof of Lemma \ref{lemma_tool} is standard,  we include it in Appendix \ref{App-lemma-tool} for completeness.
%As an application of Lemma \ref{lemma_tool},
% we also include the regularity estimate of the mild solution to \eqref{model} in Appendix \ref{Appendix}; see Proposition \ref{App_prop} for more details. {\r I do not understand eng. The regularity of the exact solu. is also included in Lemma 3.2? why need Proposition \ref{App_prop}?}
We are now in a position to derive a higher spatial regularity estimate for the numerical solution.

\begin{lemma}\label{coro1}
Let Assumptions \ref{assum2} and \ref{assum1} hold, and $\phi^0\in \dot{H}^\beta(\mathcal{O})$ for some $\beta\in[1,2)$. Then for any $p\ge 1$, there exists a constant $C:=C(p,\beta)>0$ such that 
\begin{align}\label{eq:AXtp}
	\mathbb{E}\left[\sup_{t\in[0,T]}\|(-A)^{\frac{\beta}{2}} X(t)\|^{2p}\right]\le C.
\end{align}% {\r where is $n$ in $X$?}
\end{lemma}
\begin{proof}
The proof is based on applying Lemma \ref{lemma_tool} with
\begin{align}\label{eq:FG}
	\tilde{F}(s)=\tilde{f}^{\,\kappa_N(s)/\tau}
	\quad\text{and}\quad
	\tilde{G}(s)=S\big(s-\kappa_N(s)\big)g(X(\kappa_N(s))).
\end{align}
By the contractivity of the semigroup $S(\cdot)$ and \eqref{lema1_eq_R1_insert_Q},
\begin{align}\label{eq:GL2}
	&\mathbb{E}\int_0^{T} \|\tilde {G}(s)\|_{\mathcal{L}_2^0}^{p}ds
	=\sum_{n=0}^{N-1}\int_{t_n}^{t_{n+1}}\mathbb{E}
	\left[\|S(s-t_n)g(X^n)\|_{\mathcal{L}_2^0}^{p}\right] ds\\\notag
	&\le\sum_{n=0}^{N-1}\int_{t_n}^{t_{n+1}} \|S(s-t_n)\|_{\mathcal{L}(H)}^{p}\mathbb{E}\left[\|g(X^n)\|_{{\mathcal{L}_2^0}}^{p}\right]ds
	\le C.
\end{align}
%		\begin{align}\label{eq:GL2}
	%			&\mathbb{E}\int_0^{T} \|\tilde {G}(s)\|_{\mathcal{L}_2^0}^{p}ds\\\notag
	%			&
	%			=\sum_{n=0}^{N-1}\int_{t_n}^{t_{n+1}}\mathbb{E}\bigg[\Big(\sum_{k=1}^{\infty}\|S(s-t_n)\,g(X^n)Q^{\frac 12}e_k\|^2\Big)^{\frac p2}\bigg] ds\\\notag
	%			&\le\sum_{n=0}^{N-1}\int_{t_n}^{t_{n+1}} \|S(s-t_n)\|_{\mathcal{L}(H)}^{p}\mathbb{E}\bigg[\Big(\sum_{k=1}^{\infty}\|g\left(X^n\right)Q^{\frac 12}e_k\|^2\Big)^{\frac p2}\bigg]ds
	%			\le C.
	%		\end{align}
Using H\"older's inequality, {Corollary \ref{lemma_stability}}, Assumption \ref{assum2},  the Sobolev embedding $H^1(\mathcal{O})\hookrightarrow L^6 (\mathcal{O})$, the fact that $X^n\in \dot{H}^1(\mathcal{O})$ yields that for any $p\ge1$,
\begin{equation}\label{eq:FXn}%\label{lemma2_J1_term1}
	\mathbb{E}\big[\|F'( X^{n})\|^{4p}\big]\leq C +C\mathbb{E}\big[\|\nabla X^n\|^{12p}\big]
	\leq C(p).
\end{equation}

Recall that by \eqref{lema1_eq_R1_insert_hey}, the functional $E_{\textup{p}}$ is bounded from below by a positive constant.
In view of \eqref{fn}, {Corollary \ref{lemma_stability}}, and Young's inequality, it holds that
\begin{align}\label{lemma6_eq3_R12_12-old}
	\mathbb{E}[\|\tilde{f}^{n}\|^p]
	&\le C\mathbb{E}[|r^{n+1}|^{2p}]+
	\mathbb{E}[\|F'(X^{n})\|^{2p}] + C\mathbb{E}[\|\chi^n\|^p]\le C+C\mathbb{E}[\|\chi^n\|^p]
\end{align}
for any $n=0,1,\cdots,N$.  
We claim that  for any $p\ge1$,
\begin{equation}\label{eq:chinp}
	\mathbb{E}\left[\|\chi^{n}\|^p\right]\le  C(p)\tau^{\frac{p}{2}},
\end{equation} 
whose proof is put in the Appendix \ref{App_4.24} for completeness.
Substituting \eqref{eq:chinp} into \eqref{lemma6_eq3_R12_12-old} ensures that for any $p\ge1$,
\begin{equation}\label{lemma6_eq3_R12_12}
	\mathbb{E}[\|\tilde{f}^{n}\|^p]
	\le C(p).
\end{equation} 
This, along with \eqref{eq:GL2}, validates the conditions of Lemma~\ref{lemma_tool} for the functions $\tilde F$ and $\tilde G$ defined in \eqref{eq:FG}. Finally, we can obtain \eqref{eq:AXtp} by applying Lemma~\ref{lemma_tool}.
\end{proof}

%  {\r we may remove the remark environment }\begin{remark}\label{rem:H2}
\begin{remark}\label{rem:H2-reg}
	Based on Lemma~\ref{coro1}, if in addition $\phi^0\in \dot{H}^2(\mathcal{O})$, one can further derive an $H^2(\mathcal{O})$-spatial regularity estimate for both the mild solution $\phi(t)$ and the continuous version $X(t)$ of the numerical solution. For instance, for the numerical solution, by noticing that $Z:=(-A)^{\frac12}X$ solves
	\begin{equation*}
		dZ(t) =-A^2 Z(t) d t+A (-A)^{\frac12}\tilde{f}^{\kappa_N(t) / \tau} d t+S(t-\kappa_N(t)) (-A)^{\frac12}g(X(\kappa_N(t))) d W(t),
	\end{equation*}
	{one can apply Lemma~\ref{lemma_tool} with the choices $\tilde{F}(s)=(-A)^{1/2}\tilde{f}^{\kappa_N(s)/\tau}$ and $\tilde{G}(s)=(-A)^{1/2} S(s-\kappa_N(s)) g(X(\kappa_N(s)))$ to obtain $(-A)^{\frac12}X\in L^p(\Omega;L^\infty(0,T;D((-A)^{\frac12}))$.}
\end{remark}
%{\r This single sentence is a misunderstanding. You may want to add one-two sentences to explain it.}

%	\end{remark}

The Sobolev embedding $H^{\beta}(\mathcal{O}) \hookrightarrow L^{\infty}(\mathcal{O})$ for $\beta >\tfrac{d}{2}$, the norm equivalence between $\|\cdot\|_{H^{\beta}}$ and $\|(-A)^{\frac{\beta}{2}}\cdot\|$ in the space $\dot{H}^{\beta}(\mathcal{O})$,  together with Lemma \ref{coro1} with $\beta\in(\frac{d}{2},2)\cap[1,2)$, %{\b $2>\beta>\frac d2$ in this paper, please modify them all?}
give that for any $p\ge1$,
\begin{equation}\label{eq:Linfty}
\mathbb{E}\left[\sup_{0\le n\le N}\|X^{n}\|_{L^\infty}^p\right]
%\le C \mathbb{E}\bigg[\sup_{t\in[0,T]}\|X(t)\|_{H^{\beta}}^{2p}\bigg] 
\le C \mathbb{E}\bigg[\sup_{t\in[0,T]}\|(-A)^{\frac{\beta}{2}}X(t)\|^{2p}\bigg]
\leq C(p,\beta).
\end{equation}
% {\b modify them $C(p,T,\beta)$?}

%{\r general comment: this section is too long, put some proof into appendix}

\subsection{Temporal regularity estimate}
In this subsection, we present temporal regularity results for the numerical solution of the exponential Euler SSAV scheme \eqref{schme_phi-intro}, including the temporal H\"older continuity in $L^2(\mathcal{O})$ and an a priori bound for the accumulated discrete  increment sum in  \( H^1 (\mathcal{O})\).

\begin{lemma}\label{lem:time}
Let Assumptions \ref{assum2} and \ref{assum1} hold, and let $\phi^0\in \dot{H}^\beta(\mathcal{O})$ for some $\beta\in (\frac{d}{2},2)\cap[1,2)$. %{\b same comment for $\beta$?}. 
Then for any $p\ge 1$, there exists a constant $C:=C(p,\beta)>0$ such that 
\begin{align*}
	\|X^{n+1}- X^n\|_{L^{2p}(\Omega;H)}\leq C\tau^{\frac{\beta}{4}},\qquad 
	n\in\{0,1,\cdots,N-1\}. 
\end{align*}

\end{lemma}
\begin{proof}
The proof is standard, relying on \eqref{assum1_eq1}, \eqref{assum1_eq2}, and the $H^\beta(\mathcal{O})$-spatial regularity estimate of $X^n$ (see Remark \ref{rem:H2-reg} for $\beta=2$ and Lemma \ref{coro1} for $\beta<2$). We therefore omit the details and refer the reader to, for example, \cite[Theorem 2.31]{KR14} for a similar argument. 	
\end{proof}
% {\color{red} for $\beta=2$ it seems that you need additional works?} {\color{blue}see Remark 4.1.}

%  {\b improve the following eng}
% {\b $\phi$ or $\phi^0$? improve eng. "may still be ?"}
% {\r In fact, even if the initial value $\phi^0 \notin \dot{H}^2(\mathcal{O})$, the numerical solution may still be $\frac{1}{2}$-order temporally H\"older continuous for times away from the initial time \(t_0 = 0\). }

We notice that  even if the initial value $\phi^0 \notin \dot{H}^2(\mathcal{O})$, 
one can still observe $\tfrac12$-H\"older continuity in time for the numerical solution 
away from the initial time $t_0 = 0$.

%{\r if eventually in our theorem, we impose $H^2$-regularity, Lemma 4.4. is enough, right?} {\color{blue}We need to use \eqref{lemma6_eq3_R1} in section 6.}
\begin{lemma}\label{lemma6}
Let Assumptions \ref{assum2} and \ref{assum1} hold, and let $\phi^0\in \dot{H}^{\beta}(\mathcal{O})$ for some $\beta\in(\frac{d}{2},2)\cap[1,2)$. Then for any $p\ge 1$, there exists a constant $C:=C(p,\beta)>0$ %{\b depending on $T$?} 
such that for any $n\in\{1,2,\cdots,N-1\}$,
\begin{align}
	\label{lemma6_Re}
	\|X^{n+1}- X^n\|_{L^{2p}(\Omega;H)}\leq  C\tau^{\frac{1}{2}}(1+t_n^{\frac\beta4-\frac12}). 
\end{align}

%In addition, $\|X^{n+1}- X^n\|_{L^{2p}(\Omega;H)}\le C\tau^{\frac\beta 4}$ for all $n\in\{0,1,\cdots,N-1\}$.
\end{lemma}
\begin{proof}
For any $n=1,2,\cdots,N-1$, we can decompose
\begin{align*}
	X^{n+1}-X^n&= (S(\tau)-I)X^n +\int_{t_n}^{t_{n+1}}S(t_{n+1}-s)A\tilde{f}^nds
	\\
	&\quad+\int_{t_n}^{t_{n+1}}S(\tau)g(X^n)dW(s)
	=:R_1^n+R_2^n+R_3^n.
\end{align*}
Utilizing \eqref{assum1_eq2}, the Minkowski inequality, and \eqref{lemma6_eq3_R12_12}, it holds that 
\begin{align}\label{lemma6_eq3_R2}
	\mathbb{E}[\|R_2^n\|^{2p}]\leq C \mathbb{E}\bigg[\Big|\int_{t_n}^{t_{n+1}}(t_{n+1}-s)^{-\frac{1}{2}}
	\|\tilde{f}^n\|ds\Big|^{2p}\bigg]
	\le C\tau^p \mathbb{E}[\|\tilde{f}^n\|^{2p}]
	\leq C\tau^{p}.
\end{align}
By the BDG inequality, the contractivity of $S(\cdot)$, and \eqref{lema1_eq_R1_insert_Q}, we derive
\begin{align}\label{lemma6_eq3_R3}
	\mathbb{E}[\|R_3^n\|^{2p}]\leq C\mathbb{E}\bigg[\Big|\int_{t_n}^{t_{n+1}}\|S(\tau)g(X^n)
	\|_{\mathcal{L}_2^0}^{2}ds\Big|^{p}\bigg]\leq C\tau^p.
\end{align}

To estimate $R_1^n$ for $1\le n\le N$, recall that 
\begin{align*}
	X^n=S(t_n) \phi^0+\sum_{i=0}^{n-1}\int_{t_i}^{t_{i+1}}S(t_n-s) A\tilde{f}^i d s+\sum_{i=0}^{n-1}\int_{t_i}^{t_{i+1}} S(t_n-t_i)g(X^i) dW(s).
\end{align*}
Since ${\phi^0}\in \dot{H}^\beta(\mathcal{O})$, in view of \eqref{assum1_eq1} and \eqref{assum1_eq2}, for any $\alpha\in[0,1]$,
\begin{align}\label{eq:intial}\notag
	\|(S(\tau)-I)S(t_n) {\phi^0}\|&=\|(S(\tau)-I)(-A)^{-1-\alpha}(-A)^{-\frac\beta2+1+\alpha} S(t_n)(-A)^{\frac\beta2}{\phi^0}\|\\
	&\le C(\beta,\alpha)t_n^{\frac{\beta}{4}-\frac12\alpha-\frac12}\tau^{\frac12(1+\alpha)}.
\end{align}  
% {\r why $\alpha \neq \frac 12$?}
Since $\phi^0\in \dot{H}^{1}(\mathcal{O})$, the BDG inequality, \eqref{assum1_eq1}, \eqref{assum1_eq2},  \eqref{lema1_eq_R6-new}, and {Corollary \ref{lemma_stability}} imply that for any $\alpha\in[0,\frac12)$,
\begin{align}\label{eq:sto-re}
	&\mathbb{E}\bigg[\Big\|\sum_{i=0}^{n-1}\int_{t_i}^{t_{i+1}} (S(\tau)-I)S(t_n-t_i)g(X^i) dW(s)\Big\|^{2p}\bigg]\\\notag
	%	&\le C \mathbb{E}\bigg[\Big|\sum_{i=0}^{n-1}\int_{t_i}^{t_{i+1}} \|(S(\tau)-I)S(t_n-t_i) g(X^i)\|_{\mathcal{L}_2^0}^2 ds\Big|^p\bigg]\\\notag
	&\le C \mathbb{E}\bigg[\Big|\sum_{i=0}^{n-1}\int_{t_i}^{t_{i+1}} \|(S(\tau)-I)(-A)^{-1-\alpha} (-A)^{\frac12+\alpha}S(t_n-t_i) (-A)^{\frac12}g(X^i)\|_{\mathcal{L}_2^0}^2 ds\Big|^p\bigg]\\\notag
	&\le C \mathbb{E}\bigg[\Big|\sum_{i=0}^{n-1}\int_{t_i}^{t_{i+1}} \tau^{1+\alpha}(t_n-t_i)^{-\frac12-\alpha}(1+\|\nabla X^i\|^2) ds\Big|^p\bigg]\le C\tau^{p(1+\alpha)}.
\end{align}
{
	Due to \eqref{assum1_eq1} and \eqref{assum1_eq2},  for any $\alpha\in[0,\frac12)$,
	\begin{align*}
		&\bigg\|\sum_{i=0}^{n-1}\int_{t_i}^{t_{i+1}}(S(\tau)-I)S(t_n-s) A\tilde{f}^i d s\bigg\|\\
		&\le \sum_{i=0}^{n-1}\int_{t_i}^{t_{i+1}}\|(-A)^{-1-\alpha}(S(\tau)-I)\|_{\mathcal{L}(H)}
		\|(-A)^{\frac32+\alpha}S(t_n-s) \|_{_{\mathcal{L}(H)}}
		\|(-A)^{\frac12}\tilde{f}^i\| ds\\
		&\le C\tau^{\frac12(1+\alpha)}\sum_{i=0}^{n-1}\int_{t_i}^{t_{i+1}}(t_n-s)^{-\frac34-\frac12\alpha}\|(-A)^{\frac12}\tilde{f}^i\| ds.
	\end{align*}
	%  Then as a result of \eqref{eq:fnH1} and \eqref{lemma6_beta=2_1_insert}, it holds that for any $q\ge1$, 
	% \begin{equation}\label{eq:fder}
		% \mathbb{E}\left[\|(-A)^{\frac12}\tilde{f}^n\|^q\right]+\mathbb{E}\left[\|(-A)^{\frac12}F'(X^n)\|^q\right]\le C.
		% \end{equation}
	To proceed, we claim that for any $q\ge1$,
	\begin{equation}\label{eq:fnH1}
		\mathbb{E}\left[\|(-A)^{\frac12} \tilde{f}^n\|^{q}\right]\le C(q).
	\end{equation}
	The proof of \eqref{eq:fnH1}  can be found in Appendix \ref{App-fnH1}.
	%{\r It seems that you say this twice. why?}
	Invoking \eqref{eq:fnH1} and the Minkowski inequality, it follows that for any $p\ge1$ and $\alpha\in[0,\frac12)$,     }
\begin{equation*}%\label{eq:dete-re}
	\mathbb{E}\bigg[\Big\|\sum_{i=0}^{n-1}\int_{t_i}^{t_{i+1}}(S(\tau)-I)S(t_n-s) A\tilde{f}^i d s\Big\|^{2p}\bigg]\le C(p,\alpha)\tau^{p(1+\alpha)}.
\end{equation*}
This, together with \eqref{eq:intial} and \eqref{eq:sto-re}, yields that for any $p\ge1$ and $\alpha\in [0,\frac12)$,
\begin{equation}\label{lemma6_eq3_R1}
	\|R_1^n\|_{L^{2p}(\Omega;H)}=\|(S(\tau)-I)X^n \|_{L^{2p}(\Omega;H)}	\le  C(\beta,\alpha)(1+t_n^{\frac{\beta}{4}-\frac12\alpha-\frac12})\tau^{\frac12(1+\alpha)}.
\end{equation}
for all $n\in\{1,2,\cdots,N\}$.
Consequently, the required H\"older continuity estimate \eqref{lemma6_Re} comes from \eqref{lemma6_eq3_R1} with $\alpha=0$, \eqref{lemma6_eq3_R2}, and \eqref{lemma6_eq3_R3}.
% {\r you can put some part of the proof in the appendix}
The proof is completed.
\end{proof}

%We conclude this section by establishing  {\r an a priori bound for the accumulated discrete increment sum in  \( H^1 (\mathcal{O})\)}.
The following lemma  provides a moment estimate for the discrete quadratic variation of the numerical solution in the energy space $\dot{H}^1(\mathcal{O})$.
%{\b do we need explain why the following summation implies temporal H\"older regularity? }
%{\r we need say some words to end this section here?}

\begin{lemma}\label{lemma5}
Let Assumptions \ref{assum2} and \ref{assum1} hold, and let $\phi^0\in \dot{H}^1(\mathcal{O})$. Then for any $p\ge 1$, there exists a constant $C:=C(p)>0$ such that 
\begin{align*}
	\mathbb{E}\bigg[\bigg(\sum_{n=0}^{N-1}
	\|(-A)^{\frac12}(X^{n+1}-X^n)\|^{2}\bigg)^p\bigg]\leq C.
\end{align*}
\begin{proof}
	From \eqref{schme_phi-intro}, we have 
	\begin{align}\label{lemma4_eq1}
		X^{n+1}-X^n
		&=(S(\tau)-I)(-A)^{-1}\tilde{\mu}^n+g(X^n)\delta  W^n,
	\end{align}
	where $\tilde{\mu}^n$ is defined in \eqref{eq:tildemun}.
	By H\"older's inequality and Young's inequality, 
	\begin{align}\label{lemma4_eq2}
		&\bigg(\sum_{n=0}^{N-1}\|(-A)^{\frac12}(X^{n+1}-X^n)\|^{2}\bigg)^p\\
		&\leq C\Big|\sum_{n=0}^{N-1}\left\langle{(I-S(\tau))}\tilde{\mu}^n,-(X^{n+1}-X^n)\right\rangle\Big|^p+C\Big|\sum_{n=0}^{N-1}\langle g(X^n)\delta  W^n,(-A)(X^{n+1}-X^n)\rangle\Big|^p\notag\\
		&\le \frac{1}{4}\bigg(\sum_{n=0}^{N-1}\|{(I-S(\tau))^{\frac12}(-A)^{\frac12}}(X^{n+1}-X^n)\|^{2}\bigg)^p+C\bigg(\sum_{n=0}^{N-1}\|{(I-S(\tau))^{\frac12}(-A)^{-\frac12}}\tilde{\mu}^n\|^{2}\bigg)^p\notag\\
		&\quad+\frac{1}{4}\bigg(\sum_{n=0}^{N-1}\|(-A)^{\frac12}(X^{n+1}-X^n)\|^{2}\bigg)^p+C\bigg(\sum_{n=0}^{N-1}\|(-A)^{\frac 12} (g(X^n)\delta  W^n)\|^{2}\bigg)^p.\notag
	\end{align}
	% \begin{align}\label{lemma4_eq2}
		% 	&\mathbb{E}\bigg[\bigg(\sum_{n=0}^{N-1}\|(-A)^{\frac12}(X^{n+1}-X^n)\|^{2}\bigg)^p\bigg]\\
		% 	&\leq C\mathbb{E}\bigg[\Big|\sum_{n=0}^{N-1}\left\langle{(I-S(\tau))}\tilde{\mu}^n,-(X^{n+1}-X^n)\right\rangle\Big|^p\bigg]\notag\\
		%              &\quad+C\mathbb{E}\bigg[\Big|\sum_{n=0}^{N-1}\langle g(X^n)\delta  W^n,(-A)(X^{n+1}-X^n)\rangle\Big|^p\bigg]\notag\\
		% 	&\le \frac{1}{4}\mathbb{E}\bigg[\bigg(\sum_{n=0}^{N-1}\|{(I-S(\tau))^{\frac12}(-A)^{\frac12}}(X^{n+1}-X^n)\|^{2}\bigg)^p\bigg]\notag\\&\quad+C\mathbb{E}\bigg[\bigg(\sum_{n=0}^{N-1}\|{(I-S(\tau))^{\frac12}(-A)^{-\frac12}}\tilde{\mu}^n\|^{2}\bigg)^p\bigg]\notag\\
		% 	&\quad+\frac{1}{4}\mathbb{E}\bigg[\bigg(\sum_{n=0}^{N-1}\|(-A)^{\frac12}(X^{n+1}-X^n)\|^{2}\bigg)^p\bigg]\notag\\ &\quad+C\mathbb{E}\bigg[\bigg(\sum_{n=0}^{N-1}\|(-A)^{\frac 12} (g(X^n)\delta  W^n)\|^{2}\bigg)^p\bigg].\notag
		% \end{align}
	Using the fact that \( \| (I - S(\tau))^{\frac{1}{2}} \|_{\mathcal{L}(H)} \leq 1 \), it follows that
	\begin{align}\label{lemma4_eq2_1}
		\|({I-S(\tau)})^{\frac12}(-A)^{\frac 12}(X^{n+1}-X^n)\|^{2}
		\leq \|(-A)^{\frac12}(X^{n+1}-X^n)\|^{2}.
	\end{align}
	Applying \eqref{lema1_eq_R6} and {Corollary \ref{lemma_stability}}, we have 
	\begin{align}\label{lemma4_eq2_3}
		\mathbb{E}\bigg[\bigg(\sum_{n=0}^{N-1}\|(-A)^{\frac 12} (g(X^n)\delta  W^n)\|^{2}\bigg)^p\bigg]%= \mathbb{E}\bigg[\bigg(\sum_{n=0}^m\left\|\nabla (g(X^n)\delta  W^n)\right\|^{2}\bigg)^p\bigg]
		\le C.
	\end{align}
	Taking expectations on both sides of \eqref{lemma4_eq2}, and then on account of
	\eqref{lemma4_eq2_1} and \eqref{lemma4_eq2_3}, as well as using {Corollary \ref{lemma_stability}}, we complete the proof.
\end{proof}
\end{lemma}

\section{Proof of Theorem \ref{Thm3_main}}\label{section_converge}
{This section is devoted to proving Theorem \ref{Thm3_main} on the strong convergence rate of the exponential Euler SSAV scheme \eqref{schme_phi-intro}. 
To this end,}
inspired by e.g., \cite{CHS21,  hong2024density}, we introduce an auxiliary process $\Phi=\{\Phi(t)\}_{t \in [0, T]}$, defined by
\begin{align}\label{auxiliary_solution}
d\Phi(t) &= -A^2\Phi(t)dt+AF'(X^{\kappa_N(t)/\tau})dt+g(X^{\kappa_N(t)/\tau})dW(t),\qquad t\in(0,T],
\end{align}
with the initial value $\Phi(0) = \phi^0$.
By the triangle inequality, we have
\begin{equation}\label{eq:phiX}
\left\|\phi(t_n) - X^{n}\right\|_{L^2(\Omega; H)} 
\leq \left\|\phi(t_n) - \Phi(t_n)\right\|_{L^2(\Omega; H)} 
+ \left\|\Phi(t_n) - X^{n}\right\|_{L^2(\Omega; H)}.
\end{equation}
In the following, we provide estimates for the errors between the auxiliary process $\Phi$ and the numerical solution $X$, as well as between $\Phi$ and the exact solution $\phi$. The corresponding analyses are given in subsections~\ref{S:4.1} and~\ref{S:4.2}, respectively.

\subsection{Error estimate between $\Phi$ and $X$}\label{S:4.1}
By applying the H\"older's inequality, it follows from \eqref{Xn_contin} and \eqref{auxiliary_solution} that for any $p\ge1$ and $n\in\{0,1,\cdots,N-1\}$,      
\begin{align}\label{Thm1_6}\notag
\left\|\Phi(t_{n+1})-\b X(t_{n+1})\right\|^p &\leq C\bigg\|\sum_{j=0}^n\int_{t_{j}}^{t_{j+1}}S(t_{n+1}-s)A[F'(X^{j})-\tilde{f}^{j}]ds\bigg\|^p\\ \notag
&\quad+C\bigg\|\sum_{j=0}^{n}\int_{t_{j}}^{t_{j+1}}S(t_{n+1}-s)\left(I-S(s-t_j)\right)g(X^{j})dW(s)\bigg\|^p\\
&=:\mathcal{B}_1^n+\mathcal{B}_2^n.
\end{align}   
Applying the BDG inequality, \eqref{assum1_eq1}, \eqref{assum1_eq2},  \eqref{lema1_eq_R6-new}, and {Corollary \ref{lemma_stability}}, we deduce that
\begin{align}\label{Thm1_6_B}\notag
&\mathbb{E}\left[\mathcal{B}_{2}^n\right] \leq  C\mathbb{E}\bigg[\Big\|\sum_{j=0}^{n}\int_{t_{j}}^{t_{j+1}}(-A)^{\frac{1}{2}}S(t_{n+1}-s)(-A)^{-\frac12}\left(I-S(s-t_{j})\right)g(X^{j})dW(s)\Big\|^p\bigg]\\
& \leq C\mathbb{E}\bigg[\bigg(\sum_{j=0}^{n}\int_{t_{j}}^{t_{j+1}}(t_{n+1}-s)^{-\frac{1}{2}}\|(-A)^{-1}\left(S(s-t_j)-I\right)(-A)^{\frac{1}{2}}g(X^j)\|^2_{\mathcal{L}_2^0}ds\bigg)^{\frac{p}{2}}\bigg]\notag\\
&\leq C\mathbb{E}\bigg[\bigg(\sum_{j=0}^{n}\int_{t_{j}}^{t_{j+1}}(t_{n+1}-s)^{-\frac{1}{2}}|s-t_j|\|(-A)^{\frac{1}{2}}g(X^j)\|^2_{\mathcal{L}_2^0}ds\bigg)^{\frac{p}{2}}\bigg]\notag\\
&\leq C\tau^{\frac{p}{2}}\mathbb{E}\bigg[\bigg(\sum_{j=0}^{n}\int_{t_{j}}^{t_{j+1}}(t_{n+1}-s)^{-\frac{1}{2}}(1 +\|\nabla  X^j\|^{2})ds\bigg)^{\frac{p}{2}}\bigg] \leq C\tau^{\frac{p}{2}}.
\end{align}
According to \eqref{fn}, we further split the term $\mathcal{B}_1^n$ as follows
\begin{align}\label{Thm1_6_A}
\mathcal{B}_1^n
&\leq C\bigg\|\sum_{j=0}^n\int_{t_{j}}^{t_{j+1}}S(t_{n+1}-s)A\bigg(\frac{r^{j+1}-r^j}{\sqrt{E_{\textup{p}}(X^{j})}}F'(X^{j})\bigg)ds\bigg\|^p\\\notag
&\quad+C\bigg\|\sum_{j=0}^n\int_{t_{j}}^{t_{j+1}}S(t_{n+1}-s)A\bigg(\frac{r^{j}-\sqrt{E_{\textup{p}}(X^{j})}}{\sqrt{E_{\textup{p}}(X^{j})}}F'(X^{j})+\chi^j\bigg)ds\bigg\|^p\\\notag
&
=:\mathcal{B}_{11}^n+\mathcal{B}_{12}^n. 
\end{align}
For the first term $\mathcal{B}_{11}^n$, using \eqref{assum1_eq1}, the Minkowski inequality, and Young's inequality, we have 
\begin{align*}
\mathbb{E}\left[\mathcal{B}_{11}^n\right]
&\leq C\mathbb{E}\bigg[\bigg(\sum_{j=0}^n\int_{t_{j}}^{t_{j+1}}\Big\|S(t_{n+1}-s)(-A)^{\frac{1}{2}}\frac{r^{j+1}-r^j}{\sqrt{E_{\textup{p}}(X^{j})}}(-A)^{\frac{1}{2}}F'(X^{j})\Big\|ds\bigg)^p\bigg]\\
&\leq C\mathbb{E}\bigg[\bigg(\sum_{j=0}^n\int_{t_{j}}^{t_{j+1}}(t_{n+1}-s)^{-\frac{1}{4}}
|r^{j+1}-r^j|\|(-A)^{\frac{1}{2}}F'(X^{j})\|ds\bigg)^p\bigg]\\
&\leq C\tau^{\frac{p}{2}}\mathbb{E}\bigg[\bigg(\sum_{j=0}^n|r^{j+1}-r^j|^2\bigg)^p\bigg]\\
&\quad+C\tau^{-\frac{p}{2}}\mathbb{E}
\bigg[\bigg(\sum_{j=0}^n\bigg(\int_{t_{j}}^{t_{j+1}}(t_{n+1}-s)^{-\frac{1}{4}}\|(-A)^{\frac{1}{2}}F'(X^{j})\|ds\bigg)^2
\bigg)^p\bigg].
\end{align*}
Furthermore, by {Corollary \ref{lemma_stability}}, {the Minkowski inequality}, and \eqref{lemma6_beta=2_1_insert}, we obtain 
\begin{align}\label{Thm1_6_A_term1}
\mathbb{E}\left[\mathcal{B}_{11}^n\right]
%&\leq C\tau^{\frac{p}{2}}
%+ C\tau^{\frac{p}{2}}\left(\sum_{j=0}^n\int_{t_{j}}^{t_{j+1}}(t_{n+1}-s)^{-\frac{1}{2}}\|(-A)^{\frac{1}{2}}
%	F'(X^j)\|^2_{L^{2p}(\Omega;H)}ds\right)^{p}
\leq C\tau^{\frac{p}{2}}.
\end{align}

The estimate of the term $\mathcal{B}_{12}^n$ requires the following result,  which implies that the discrete SSAV \( r^{m} \)  approximates \( \sqrt{E_{\textup{p}}(X^{m}) } \) for all \( m \in \{0, 1, \cdots, N\} \), with a convergence rate of order $1/2$.
\begin{lemma}\label{lemma3}
Let Assumptions \ref{assum2} and \ref{assum1} hold, and let $\phi^0\in \dot{H}^2(\mathcal{O})$. Then for any $p \geq 1$, there exists a constant $C:=C(p) > 0$ such that for any $m \in \{0, 1, \cdots, N\}$,
\begin{align*}
	\mathbb{E}\left[\big|r^{m}-\sqrt{E_{\textup{p}}( X^{m})}\big|^p\right]\leq C\tau^{\frac{p}{2}}.
\end{align*}
\end{lemma}
\begin{proof}
To start with, we take the first- and second-order Taylor expansions of the real valued function $b_1(\xi):=E_{\textup{p}}(\xi X^{n+1}+(1-\xi)X^n)$ for $\xi\in[0,1]$, which yields
\begin{align*}
	& E_{\textup{p}}( X^{n+1})-E_{\textup{p}}( X^n)=b_1(1)-b_1(0)
	=\int_0^1\left\langle F'( \psi_\xi^n), X^{n+1}- X^n\right\rangle d\xi,\\
	& E_{\textup{p}}( X^{n+1})-E_{\textup{p}}( X^n) =\left\langle F'( X^n), X^{n+1}- X^n\right\rangle\\
	& +\int_0^1(1-\xi)\left\langle F''( \psi_\xi^n)( X^{n+1}- X^n),X^{n+1}- X^n\right\rangle d\xi,
\end{align*}
where $\psi_\xi^n:=\xi X^{n+1}+(1-\xi)X^n$.
Similarly, by taking the Taylor expansion of $b_2(\theta):=\sqrt{\theta E_{\textup{p}}( X^{n 
		+1})+(1-\theta)E_{\textup{p}}( X^{n 
	})}$ for $\theta\in[0,1]$, we further have
%{\r what is $E_1$?}
%      \begin{align*}
	% &\sqrt{E_{\textup{p}}( X^{n+1})}-\sqrt{E_{\textup{p}}(X^{n})} =b(1)-b(0)
	% \\\notag
	% &=\frac{1}{2\sqrt{E_{\textup{p}}(X^{n})}}\left(E_{\textup{p}}( X^{n+1})-E_{\textup{p}}( X^n)\right)-\frac{1}{8E_{\textup{p}}(X^{n})^{3/2}}\left(E_{\textup{p}}( X^{n+1})-E_{\textup{p}}( X^n)\right)^2 \\
	% &\quad+\frac{3}{16}\int_0^1(1-\theta)^2(\tilde{E}_\theta^n)^{-5/2}\left(E_{\textup{p}}( X^{n+1})-E_{\textup{p}}( X^n)\right)^3 d\theta,\\
	%          &=\frac{1}{2\sqrt{E_{\textup{p}}(X^{n})}}\big(\left\langle F'( X^n), X^{n+1}- X^n\right\rangle+\frac{1}{2}\left\langle F''( X^n)( X^{n+1}- X^n),X^{n+1}- X^n\right\rangle\big)\notag\\
	% &\quad +\frac{1}{2\sqrt{E_{\textup{p}}(X^{n})}}\int_0^1(1-\xi)\left\langle (F''(\psi_\xi^n)-F''( X^n))( X^{n+1}- X^n), X^{n+1}- X^n\right\rangle d\xi\notag\\
	% &\quad -\frac{1}{8E_{\textup{p}}(X^{n})^{3/2}}\left\langle F'( X^n),X^{n+1}- X^n\right\rangle^2\notag\\
	% &\quad-\frac{1}{4E_{\textup{p}}(X^{n})^{3/2}}\langle F'( X^n), X^{n+1}- X^n\rangle\int_0^1(1-\xi)\left\langle F''( \psi_\xi^n)( X^{n+1}- X^n),X^{n+1}- X^n\right\rangle d\xi\notag\\
	% &\quad -\frac{1}{8E_{\textup{p}}(X^{n})^{3/2}}\bigg(\int_0^1(1-\xi)\left\langle F''( \psi_\xi^n)( X^{n+1}- X^n),X^{n+1}- X^n\right\rangle d\xi\bigg)^2 \\
	%         &\quad+\frac{3}{16}\int_0^1(1-\theta)^2(\tilde{E}_\theta^n)^{-5/2}\left(\int_0^1\left\langle F'( \psi_\xi^n), X^{n+1}- X^n\right\rangle d\xi\right)^3 d\theta,
	%          \end{align*}
\begin{align*}
	&\sqrt{E_{\textup{p}}( X^{n+1})}-\sqrt{E_{\textup{p}}(X^{n})} =b_2(1)-b_2(0)
	\\\notag
	&=\frac{1}{2\sqrt{E_{\textup{p}}(X^{n})}}\left(E_{\textup{p}}( X^{n+1})-E_{\textup{p}}( X^n)\right)-\frac{1}{8E_{\textup{p}}(X^{n})^{3/2}}\left(E_{\textup{p}}( X^{n+1})-E_{\textup{p}}( X^n)\right)^2 \\
	&\quad+\frac{3}{16}\int_0^1(1-\theta)^2(\tilde{E}_\theta^n)^{-5/2}\left(E_{\textup{p}}( X^{n+1})-E_{\textup{p}}( X^n)\right)^3 d\theta,\\
	&=\frac{1}{2\sqrt{E_{\textup{p}}(X^{n})}}\Big(\left\langle F'( X^n), X^{n+1}- X^n\right\rangle+\frac{1}{2}\left\langle F''( X^n)( X^{n+1}- X^n),X^{n+1}- X^n\right\rangle\Big)\notag\\
	%&\quad +\frac{1}{2\sqrt{E_{\textup{p}}(X^{n})}}\int_0^1(1-\xi)\left\langle (F''(\psi_\xi^n)-F''( X^n))( X^{n+1}- X^n), X^{n+1}- X^n\right\rangle d\xi\notag\\
	&\quad -\frac{1}{8E_{\textup{p}}(X^{n})^{3/2}}\left\langle F'( X^n),X^{n+1}- X^n\right\rangle^2\notag+\mathcal{P}_1^n+\mathcal{P}_2^n+\mathcal{P}_3^n+\mathcal{P}_4^n,
	% &\quad-\frac{1}{4E_{\textup{p}}(X^{n})^{3/2}}\langle F'( X^n), X^{n+1}- X^n\rangle\int_0^1(1-\xi)\left\langle F''( \psi_\xi^n)( X^{n+1}- X^n),X^{n+1}- X^n\right\rangle d\xi\notag\\
	% &\quad -\frac{1}{8E_{\textup{p}}(X^{n})^{3/2}}\bigg(\int_0^1(1-\xi)\left\langle F''( \psi_\xi^n)( X^{n+1}- X^n),X^{n+1}- X^n\right\rangle d\xi\bigg)^2 \\
	%         &\quad+\frac{3}{16}\int_0^1(1-\theta)^2(\tilde{E}_\theta^n)^{-5/2}\left(\int_0^1\left\langle F'( \psi_\xi^n), X^{n+1}- X^n\right\rangle d\xi\right)^3 d\theta,
\end{align*}
where $\tilde{E}_\theta^n:=\theta E_{\textup{p}}(X^n)+(1 - \theta) E_{\textup{p}}(X^{n+1})$
and the remainder terms are  
\begin{align*}
	\mathcal{P}_1^n&:=\frac{1}{2\sqrt{E_{\textup{p}}(X^{n})}}\int_0^1(1-\xi)\left\langle (F''(\psi_\xi^n)-F''( X^n))( X^{n+1}- X^n), X^{n+1}- X^n\right\rangle d\xi,\\
	\mathcal{P}_2^n&:=-\frac{1}{4E_{\textup{p}}(X^{n})^{3/2}}\langle F'( X^n), X^{n+1}- X^n\rangle\notag\\
	&\quad\times\int_0^1(1-\xi)\left\langle F''( \psi_\xi^n)( X^{n+1}- X^n),X^{n+1}- X^n\right\rangle d\xi,\notag\\
	\mathcal{P}_3^n&:=-\frac{1}{8E_{\textup{p}}(X^{n})^{3/2}}\bigg(\int_0^1(1-\xi)\left\langle F''( \psi_\xi^n)( X^{n+1}- X^n),X^{n+1}- X^n\right\rangle d\xi\bigg)^2, \\
	\mathcal{P}_4^n&:=\frac{3}{16}\int_0^1(1-\theta)^2(\tilde{E}_\theta^n)^{-5/2}\left(\int_0^1\left\langle F'( \psi_\xi^n), X^{n+1}- X^n\right\rangle d\xi\right)^3 d\theta.
\end{align*}
From \eqref{lemma4_eq1} and \eqref{scheme_r_sto-intro}, we can reformulate
\begin{align*}
	\sqrt{E_{\textup{p}}( X^{n+1})}-\sqrt{E_{\textup{p}}(X^{n})}%(\mathcal{P}_1^n+\mathcal{P}_2^n+\mathcal{P}_3^n+\mathcal{P}_4^n)\\&=r^{n+1}-r^n+\frac{1}{4\sqrt{E_{\textup{p}}(X^{n})}}\left\langle F''( X^n)( X^{n+1}- X^n),(S(\tau)-I)(-A)^{-1}\tilde{\mu}^n\right\rangle \notag\\
	% &\quad -\frac{1}{8E_{\textup{p}}(X^{n})^{3/2}}\left\langle F'( X^n),X^{n+1}- X^n\right\rangle \left\langle F'( X^n),(S(\tau)-I)(-A)^{-1}\tilde{\mu}^n\right\rangle\notag\\
	=(r^{n+1}-r^n)+\mathcal{P}_1^n+\mathcal{P}_2^n+\mathcal{P}_3^n+\mathcal{P}_4^n+\mathcal{P}_5^n+\mathcal{P}_6^n,
\end{align*}
where 
\begin{align*}
	\mathcal{P}_5^n&:= \frac{1}{4\sqrt{E_{\textup{p}}(X^{n})}}\left\langle F''( X^n)( X^{n+1}- X^n),(S(\tau)-I)(-A)^{-1}\tilde{\mu}^n\right\rangle,\\
	\mathcal{P}_6^n&: -\frac{1}{8E_{\textup{p}}(X^{n})^{3/2}}\left\langle F'( X^n),X^{n+1}- X^n\right\rangle \left\langle F'( X^n),(S(\tau)-I)(-A)^{-1}\tilde{\mu}^n\right\rangle.
\end{align*}
% {\r check the above computation, you skip some steps      
	% Another suggestion: use integral form for mean value. 
	%}
%        
%	 From \eqref{lemma4_eq1},  the above identity can be reformulated as
%		\begin{align*}
	%			&\sqrt{E_{\textup{p}}( X^{n+1})}-\sqrt{E_{\textup{p}}(X^{n})}\\&=r^{n+1}-r^n+\frac{1}{4\sqrt{E_{\textup{p}}(X^{n})}}\left\langle F''( X^n)( X^{n+1}- X^n),(S(\tau)-I)(-A)^{-1}\tilde{\mu}^n\right\rangle \notag\\
	%			&\quad +\frac{1}{2\sqrt{E_{\textup{p}}(X^{n})}}\int_0^1(1-\xi)\left\langle (F''(\psi_\xi^n)-F''( X^n))( X^{n+1}- X^n), X^{n+1}- X^n\right\rangle d\xi\notag\\
	%			&\quad -\frac{1}{8E_{\textup{p}}(X^{n})^{3/2}}\left\langle F'( X^n),X^{n+1}- X^n\right\rangle \left\langle F'( X^n),(S(\tau)-I)(-A)^{-1}\tilde{\mu}^n\right\rangle\notag\\
	%			&\quad-\frac{1}{8E_{\textup{p}}(X^{n})^{3/2}}\left\langle F'( X^n), X^{n+1}- X^n\right\rangle \left\langle F''(\psi_1^n)( X^{n+1}- X^n),X^{n+1}- X^n\right\rangle \notag\\
	%			&\quad -\frac{1}{32E_{\textup{p}}(X^n)^{3/2}}\left\langle F''(\psi_1^n)( X^{n+1}- X^n), X^{n+1}- X^n\right\rangle ^2\notag\\
	%			&\quad +\frac{1}{16(\tilde{E}^n_{\textup{p}})^{5/2}}\left\langle F'(\psi_2^n), X^{n+1}- X^n\right\rangle^3\notag\\
	%			&=:r^{n+1}-r^n+P_1^n+P_2^n+P_3^n+P_4^n+P_5^n+P_6^n.
	%		\end{align*}
Summing from $n=0$ to $m$, since $r_0 = \sqrt{E_\textup{p}( \phi^0)}$,  we obtain from H\"older's inequality that for any $p\ge1$,
\begin{equation}\label{lemma3_eq4}
	\mathbb{E}\left[\Big|r^{m+1}-\sqrt{E_\textup{p}( X^{m+1})}\Big|^p\right]\leq C\sum_{l=1,5,6}\mathbb{E}\bigg[\Big|\sum_{n=0}^m \mathcal{P}_l^n\Big|^p\bigg]+C(m+1)^{p-1}\sum_{l=2}^4\sum_{n=0}^m\mathbb{E}\left[|\mathcal{P}_l^n|^p\right].
\end{equation}

{	\textbf{Estimate of $\sum_{n=0}^m\mathcal{P}_1^n$.} Invoking the mean value theorem, we can rewrite
	\begin{equation*}%\label{lemma3_eq4_P2}
		\mathcal{P}_1^n\!=\! \frac{1}{2\sqrt{E_{\textup{p}}(X^{n})}}\int_0^1(1-\xi)\int_0^1\int_{\mathcal{O}} F'''(\theta\psi_\xi^n+(1-\theta)X^n)(\psi_\xi^n-X^n) (X^{n+1}- X^n)^2 dxd\theta d\xi.
		%      \\
		% \Big|\sum_{n=0}^m P_2^n\Big|&=\Big|\sum_{n=0}^m\frac{1}{4\sqrt{E_{\textup{p}}(X^{n})}}\left\langle F'''(\theta\psi_1^n+(1-\theta)X^n)(\psi_1^n-X^n)( X^{n+1}- X^n),X^{n+1}- X^n\right\rangle\Big|\\
		% &\le C\Big|\sum_{n=0}^m\|F'''({\psi_3^n})\|_{L^\infty}\| X^{n+1}- X^n\|^3_{L^3}\Big|.
	\end{equation*}
	The Gagliardo--Nirenberg inequality (see, e.g., \cite[Chapter 5]{adams2003sobolev}) reads that for $d\in\{1,2,3\}$ and any $p\in(2,6]$,
	\begin{equation}\label{eq:GN}
		\|X^{n+1}- X^n\|_{L^p}\le C(p,d)\|X^{n+1}- X^n\|^{1-d(\frac12-\frac{1}{p})}\|\nabla (X^{n+1}- X^n)\|^{d(\frac12-\frac{1}{p})}.
	\end{equation}
	Hence, by \eqref{lema1_eq_R1_insert_hey}, Assumption \ref{assum2}, H\"older's inequality, and \eqref{eq:GN} with $p=3$,
	\begin{align*}
		|\mathcal{P}_1^n|
		&\le C(1+\|X^n\|_{L^\infty}+\|X^{n+1}\|_{L^\infty}) \|X^{n+1}- X^n\|_{L^3}^3\\
		&\le C(1+\|X^n\|_{L^\infty}+\|X^{n+1}\|_{L^\infty}) \|X^{n+1}- X^n\|^{3(1-\frac{d}{6})}\|\nabla (X^{n+1}- X^n)\|^{\frac{d}{2}}.
	\end{align*}
	%where the second step is due to the Gagliardo--Nirenberg inequality {\b can we write the equality and cite it later?} .
	Moreover, by  Young's inequality and Poincare's inequality, for $d\in\{1,2,3\}$,
	\begin{align*}
		|\mathcal{P}_1^n|
		&\le C(1+\|X^n\|_{L^\infty}+\|X^{n+1}\|_{L^\infty}) \|X^{n+1}- X^n\|^{\frac32}\|\nabla (X^{n+1}- X^n)\|^{\frac32}\\
		&\le C\tau^{-\frac{3}{2}}(1+\|X^n\|_{L^\infty}^4+\|X^{n+1}\|_{L^\infty}^4) \|X^{n+1}- X^n\|^{6}+C\tau^{\frac12}\|\nabla (X^{n+1}- X^n)\|^{2}
	\end{align*}
	It follows from H\"older's inequality, \eqref{eq:Linfty}, as well as Lemmas \ref{lem:time} and \ref{lemma5} that}
\begin{align*}%\label{lemma3_eq4_P2}
	\mathbb{E}\bigg[\Big|\sum_{n=0}^m \mathcal{P}_1^n\Big|^p\bigg]
	&\leq C\mathbb{E}\bigg[\bigg(\tau^{\frac{1}{2}}\sum_{n=0}^m\| (-A)^{\frac12}(X^{n+1}- X^n)\|^2\bigg)^p\bigg]\\
	&\quad+C\mathbb{E}\bigg[\bigg(\tau^{-\frac{3}{2}}\sum_{n=0}^m\left(1+\|X^n\|^4_{L^\infty}+\|X^{n+1}\|^4_{L^\infty}\right)\| X^{n+1}- X^n\|^6\bigg)^p\bigg]\\
	&\leq C\tau^{\frac{p}{2}}.
\end{align*}
% {\r please check the result in lemma 3.5 will affect the proof or not}

%{\r the subtitle could be better}
\textbf{Estimate of $\sum_{n=0}^m\mathcal{P}_5^n$.}
Young's inequality, \eqref{assum1_eq2}, and {Corollary \ref{lemma_stability}} lead to
\begin{align}\label{lemma3_eq4_P1}
	& \mathbb{E}\bigg[\Big|\sum_{n=0}^m \mathcal{P}_5^n\Big|^p\bigg]  
	%&=\mathbb{E}\bigg[\Big|\sum_{n=0}^m\frac{1}{4\sqrt{E_{\textup{p}}(X^{n})}}\bigg\langle (I-S(\tau))^{\frac{1}{2}}(-A)^{-\frac 12}F''( X^n)( X^{n+1}- X^n),(I-S(\tau))^{\frac{1}{2}}(-A)^{-\frac 12}\tilde{\mu}^n\bigg\rangle\Big|^p\bigg]\notag\\
	\leq C\mathbb{E}\bigg[\Big|\sum_{n=0}^m\tau^{-\frac{1}{2}}\|
	(I-S(\tau))^{\frac{1}{2}}(-A)^{-\frac 12}[F''( X^n)( X^{n+1}- X^n)]\|^2\\
	&\qquad\qquad\qquad\qquad+C\tau^{\frac{1}{2}}\sum_{n=0}^m\|(I-S(\tau))^{\frac{1}{2}}(-A)^{-\frac 12}\tilde{\mu}^n\|^2\Big|^p\bigg]\notag\\
	&\leq C\mathbb{E}\bigg[\Big|\sum_{n=0}^m\tau^{-\frac{1}{2}}\| (I-S(\tau))^{\frac{1}{2}}(-A)^{-1}(-A)^{\frac{1}{2}}[F''( X^n)( X^{n+1}- X^n)]|^2\Big|^p\bigg]+ C\tau^{\frac{p}{2}}\notag\\
	&\leq C\mathbb{E}\bigg[\Big|\sum_{n=0}^m\tau^{\frac{1}{2}}\| (-A)^{\frac{1}{2}}[F''( X^n)( X^{n+1}- X^n)]\|^2\Big|^p\bigg]+ C\tau^{\frac{p}{2}}.\notag
\end{align} 
Using the integration by parts formula, the chain rule, H\"older's inequality, Assumption~\ref{assum2}, {the Sobolev embedding $\dot{H}^1(\mathcal{O})\hookrightarrow L^6(\mathcal{O})$, 
	% {\r this estimate is dimension dependent. however, the below estimate is independent of $d$? please check and modify the presentation}, 
	we have that for any $\beta\in(\frac {3}{2},2)$,		\begin{align*}%\label{lemma3_eq4_P1_term1}
		&\| (-A)^{\frac{1}{2}}[F''( X^n)( X^{n+1}- X^n)]\|^2=\left\| \nabla [F''( X^n)( X^{n+1}- X^n)]\right\|^2\\
		&\leq C\left\| F'''( X^n)\nabla X^n( X^{n+1}- X^n)\right\|^2+C\left\|F''( X^n)\nabla( X^{n+1}- X^n)\right\|^2\\
		&\leq C\| X^{n+1}- X^n\|^2_{L^6}\left\| { (1+X^n)}\nabla X^n\right\|^2_{L^3}
		+C\left\|F''( X^n)\right\|^2_{L^\infty}\|\nabla( X^{n+1}- X^n)\|^2\\
		&\leq C\|\nabla( X^{n+1}- X^n)\|^{2}\left({(1+\|X^n\|^2_{L^\infty})}\|\nabla X^n\|^2_{L^3}+C(1+\| X^n\|^4_{L^\infty})\right)\\
		& \leq C\|\nabla( X^{n+1}- X^n)\|^2{(1+\|X^n\|^{4}_{H^{\beta}})},
	\end{align*}
	where we have used the Sobolev embeddings $ H^{\beta-1}(\mathcal{O})\hookrightarrow L^3(\mathcal{O})$ and $H^{\beta}(\mathcal{O})\hookrightarrow L^\infty(\mathcal{O})$ in the last step.}
Then the combination of Lemma \ref{coro1} and Lemma~\ref{lemma5} gives
\begin{align*}%\label{lemma3_eq4_P1final}
	&\mathbb{E}\bigg[\Big|\sum_{n=0}^m\| (-A)^{\frac{1}{2}}[F''( X^n)( X^{n+1}- X^n)]\|^2\Big|^p\bigg]\\
	&\leq C\left(1+\mathbb{E}\left[\sup_{0\leq n\leq m}\| X^n\|^{8p}_{H^{\beta}}\right]\right)^{\frac{1}{2}}\bigg(\mathbb{E}\bigg[\bigg(\sum_{n=0}^m\| \nabla(X^{n+1}- X^n)\|^2\bigg)^{2p}\bigg]\bigg)^{\frac 12} 
	\le C.
\end{align*}
Inserting the above estimate into \eqref{lemma3_eq4_P1} yields  $\mathbb{E}[|\sum_{n=0}^m \mathcal{P}_5^n|^p]\le C\tau^{\frac{p}{2}}.$

{\textbf{Estimate of $\sum_{n=0}^m\mathcal{P}_6^n$.}}
By  Cauchy--Schwarz inequality {and \eqref{lema1_eq_R1_insert_hey},}%{\r the property of $E_p$ is also used}, 
\begin{align*}%\label{lemma3_eq4_P3}\notag
	\Big|\sum_{n=0}^m \mathcal{P}_6^n\Big|^2%&= \bigg|\sum_{n=0}^m\frac{1}{8E_{\textup{p}}(X^{n})^{\frac{3}{2}}}\left\langle F'( X^n), X^{n+1}- X^n\right\rangle\left\langle F'( X^n),(I-S(\tau))(-A)^{-1}\tilde{\mu}^n\right\rangle
	%\bigg|^2\\
	&\leq C\bigg(\sum_{n=0}^m\left\|F'( X^n)\right\|^2\| X^{n+1}- X^n\|^2\|(I-S(\tau))^{\frac 12}(-A)^{-\frac{1}{2}}F'( X^n)\|^2\bigg)\\
	&\quad\times\bigg(\sum_{n=0}^m\|(I-S(\tau))^{\frac{1}{2}}(-A)^{-\frac 12}
	\tilde{\mu}^n\|^2\bigg).
\end{align*}
Furthermore, utilizing \eqref{assum1_eq2},   {Corollary \ref{lemma_stability}}, and H\"older's inequality, 
\begin{align*}%\label{lemma3_eq4_P3}\notag
	&\mathbb{E}\bigg[\Big|\sum_{n=0}^m \mathcal{P}_6^n\Big|^p\bigg]
	%			&= \mathbb{E}\bigg[\bigg|\sum_{n=0}^m\frac{1}{8E_{\textup{p}}(X^{n})^{\frac{3}{2}}}\left\langle F'( X^n), X^{n+1}- X^n\right\rangle\left\langle F'( X^n),(I-S(\tau))(-A)^{-1}\tilde{\mu}^n\right\rangle\bigg|^p\bigg]\\
	%			&\leq C\mathbb{E}\bigg[\bigg|\sum_{n=0}^m\left\|F'( X^n)\right\|\| X^{n+1}- X^n\|\|(I-S(\tau))^{\frac 12}(-A)^{-\frac{1}{2}}F'( X^n)\|\|(I-S(\tau))^{\frac{1}{2}}(-A)^{-\frac 12}\tilde{\mu}^n\|\bigg|^p\bigg]\\
	%	& \leq C\left(\mathbb{E}\left[\bigg(\sum_{n=0}^m\|F'( X^n)\|^2\| X^{n+1}- X^n\|^2\|(I-S(\tau))^{\frac 12}(-A)\|^2\|(-A)^{\frac{1}{2}}F'( X^n)\|^2\bigg)^p\right]\right)^{\frac12}\\
	\leq C\tau^{\frac{p}{2}}\left(\mathbb{E}\bigg[\bigg(\sum_{n=0}^m\|F'( X^n)\|^2\| X^{n+1}- X^n\|^2\|(-A)^{\frac{1}{2}}F'( X^n)\|^2\bigg)^p\bigg]\right)^{\frac12}\\
	&\leq C\tau^{\frac{p}{2}}(m+1)^{\frac{p-1}{2}}\left(\sum_{n=0}^m\mathbb{E}\left[\|F'( X^n)\|^{2p}\| X^{n+1}- X^n\|^{2p}\|(-A)^{\frac{1}{2}}F'( X^n)\|^{2p}\right] \right)^{\frac12}.
\end{align*}
%The assumption $f'(0)=0$,  the integration by parts formula, and \eqref{lemma6_beta=2_1_insert} yield that for any $q\ge1$,
%\begin{align*}
%  \mathbb{E}\left[\|(-A)^{\frac12} F'(X^n)\|{\color{red}^{2q}}\right]=  \mathbb{E}\left[\|\nabla F'(X^n)\|{\color{red}^{2q}}\right]\le C(q,T).
%\end{align*}
Together with Lemma \ref{lem:time}, \eqref{eq:FXn}, and \eqref{lemma6_beta=2_1_insert},
this proves $\mathbb{E}[|\sum_{n=0}^m \mathcal{P}_6^n|^p]
\le C\tau^{\frac{p}{2}}.$

%{\r why there is no summation on the index $n$ below }{\color{blue}see \eqref{lemma3_eq4} for more details.}
\textbf{Estimates of $\mathcal{P}_2^n$ and $\mathcal{P}_3^n$.}
In virtue of Assumption \ref{assum2} and \eqref{eq:Linfty},  it holds that for any $q\ge1$ and $\xi\in[0,1]$, 
\begin{equation}\label{eq:F2dpsi1}
	\mathbb{E}\left[\|F''(\psi_\xi^n)\|_{L^{\infty}}^q\right]\le  C\mathbb{E}\left[(1+\|X^n\|^{2}_{L^{\infty}}+\|X^{n+1}\|^{2}_{L^{\infty}})^q\right]\le C(q).
\end{equation}
Applying H\"older's inequality, {\color{blue}\eqref{lema1_eq_R1_insert_hey},} Lemma \ref{lem:time}, \eqref{eq:FXn}, and \eqref{eq:F2dpsi1}, we infer that
\begin{align*}%\label{lemma3_eq4_P4}
	\mathbb{E}\left[\left|\mathcal{P}_2^n\right|^p\right]
	% &\le C\int_0^1\mathbb{E}\left[\left|\left\langle F'( X^n), X^{n+1}- X^n\right\rangle \left\langle F''(\psi_\xi^n)( X^{n+1}- X^n),X^{n+1}- X^n\right\rangle\right|^p\right]d\xi\notag\\
	&\leq C\int_0^1\mathbb{E}\left[\left\|F'( X^n)\right\|^p\left\|F''(\psi_\xi^n)\right\|^p_{L^{\infty}}\| X^{n+1}- X^n\|^{3p}\right]d\xi
	\leq C\tau^{\frac{3}{2}p}.
\end{align*}
Analogously, we also have
\begin{align*}%\label{lemma3_eq4_P5}
	\mathbb{E}\left[\left|\mathcal{P}_3^n\right|^p\right]
	\leq C\int_0^1\mathbb{E}\left[\|F''(\psi_\xi^n)\|^{2p}_{L^\infty}\| X^{n+1}- X^n\|^{4p}\right]d\xi
	\le C\tau^{2p}.
\end{align*}

\textbf{Estimate of $\mathcal{P}_4^n$.}
By {\color{blue}\eqref{lema1_eq_R1_insert_hey},} H\"older's inequality,  \eqref{eq:Linfty}, and Lemma \ref{lem:time}, 
\begin{align*}%\label{lemma3_eq4_P6}
	\mathbb{E}\left[\left| \mathcal{P}_4^n\right|^p\right]&\leq C \int_0^1\mathbb{E}\left[\left|\left\langle F'(\psi_\xi^n), X^{n+1}- X^n\right\rangle\right|^{3p}\right]d\xi\\
	%   &\le C\mathbb{E}\left[\left(1+\|X^n\|^{9p}_{L^6}+\|X^{n+1}\|^{9p}_{L^6}\right)\| X^{n+1}- X^n\|^{3p}\right]\\
	&\leq C\mathbb{E}\left[\left(1+\|X^n\|^{9p}_{L^\infty}+\|X^{n+1}\|^{9p}_{L^\infty}\right)\| X^{n+1}- X^n\|^{3p}\right]\leq C\tau^{\frac{3}{2}p}.
\end{align*}
Finally, substituting the above estimates on \( \mathcal{P}_1^n\) through \(\mathcal{P}_6^n \) into \eqref{lemma3_eq4}, we complete the proof.
\end{proof}

%{\r better to reorganized the structure and presentation }

{
With Lemma \ref{lemma3} at hand, we now present the error estimate between the auxiliary process \(\Phi(t_n)\) and the numerical solution \(X^n\).}
\begin{proposition}\label{Thm1}
Let Assumptions \ref{assum2} and \ref{assum1} hold, and let $\phi^0\in \dot{H}^2(\mathcal{O})$.  Then for any $p\geq 1$, there exists a constant $C>0$ such that for $n \in\{ 0,1,\cdots,N-1\}$,
\begin{align*}
	\mathbb{E} \left[\left\|\Phi(t_{n+1})-X^{n+1}\right\|^p\right]\leq C \tau^{\frac{p}{2}}.
\end{align*}
\end{proposition}
\begin{proof}

{
	Recalling that by \eqref{Thm1_6} and \eqref{Thm1_6_A},
	\begin{align}\label{eq:Phi-X-new}
		\left\|\Phi(t_{n+1})-X^{n+1}\right\|^p &\leq \mathcal{B}_{11}^n+\mathcal{B}_{12}^n+\mathcal{B}_2^n,
	\end{align}   
	where the  expectations of $\mathcal{B}_{11}^n$ and $\mathcal{B}_{2}^n$ have been estimated in \eqref{Thm1_6_A_term1} and \eqref{Thm1_6_B}, respectively. To estimate $\mathcal{B}_{12}^n$, we utilize \eqref{lema1_eq_R1_insert_hey}, \eqref{eq:chinp}, \eqref{eq:FXn}, and Lemma \ref{lemma3} to obtain that for any $p\ge1$,
	\begin{align*}
		& \bigg\|\frac{r^{j}-\sqrt{E_{\textup{p}}(X^{j})}}{\sqrt{E_{\textup{p}}(X^{j})}}F'(X^{j})+\chi^j\bigg\|_{L^p(\Omega;H)}\\
		&\le C\|r^{j}-\sqrt{E_{\textup{p}}(X^{j})}\|_{L^{2p}(\Omega)}
		\|F'(X^{j})\|_{L^{2p}(\Omega;H)}+\|\chi^{j}\|_{L^{p}(\Omega;H)}\le C(p)\tau^{\frac12}.
	\end{align*}
	Then} using  the Minkowski inequality and \eqref{assum1_eq1}, one has 
\begin{align*}%\label{Thm1_6_A_term2}
	\mathbb{E}\left[\mathcal{B}_{12}^n\right]
	% \leq C\bigg(\sum_{j=0}^n\int_{t_{j}}^{t_{j+1}}(t_{n+1}-s)^{-\frac{1}{2}}
	% \bigg\|\frac{r^{j}-\sqrt{E_{\textup{p}}(X^{j})}}{\sqrt{E_{\textup{p}}(X^{j})}}F'(X^{j})+\chi^j\bigg\|_{L^p(\Omega;H)}ds\bigg)^p
	\leq C \tau^{\frac{p}{2}}.
\end{align*}
This, along with  \eqref{eq:Phi-X-new}, \eqref{Thm1_6_A_term1}, and \eqref{Thm1_6_B}, finishes the proof.
\end{proof}
\subsection{Error estimate between $\phi$ and $\Phi$}\label{S:4.2}
In this subsection, we estimate the  error 
$\tilde{\varepsilon}(t):=\phi(t)-\Phi(t)$ 
between the mild solution $\phi$ and the auxiliary process $\Phi$.  According to \eqref{auxiliary_solution} and \eqref{model}, we have
\begin{align}\label{Thm2_1}
d\tilde{\varepsilon}(t)&=-A^2\tilde{\varepsilon}(t)dt+A\left(F'(\phi(t))-F'(X(\kappa_N(t)))\right)dt\\\notag
&\quad+\left(g(\phi(t))-g(X(\kappa_N(t)))\right)dW(t)
\end{align}
for any $t\in(0,T]$,
with the initial value $\tilde{\varepsilon}(0)=0$. To handle the non-globally Lipschitz nonlinearity $F'$, we decompose
\begin{equation*}
F'(\phi(t))-F'(X(\kappa_N(t))=[F'(\phi(t))-F'(\Phi(t))]+[F'(\Phi(t))-F'(X(\kappa_N(t)))],
\end{equation*}
where the first term on the right-hand side can be handled by the one-sided Lipschitz continuity of $-F'$. As for the second term, taking advantage of the local Lipschitz continuity property \eqref{assum2_eq2} of $F'$, it follows 
\begin{align}\label{eq:F'-F'}\notag
\|F'(\Phi(t))-F'(X(\kappa_N(t)))\|\le C(1+\|\Phi(t)\|_{L^\infty}^2+\|X(\kappa_N(t))\|_{L^\infty}^2)\|\Phi(t)-X(\kappa_N(t))\|\\
\le C(1+\|\Phi(t)\|_{L^\infty}^2+\|X(\kappa_N(t))\|_{L^\infty}^2)
(\|\Phi(t)-\Phi(\kappa_N(t))\|+\|\Phi(\kappa_N(t))-X(\kappa_N(t))\|).
\end{align}
To proceed, 
we will use the spatial regularity and temporal H\"older regularity estimates for the auxiliary process $\Phi$,  
provided in Lemmas~\ref{sec4_lem2} and~\ref{sec4_lem1}, respectively.
\begin{lemma}\label{sec4_lem2}
Let Assumptions \ref{assum2} and \ref{assum1} hold and let $\phi^0\in \dot{H}^{\beta}(\mathcal{O})$ for some $\beta\in(\frac{d}{2},2)\cap[1,2)$. Then for any $p \geq 1$, there exists a constant $C:=C(p) > 0$ such that
\begin{equation}\label{eq:PhiLinf}
	\mathbb{E}\bigg[\sup_{t\in[0,T]}\|\Phi(t)\|_{L^\infty}^p\bigg]\leq C.
\end{equation}
\end{lemma}
\begin{proof}
In view of \eqref{auxiliary_solution}, %, we have  that
%         for any $t \in[0,T] $,  
%        \begin{equation*}%\label{sec4_lem2_eq1}
	%            \Phi(t) = S(t)\phi^0+\int_{0}^{t}S(t-s)AF'(X(\kappa_{N}(s)))ds+\int_{0}^{t}S(t-s)g(X(\kappa_{N}(s)))dW(s).
	%        \end{equation*}
similar to the proof of {Lemma \ref{coro1}}, an application of  Lemma~\ref{lemma_tool} with $\tilde{F}(s)=$ and $\tilde{G}(s)=g(X(\kappa_{N}(s)))$ can produce that for any $\beta\in[1,2)$,
\begin{equation}\label{eq:Aphit}
	\mathbb{E}\bigg[\sup_{t\in[0,T]}\|(-A)^{\beta}\Phi(t)\|^p\bigg]\leq C.
\end{equation}
Finally, the  result \eqref{eq:PhiLinf} comes from the Sobolev embedding $H^{\beta}(\mathcal{O})\hookrightarrow L^\infty(\mathcal{O})$ for any $\beta> d/2$.
\end{proof}

\begin{lemma}\label{sec4_lem1}
Let Assumptions \ref{assum2} and \ref{assum1} hold and let $\phi^0\in \dot{H}^2(\mathcal{O})$. Then for any $p\ge 1$, there exists a constant $C:=C(p) > 0$ such that for any $ 0\leq s<t\leq T$,
\begin{align*}
	%\label{sec4_lem1_Re}
	\mathbb{E}\left[\left\|\Phi(t)-\Phi(s)\right\|^p\right]\leq C(t-s)^{\frac{p}{2}}.
\end{align*}
\end{lemma}
Since the proof of Lemma \ref{sec4_lem1} follows essentially the same arguments as in Lemma~\ref{lem:time}, it is omitted for brevity.
Now based on \eqref{Thm2_1} and the one-sided Lipschitz continuity of $-F'$, we can establish a strong error bound for $\tilde{\varepsilon}({t})$ in $H^{-1}(\mathcal{O})$.
\begin{lemma}\label{Lemma4Thm2}
Let Assumptions \ref{assum2} and \ref{assum1} hold, and let $\phi^0\in \dot{H}^2(\mathcal{O})$.   Then for any $p\geq 2$, there exists a constant $C := C(p)>0$ such that for any $t\in[0,T]$,
\begin{align*}
	%\label{Thm2_34_2_Re1} 
	&\mathbb{E}\left[\|(-A)^{-\frac{1}{2}}\tilde{\varepsilon}({t})\|^p\right]+\mathbb{E}\left[\int_{0}^{{t}}\|(-A)^{-\frac{1}{2}}\tilde{\varepsilon}(s)\|^{p-2}\|(-A)^{\frac 12}\tilde{\varepsilon}(s)\|^2ds\right]\leq C\tau^{\frac{p}{2}}.
\end{align*}
\end{lemma}
\begin{proof}    
Recalling \eqref{Thm2_1} and then applying It\^o formula to $\|(-A)^{-\frac{1}{2}}\tilde{\varepsilon}(t)\|^p$, we obtain
\begin{align}\label{Thm2_2}
	&d\|(-A)^{-\frac{1}{2}}\tilde{\varepsilon}(t)\|^p=
	-p\|(-A)^{-\frac{1}{2}}\tilde{\varepsilon}(t)\|^{p-2}\|(-A)^{\frac 12}\tilde{\varepsilon}(t)\|^2 dt\\
	%			p\|(-A)^{-\frac{1}{2}}\tilde{\varepsilon}(t)\|^{p-2}\left\langle -A^2\tilde{\varepsilon}(t),(-A)^{-1}\tilde{\varepsilon}(t)\right\rangle dt\\
	&\quad+p\|(-A)^{-\frac{1}{2}}\tilde{\varepsilon}(t)\|^{p-2}\left\langle AF'(\phi(t))-AF'(X(\kappa_N(t))),(-A)^{-1}\tilde{\varepsilon}(t)\right\rangle dt\notag\\
	&\quad+p\|(-A)^{-\frac{1}{2}}\tilde{\varepsilon}(t)\|^{p-2}\left\langle\left(g(\phi(t))-g(X(\kappa_N(t)))\right)dW(t),(-A)^{-1}\tilde{\varepsilon}(t)\right\rangle\notag\\
	&\quad+\frac{1}{2}p\|(-A)^{-\frac{1}{2}}\tilde{\varepsilon}(t)\|^{p-2}
	\|(-A)^{-\frac{1}{2}}
	(g(\phi(t))-g(X(\kappa_N(t)))) \|^2_{\mathcal{L}_2^0}dt\notag\\
	&\quad+\frac{1}{2}p(p-2)\|(-A)^{-\frac{1}{2}}\tilde{\varepsilon}(t)\|^{p-4}\notag\\
	&\qquad\qquad\times\sum_{k=1}^{\infty}\left\langle(-A)^{-\frac{1}{2}}\tilde{\varepsilon}(t),(-A)^{-\frac{1}{2}}\left(g(\phi(t))-g(X(\kappa_N(t)))\right)Q^{\frac{1}{2}}e_k\right\rangle ^2dt.\notag\\
	&=:\mathcal{J}_1dt + \mathcal{J}_2dt + \mathcal{J}_3 + \mathcal{J}_4dt + \mathcal{J}_5dt.\notag
\end{align}

%		For the term $\mathcal{J}_1$, we have
%		\begin{align}\label{Thm2_2_J1}
	%			\mathcal{J}_1 = &-p\|(-A)^{-\frac{1}{2}}\tilde{\varepsilon}(t)\|^{p-2}\|(-A)^{\frac 12}\tilde{\varepsilon}(t)\|^2.
	%		\end{align}
Using the one sided Lipschitz continuity \eqref{assum2_eq1} of $-F'$, the Cauchy--Schwarz inequality, and the first inequality of \eqref{eq:F'-F'}, we arrive at
\begin{align}\label{Thm2_2_J2}
	%{\color{blue}\bar{\mathcal{J}}}_{2}&:=
	&-\left\langle F'(\phi(t))-F'(\Phi(t))+F'(\Phi(t))-F'(X(\kappa_N(t))),\tilde{\varepsilon}(t)\right\rangle \\\notag
	& \leq L_f\left\|\tilde{\varepsilon}(t)\right\|^2+ \left\|\tilde{\varepsilon}(t)\right\|\|F'(\Phi(t))-F'(X(\kappa_N(t)))\|\\
	&\le {L_f'}\langle (-A)^{\frac{1}{2}}\tilde{\varepsilon}(t),(-A)^{-\frac{1}{2}}\tilde{\varepsilon}(t)\rangle+
	C|\mathbb{I}(t)|^2,\notag
\end{align}
% {\r why name it as $J_{22}$?}
where { $L_f':=L_f+\frac{1}{2}$} and \begin{equation}\label{eq:I1t}
	\mathbb{I}(t):=(1+\|\Phi(t)\|_{L^\infty}^2+\|X(\kappa_N(t))\|_{L^\infty}^2)\|\Phi(t)-X(\kappa_N(t))\|,\quad t\in[0,T].
\end{equation}
By virtue of \eqref{Thm2_2_J2} and using Young's inequality,  for any $\upsilon>0$ one has  
\begin{align}\label{Thm2_2_J1_update}
	\mathcal{J}_2 &{\le p\|(-A)^{-\frac{1}{2}}\tilde{\varepsilon}(t)\|^{p-2}({L_f'}\langle (-A)^{\frac{1}{2}}\tilde{\varepsilon}(t),(-A)^{-\frac{1}{2}}\tilde{\varepsilon}(t)\rangle+
		C|\mathbb{I}(t)|^2)}\\
	%&=p\|(-A)^{-\frac{1}{2}}\tilde{\varepsilon}(t)\|^{p-2}{\color{blue}\bar{\mathcal{J}}}_{2}\\   
	&\leq p{L_f'}\left(\frac{\upsilon^2}{2}\|(-A)^{-\frac{1}{2}}\tilde{\varepsilon}(t)\|^{p-2}\|(-A)^{\frac{1}{2}}\tilde{\varepsilon}(t)\|^2+\frac{1}{2\upsilon^2}\|(-A)^{-\frac{1}{2}}\tilde{\varepsilon}(t)\|^p\right)\notag\\
	&\quad+
	C\|(-A)^{-\frac{1}{2}}\tilde{\varepsilon}(t)\|^{p}+C|\mathbb{I}(t)|^p.\notag
\end{align}

{To estimate the terms $\mathcal{J}_4$ and $\mathcal{J}_5$ on the right hand side of \eqref{Thm2_2}, we denote
	\begin{equation}\label{eq:I2t}
		{\mathbb{K}_{p}}(t):=\|(-A)^{-\frac{1}{2}}\tilde{\varepsilon}(t)\|^{p-2}
		\|(-A)^{-\frac{1}{2}}(g(\phi(t))-g(X(\kappa_N(t))))\|^2_{\mathcal{L}_2^0}
	\end{equation}
	for $t\in[0,T]$, so that 
	\begin{align}\label{Thm2_2_J5} 
		\mathcal{J}_4=\frac12 p{\mathbb{K}_{p}}(t)\qquad \text{and}\qquad\mathcal{J}_5\le \frac12 p(p-2){\mathbb{K}_{p}}(t).
\end{align}}	
%For $p\ge1$ and $t\in[0,T]$, by denoting
%		{\r strange , why label it. the index is sub, may cause confusing. better to use supindex   }
%		and 
%		
%		Using H\"older's inequality, it follows that 
%		\begin{align}\label{Thm2_2_J5}
	%			\mathcal{J}_5
	%			%\leq\frac{1}{2}p(p-2)\|(-A)^{-\frac{1}{2}}\tilde{\varepsilon}(t)\|^{p-2}\left\|(-A)^{-\frac{1}{2}}\left(g(\phi(t))-g(X(\kappa_N(t)))\right)\right\|_{\mathcal{L}_2^0}^2
	%			&\le\frac{1}{2}p(p-2)\|(-A)^{-\frac{1}{2}}\tilde{\varepsilon}(t)\|^{p-2}
	%			\|(-A)^{-\frac{1}{2}}
	%			(g(\phi(t))-g(X(\kappa_N(t)))) \|^2_{\mathcal{L}_2^0}\\\notag
	%			&=(p-2) \mathcal{J}_4.
	%		\end{align}
Plugging the estimates \eqref{Thm2_2_J1_update} and \eqref{Thm2_2_J5}  into \eqref{Thm2_2} gives
\begin{align}\label{Thm2_32} 
	&d\|(-A)^{-\frac{1}{2}}\tilde{\varepsilon}(t)\|^p+p\|(-A)^{-\frac{1}{2}}\tilde{\varepsilon}(t)\|^{p-2}
	\|(-A)^{\frac 12}\tilde{\varepsilon}(t)\|^2dt\\
	&\leq pL_f'\left(\frac{\upsilon^2}{2}\|(-A)^{-\frac{1}{2}}\tilde{\varepsilon}(t)\|^{p-2}\|(-A)^{\frac{1}{2}}\tilde{\varepsilon}(t)\|^2dt+\frac{1}{2\upsilon^2}\|(-A)^{-\frac{1}{2}}\tilde{\varepsilon}(t)\|^pdt\right)\notag\\
	&\quad+
	C\|(-A)^{-\frac{1}{2}}\tilde{\varepsilon}(t)\|^{p}dt+C|\mathbb{I}(t)|^p dt+\frac{1}{2}p(p-1)    \mathbb{K}_{p}(t)dt\notag\\
	&\quad+p\|(-A)^{-\frac{1}{2}}\tilde{\varepsilon}(t)\|^{p-2}\left\langle\left(g(\phi(t))-g(X(\kappa_N(t)))\right)dW(t),(-A)^{-1}\tilde{\varepsilon}(t)\right\rangle.\notag
\end{align}
% {\r why not use $J_5$ here?}
In this proof, we set $\upsilon>0$ with {$L_f'\upsilon^2 = 1$}. Integrating \eqref{Thm2_32} over time and then taking expectations on both sides, we derive that
for any $t\in[0,T]$,
\begin{align}\label{Thm2_32_2} 
	&\mathbb{E}\left[\|(-A)^{-\frac{1}{2}}\tilde{\varepsilon}( t )\|^p\right]+\frac{p}{2}\mathbb{E}\int_{0}^{ t }\|(-A)^{-\frac{1}{2}}\tilde{\varepsilon}(s)\|^{p-2}\|(-A)^{\frac 12}\tilde{\varepsilon}(s)\|^2ds\\\notag
	&\leq C\mathbb{E}\int_{0}^{ t }\|(-A)^{-\frac{1}{2}}\tilde{\varepsilon}(s)\|^pds+C\int_{0}^{{t}}\mathbb{E}[|\mathbb{I}(s)|^p]ds
	+\frac{1}{2}p(p-1)\int_0^t\mathbb{E}[ {\mathbb{K}_{p}}(s)]ds,
\end{align} 
since the expectation of the stochastic integral vanishes. {To proceed, we next estimate $\mathbb{I}(\cdot)$ and $\mathbb{K}_{p}(\cdot)$	individually.

	For the term $\mathbb{I}(\cdot)$ defined in  \eqref{eq:I1t}, by H\"older's inequality, it holds that for any $p\ge1$ and $s\in[0,T]$,
	\begin{align*}
		\mathbb{E}[|\mathbb{I}(s)|^p]&\le C\left(\mathbb{E}\left[1+\|\Phi(s)\|_{L^\infty}^{4p}+\|X(\kappa_N(s))\|_{L^\infty}^{4p}\right]\right)^{\frac{1}{2}}	\left(\mathbb{E}\left[\|\Phi(s)-X(\kappa_N(s))\|^{2p}\right]\right)^{\frac{1}{2}}.
\end{align*}}	
Lemma~\ref{sec4_lem1} and Proposition~\ref{Thm1} ensure that for any $q\ge1$ and $s\in[0,T]$,
\begin{align}\label{eq:Phi-X}
	\mathbb{E}\left[\|\Phi(s)-X(\kappa_N(s))\|^q\right]&\le   C\mathbb{E}\left[\left\|\Phi(s)-\Phi(\kappa_N(s))\right\|^q\right]\\\notag
	&\quad
	+C\mathbb{E}\left[ \left\|\Phi(\kappa_N(s))-X(\kappa_N(s))\right\|^q\right]\le C(q)\tau^{\frac{q}{2}}.
\end{align}
Further taking 
\eqref{eq:Linfty} and Lemma \ref{sec4_lem2} into account, we conclude that for any $p\ge1$, 
\begin{equation}\label{Thm2_34_1} 
	\mathbb{E}[|\mathbb{I}(s)|^p]
	%&\le C\left(\mathbb{E}\left[1+\|\Phi(s)\|_{L^\infty}^{4p}+\|X(\kappa_N(s))\|_{L^\infty}^{4p}\right]\right)^{\frac{1}{2}}	\left(\mathbb{E}\left[\|\Phi(s)-X(\kappa_N(s))\|^{2p}\right]\right)^{\frac{1}{2}}\\\notag
	\leq C(p)\tau^{\frac{p}{2}},\qquad s\in[0,T].
\end{equation}
%  {\r comment: the current proof is jumpy, you may make it better}

%		Taking the expectation on both sides of \eqref{Thm2_32_1}, and using \eqref{Thm2_34_1},
%		one can infer that for any $t\in[0,T]$,
%		\begin{align}\label{Thm2_32_2} 
	%			&\mathbb{E}\left[\|(-A)^{-\frac{1}{2}}\tilde{\varepsilon}( t )\|^p\right]+\frac{p}{2}\mathbb{E}\int_{0}^{ t }\|(-A)^{-\frac{1}{2}}\tilde{\varepsilon}(s)\|^{p-2}\|(-A)^{\frac 12}\tilde{\varepsilon}(s)\|^2ds\\\notag
	%			&\leq C\mathbb{E}\int_{0}^{ t }\|(-A)^{-\frac{1}{2}}\tilde{\varepsilon}(s)\|^pds+C\tau^{\frac{p}{2}}
	%			+\frac{1}{2}p(p-1)\int_0^t\mathbb{E}[ {\color{blue}\mathbb{K}_{p}}(s)]ds,
	%		\end{align} 

To estimate $ {\mathbb{K}_{p}}(\cdot)$ defined in \eqref{eq:I2t}, we notice that for $p=2$,  
by  \eqref{assume_Lip} and Young's inequality, for any $\upsilon_1 >0$,
\begin{align}\label{Thm2_34_3_insert} \notag
	{\mathbb{K}_2}(s)& = \|(-A)^{-\frac{1}{2}}(g(\phi(s))-g(X(\kappa_N(s))))\|^2_{\mathcal{L}_2^0}	\le  C\left\|g(\phi(s))-g(X(\kappa_N(s)))\right\|^2_{\mathcal{L}_2^0}\\
	& \le C\|\Phi(s)-X(\kappa_N(s))\|^2
	\leq C\left\|\tilde{\varepsilon}(s)\right\|^2+C\|\Phi(s)-X(\kappa_N(s))\|^2
	\notag\\
	&\leq \frac{1}{2}\upsilon_1^2\|(-A)^{\frac{1}{2}}\tilde{\varepsilon}(s)\|^2+C(\upsilon_1)\|(-A)^{-\frac{1}{2}}\tilde{\varepsilon}(s)\|^2 +C\|\Phi(s)-X(\kappa_N(s))\|^2.
\end{align}  
Now selecting $\upsilon_1=1/\sqrt{p-1}$ in \eqref{Thm2_34_3_insert}, and then utilizing  
Young's inequality, we have that for any $p\ge2$,
\begin{align*}
	{\mathbb{K}_{p}}(s)&= \|(-A)^{-\frac{1}{2}}\tilde{\varepsilon}(t)\|^{p-2} {\mathbb{K}}_2(s)
	\leq  C(\upsilon_1)\|(-A)^{-\frac{1}{2}}\tilde{\varepsilon}(s)\|^{p}\\
	&+\frac{1}{2(p-1)}
	\|(-A)^{-\frac{1}{2}}\tilde{\varepsilon}(s)\|^{p-2}
	\|(-A)^{\frac{1}{2}}\tilde{\varepsilon}(s)\|^{2}+\|\Phi(s)-X(\kappa_N(s))\|^{p}. 
\end{align*}
Taking expectations and then integrating over time on both sides of the above inequality, we infer  from  \eqref{eq:Phi-X} that 
\begin{align}\label{eq:Jps}
	& \int_0^t\mathbb{E}[ {\mathbb{K}_{p}}(s)]ds%=\int_0^t\mathbb{E}\left[ \|(-A)^{-\frac{1}{2}}\tilde{\varepsilon}(t)\|^{p-2} {\color{blue}\mathbb{K}}_2(s)\right]ds\\\notag
	%				&\leq  C(\upsilon_1)\int_{0}^{{t}}\mathbb{E}\left[\|(-A)^{-\frac{1}{2}}\tilde{\varepsilon}(s)\|^{p}\right]ds{+\frac12{\upsilon_1^{2}}\int_{0}^{{t}}\mathbb{E}\left[\|(-A)^{-\frac{1}{2}}\tilde{\varepsilon}(s)\|^{p-2}\|(-A)^{\frac{1}{2}}\tilde{\varepsilon}(s)\|^{2}\right]ds} \\
	%				&\quad+C\int_{0}^{{t}}\mathbb{E}\left[\|\Phi(s)-X(\kappa_N(s))\|^{p} \right]ds\notag\\\notag
	\leq  C\int_{0}^{{t}}\mathbb{E}\left[\|(-A)^{-\frac{1}{2}}\tilde{\varepsilon}(s)\|^{p}\right]ds\\
	&\quad + \frac{1}{2(p-1)} \int_{0}^{t}\mathbb{E}\left[\|(-A)^{-\frac{1}{2}}\tilde{\varepsilon}(s)\|^{p-2}\|(-A)^{\frac{1}{2}}\tilde{\varepsilon}(s)\|^{2}\right]ds
	+C\tau^{\frac{p}{2}}.\notag
\end{align}
Plugging the estimates \eqref{Thm2_34_1} and \eqref{eq:Jps} into \eqref{Thm2_32_2}, we can
apply Gronwall's inequality to finish the proof.
\end{proof}

We are in a position to measure the error between $\phi(t)$ and $\Phi(t)$ in $L^2(\Omega;H)$.
\begin{proposition}\label{Thm2_p}
Let Assumptions \ref{assum2} and \ref{assum1} hold, and let $\phi^0\in \dot{H}^2(\mathcal{O})$. %{\r what is $\dot H^2$ and $\dot H^{s}$?}. 
Then  there exists a constant $C>0$ such that for  any $t\in[0,T]$,
\begin{equation}
	\label{Thm2_Re}
	\mathbb{E} \left[\|\phi(t)-\Phi(t)\|^2\right]\leq C \tau.
\end{equation}
\end{proposition}

\begin{proof}
We shall split the proof into two steps. 	{In the first step, we show
	that there exists a constant \( C >0 \) such that 
	\begin{align}\label{Thm2_32_1p2_re2} 
		\mathbb{E}\left[\Big|\int_{0}^{{t}}
		\|(-A)^{\frac 12}\tilde{\varepsilon}(s)\|^2{ds}\Big|^2\right]\leq C\tau^2,\quad t\in[0,T].
	\end{align}
	Then based on the estimate on the  $L^4(\Omega;L^2(0,T; H^1(\mathcal{O})))$-norm of $\tilde{\varepsilon}=\phi-\Phi$ in \eqref{Thm2_32_1p2_re2}, we further prove the required estimate \eqref{Thm2_Re} on the $L^\infty(0,T;L^2(\Omega;H))$-norm of $\tilde{\varepsilon}$.}

%{\r logic needs one sentence }
\emph{Step 1.}
%		 In this step, we show that there exists a constant \( C >0 \) such that 
%			\begin{align}\label{Thm2_32_1p2_re2} 
	%				&\mathbb{E}\left[\|(-A)^{-\frac{1}{2}}\tilde{\varepsilon}({t})\|^4\right]+\mathbb{E}\left[\Big|\int_{0}^{{t}}
	%				\|(-A)^{\frac 12}\tilde{\varepsilon}(s)\|^2{ds}\Big|^2\right]\leq C\tau^2,\quad t\in[0,T].
	%		\end{align}
{The integral form of  \eqref{Thm2_32} with $p=2$ reads
	\begin{align*}
		&\|(-A)^{-\frac{1}{2}}\tilde{\varepsilon}({t})\|^2+\int_{0}^{{t}}\|(-A)^{\frac 12}\tilde{\varepsilon}(s)\|^2ds\\
		&\le C\int_{0}^{{t}}\|(-A)^{-\frac{1}{2}}\tilde{\varepsilon}(s)\|^2ds
		+C\int_{0}^{{t}}|\mathbb{I}(s)|^2ds+\int_0^t\mathbb{K}_{2}(s) ds\\
		&\quad+2\int_{0}^t\left\langle\left(g(\phi(s))-g(X(\kappa_N(s)))\right)dW(s),(-A)^{-1}\tilde{\varepsilon}(s)\right\rangle.  
	\end{align*}
	Squaring both sides on the above inequality,} we deduce from H\"older's inequality  that for any $t\in[0,T]$,
\begin{align}\label{Thm2_32_1p2_2} 
	&\left|\int_{0}^{ t }\|(-A)^{\frac 12}\tilde{\varepsilon}(s)\|^2ds\right|^2\leq C\int_{0}^{ t }\|(-A)^{-\frac{1}{2}}\tilde{\varepsilon}(s)\|^4ds+C\int_{0}^{ t }|\mathbb{I}(t)|^4ds\\
	&\quad+4\left|\int_{0}^{ t }{\mathbb{K}}_2(s)ds\right|^2+C\left|\int_{0}^{ t }\left\langle\left(g(\phi(s))-g(X(\kappa_N(s)))\right)dW(s),(-A)^{-1}\tilde{\varepsilon}(s)\right\rangle\right|^2.\notag
\end{align} 
From the BDG inequality, \eqref{eq:Jps}, and Lemma \ref{Lemma4Thm2} with $p=4$, it follows that
\begin{align}\label{Thm2_32_1p2_3} 
	& \mathbb{E}\bigg[\bigg|\int_{0}^{ t }\left\langle\left(g(\phi(s))-g(X(\kappa_N(s)))\right)dW(s),(-A)^{-1}\tilde{\varepsilon}(s)\right\rangle\bigg|^2\bigg]\\
	&\leq  C\mathbb{E}\left[\int_{0}^{ t }\|(-A)^{-\frac{1}{2}}(g(\phi(s))-g(X(\kappa_N(s))))\|^2_{{\mathcal{L}_2^0}}\|(-A)^{-\frac{1}{2}}\tilde{\varepsilon}(s)\|^2 ds\right]\notag\\
	&=C\mathbb{E}\left[\int_{0}^{ t } {\mathbb{K}}_4(s)ds\right]\notag
	\leq C \tau^2.
\end{align}
In view of \eqref{Thm2_34_3_insert} with $\upsilon_1=1/2$, the Cauchy--Schwarz inequality, and  \eqref{eq:Phi-X}, we obtain  that 
\begin{align}\label{Thm2_32_1p2_4}
	\mathbb{E}\left[\bigg|\int_{0}^{ t }{\mathbb{K}}_2(s)ds\bigg|^2\right]&\leq 2 \mathbb{E}\left[\left|\int_{0}^{ t } \frac{1}{8}\|(-A)^{\frac{1}{2}}\tilde{\varepsilon}(t)\|^2
	+C\|(-A)^{-\frac{1}{2}}\tilde{\varepsilon}(t)\|^2ds\right|^2\right]
	\\
	&\quad+{C} \mathbb{E}\left[\left|\int_{0}^{ t }\|\Phi(s)-X(\kappa_N(s))\|^2
	ds\right|^2\right]\notag\\
	&\leq \frac{1}{16} \mathbb{E}\left[\left|\int_{0}^{ t } \|(-A)^{\frac{1}{2}}\tilde{\varepsilon}(t)\|^2ds\right|^2\right]\notag\\&\quad+C\mathbb{E}\left[\int_{0}^{ t }\|(-A)^{-\frac{1}{2}}\tilde{\varepsilon}(t)\|^4ds\right] +C\tau^2.\notag
\end{align}
Taking expectations on both sides of  \eqref{Thm2_32_1p2_2}, and then using
\eqref{Thm2_34_1} with $p=4$,  \eqref{Thm2_32_1p2_3} and \eqref{Thm2_32_1p2_4}, it follows that for any $t\in[0,T]$,
\begin{align}
	\frac34\mathbb{E}\left[\left|\int_{0}^{ t }\|(-A)^{\frac 12}\tilde{\varepsilon}(s)\|^2{ds}\right|^2\right]\leq C\mathbb{E}\int_{0}^{ t }\|(-A)^{-\frac{1}{2}}\tilde{\varepsilon}(s)\|^4ds+C\tau^2\notag,
\end{align}
which, along with Lemma \ref{Lemma4Thm2}, results in
\eqref{Thm2_32_1p2_re2}.

\emph{Step 2.} 
It follows from \eqref{Thm2_1} and the Cauchy--Schwarz inequality that 
\begin{align}\label{L2perror_start}
	\left\|\tilde{\varepsilon}( t )\right\|^2_{L^2(\Omega;H)}
	&\leq 3\left\|\int_{0}^{ t }S(t-s)A\left(F'(\phi(s))-F'(\Phi(s))\right)ds\right\|^2_{L^2(\Omega;H)}\\
	&\quad+3\left\|\int_{0}^{ t }S(t-s)A\left(F'(\Phi(s))-F'(X(\kappa_N(s)))\right)ds\right\|^2_{L^2(\Omega;H)}\notag\\
	&\quad+3\left\|\int_{0}^{ t }S(t-s)\left(g(\phi(s))-g(X(\kappa_N(s))\right)dW(s)\right\|^2_{L^2(\Omega;H)}\notag\\
	&=: \mathcal{Q}_{1,t}+\mathcal{Q}_{2,t}+\mathcal{Q}_{3,t}.\notag
\end{align}
By Assumption \ref{assum2}, for any $\beta_1\in(\frac32,2)$, there exists a constant $C>0$ such that for any $v,w\in \dot{H}^1(\mathcal{O})$,
\begin{align*}
	&\|(-A)^{\frac{1}{2}}(F'(v)-F'(w))\|^2=\|\nabla(F'(v)-F'(w))\|^2\\	&=\int_{\mathcal{O}}|f''(v(x))\nabla v(x)-f''(w(x))\nabla w(x)|^2d x\\
	& \le 2\int_{\mathcal{O}}|f^{\prime\prime}(v(x))-f^{\prime\prime}(w(x))|^2|\nabla v(x)|^2 dx+2\int_{\mathcal{O}}|f^{\prime\prime}(w(x))|^2|\nabla( v-w)(x)|^2 dx\\
	%&\le C\int_{\mathcal{O}}(1+|v(x)|^2+|w(x)|^2)|v(x)-w(x)|^2|\nabla v(x)|^2 dx+C\int_{\mathcal{O}}(1+|v(x)|^4)|\nabla( v-w)(x)|^2 dx\\
	&\le C(1+\|v\|_{L^\infty}^2+\|w\|_{L^\infty}^2)\|v-w\|_{L^6}^2\| \nabla v\|^2_{L^3} +C(1+\|v(x)\|_{L^\infty}^4)\|\nabla( v-w)\|^2 \\
	&\le C(1+\|v\|_{H^{\beta_1}}^4+\|w\|_{H^{\beta_1}}^4)\|(-A)^{\frac12}( v-w)\|^2,
\end{align*}
where the last step is due to the Sobolev embeddings $ H^1(\mathcal{O})\hookrightarrow L^6(\mathcal{O}) $, $H^{{\beta_1}-1}(\mathcal{O})\hookrightarrow H^{1/2}(\mathcal{O})\hookrightarrow L^3(\mathcal{O})$ and $H^{\beta_1}(\mathcal{O})\hookrightarrow L^\infty(\mathcal{O})$ for any $\beta_1\in(\frac32,2)$.
Hence we can apply \eqref{assum1_eq2}, H\"older's inequality, \eqref{eq:Aphit}, Proposition \ref{App_prop}, and \eqref{Thm2_32_1p2_re2} to obtain
\begin{align*}%\label{L2p_error_start_K1}
	\mathcal{Q}_{1,t}&\leq C\mathbb{E}\left[\left(\int_{0}^{ t }\left\|S(t-s)(-A)^{\frac{1}{2}}(-A)^{\frac{1}{2}}\left(F'(\phi(s))-F'(\Phi(s))\right)\right\|ds\right)^2\right]\\
	&\leq C\mathbb{E}\left[\left(\int_{0}^{ t }(t-s)^{-\frac{1}{4}}\|(-A)^{\frac{1}{2}}(F'(\phi(s))-F'(\Phi(s)))\|ds\right)^2\right]\\
	&\leq C\mathbb{E}\left[\left(\int_{0}^{ t }(t-s)^{-\frac{1}{4}}\|(-A)^{\frac{1}{2}}\tilde{\varepsilon}(s)\|\left(1+{\|\phi(s)\|^2_{H^{\beta_1}}}+\|\Phi(s)\|^2_{H^{\beta_1}}\right)ds\right)^2\right]\notag\\
	% &\leq C\mathbb{E}\left[\left(\int_{0}^{ t }\|(-A)^{\frac{1}{2}}\tilde{\varepsilon}(s)\|^2ds\right)\left(\int_{0}^{ t }(t-s)^{-\frac{1}{2}}\left(1+\|\phi(s)\|^4_{H^{\beta_1}}
	%          +\|\Phi(s)\|^4_{H^{\beta_1}}\right)ds\right)\right]\\
	&\leq C\left(\mathbb{E}\left[\left(\int_{0}^{ t }\|(-A)^{\frac{1}{2}}\tilde{\varepsilon}(s)\|^2ds\right)^2\right]\right)^{\frac 12}\\
	&\quad\times\left(\mathbb{E}\left[\left(\int_{0}^{ t }(t-s)^{-\frac{1}{2}}\left(1+\|\phi(s)\|^4_{H^{\beta_1}}
	+\|\Phi(s)\|^4_{H^{\beta_1}}\right)ds\right)^2\right]\right)^{\frac 12}\\
	&\leq C\tau.
\end{align*}
The first inequality of \eqref{eq:F'-F'} and the Minkowski inequality give
\begin{align}\label{L2_perror_start_K2}
	\mathcal{Q}_{2,t}&\leq C\mathbb{E}\bigg[\bigg(\int_{0}^{ t }(t-s)^{-\frac{1}{2}}\left\|F'(\Phi(s)-F'(X(\kappa_N(s)))\right\|ds\bigg)^2\bigg]\\\notag
	&\leq C\mathbb{E}\bigg[\bigg(\int_{0}^{ t }(t-s)^{-\frac{1}{2}}\mathbb{I}(s)ds\bigg)^2\bigg]\le C\left(\int_{0}^{ t }(t-s)^{-\frac{1}{2}}\|\mathbb{I}(s)\|_{L^2(\Omega)}ds\right)^2,
\end{align}
where $\mathbb{I}(\cdot)$ is defined in \eqref{eq:I1t}. Furthermore, plugging   \eqref{Thm2_34_1} with $p=2$ into \eqref{L2_perror_start_K2} yields $\mathcal{Q}_{2,t}\le C\tau$ for any $t\in[0,T]$.
Applying the BDG inequality, the contractivity of $S(\cdot)$, \eqref{assume_Lip}, \eqref{eq:Phi-X}, and \eqref{Thm2_32_1p2_re2}, we arrive at 
\begin{align*}%\label{L2p_error_start_K3}
	\mathcal{Q}_{3,t}
	&\leq 2\mathbb{E}\left[\int_{0}^{ t }\left\|S(t-s)\left(g(\phi(s))-g(X(\kappa_N(s))\right)\right\|_{\mathcal{L}_2^0}^2ds\right]\\
	&\leq C \int_{0}^{ t }\mathbb{E}\left[\left\|\tilde{\varepsilon}(s)\right\|^2\right]+\mathbb{E}\left[\|\Phi(s)-X(\kappa_N(s))\|^2\right]ds \\
	&\leq C \int_{0}^{ t }\mathbb{E}\left[\|(-A)^{\frac{1}{2}}\tilde{\varepsilon}(s)\|^2\right]ds
	+C \int_{0}^{ t }\mathbb{E}\left[\|(-A)^{-\frac{1}{2}}\tilde{\varepsilon}(s)\|^2\right]ds
	+C\tau\leq C\tau.
\end{align*}
Finally, plugging the estimates
of $\mathcal{Q}_{i,t}$, $i=1,2,3$, into \eqref{L2perror_start}  
yields the desired result.
\end{proof}
%{\r {do you want to modify your proof?} We remark that the proof of Proposition \ref{Thm2_p} can be adapted, with only minor modifications, to derive the 
%$p$th-moment error estimate between $\phi(t)$ and $\Phi(t)$ for general $p\ge2$.} {\color{blue}I don't want to modify this proof now. If the reviewer ask this question, then I will try to do this. May be we can remove this remark now.} 

Finally, We are now ready to complete the proof of Theorem \ref{Thm3_main}.

\textbf{Proof of Theorem \ref{Thm3_main}}. Combining Propositions \ref{Thm1} and \ref{Thm2_p}, we can conclude the proof from \eqref{eq:phiX}. \hfill$\square$

\section{Proof of Theorem \ref{Sec:energy:pro}}\label{Sec:energy}

This section is devoted to proving Theorem \ref{Sec:energy:pro} on the averaged evolution law of the modified SAV energy \eqref{eq:Emodify}.

\textbf{Proof of Theorem \ref{Sec:energy:pro}.}
Recall that by \eqref{dis_energy_diss}, for any $n\in\{0,1,\cdots,N-1\}$,
\begin{align}\label{dis_energy_diss-new}
&	E_{\mathrm{mod}}^{n+1}-E_{\mathrm{mod}}^{n}=\frac{1}{2}\left[\|\nabla   X^{n+1}\|^{2}-\|\nabla X^{n}\|^{2}\right]
+|r^{n+1}|^{2}-|r^{n}|^{2}\\\notag
&= -\frac{1}{2}\|(I-S^2(\tau))^{\frac{1}{2}}(-A)^{-\frac 12}\tilde{\mu}^n\|^2
+\frac{1}{2}\|(-A)^{\frac12}\left(g( X^{n})\delta  W^{n}\right)\|^2\\\notag
&\quad-|r^{n+1}-r^n|^{2}+\langle -AX^{n}+\tilde{f}^{n},g( X^{n})\delta  W^{n}\rangle.
\end{align}
On account of \eqref{fn} and \eqref{chi}, it holds that
\begin{align}\label{eq:J2-5}
&  \langle -AX^{n}+\tilde{f}^{n},g( X^{n})\delta  W^{n}\rangle
=\Big\langle-A X^{n}+\frac{r^{n}}{\sqrt{E_{\textup{p}}(X^{n})}}F'( X^{n}),
g( X^{n})\delta  W^{n}\Big\rangle\\&\quad+\frac12\left\langle F''( X^{n})g( X^{n})\delta  W^{n}
,g( X^{n})\delta  W^{n}\right\rangle\notag\\
&\quad+\frac{r^{n+1}-r^n}{\sqrt{E_{\textup{p}}(X^{n})}}\left\langle F'( X^{n})
,g( X^{n})\delta  W^{n}\right\rangle
-\frac{r^{n+1}}{4E_{\textup{p}}(X^{n})^{3/2}}\left\langle F'( X^{n}), g( X^{n})\delta  W^{n}\right\rangle^2\notag\\
%	&\quad-\frac{r^{n+1}-r^n}{4E_{\textup{p}}(X^{n})^{3/2}}\left\langle F'( X^{n}), g( X^{n})\delta  W^{n}\right\rangle^2
&\quad+\frac12\bigg(\frac{r^{n+1}}{\sqrt{E_{\textup{p}}(X^{n})}}-1\bigg)\left\langle F''( X^{n})g( X^{n})\delta  W^{n},
g( X^{n})\delta  W^{n}\right\rangle.\notag
\end{align}
Substituting \eqref{eq:J2-5} into \eqref{dis_energy_diss-new} yields
\begin{align}\label{dis_energy_diss-new-en}
%		&\frac{1}{2}\left[\left\|\nabla   X^{n+1}\right\|^{2}-\left\|\nabla X^{n}\right\|^{2}\right]
%		+|r^{n+1}|^{2}-|r^{n}|^{2}\\\notag
E_{\mathrm{mod}}^{n+1}-E_{\mathrm{mod}}^{n}&=-\frac{1}{2}\|(I-S^2(\tau))^{\frac{1}{2}}(-A)^{-\frac 12}\tilde{\mu}^n\|^2\\\notag
&\quad+\Big\langle-A X^{n}+\frac{r^{n}}{\sqrt{E_{\textup{p}}(X^{n})}}F'( X^{n}),
g( X^{n})\delta  W^{n}\Big\rangle\\\notag
&\quad+\frac12\left\langle (-A+F''( X^{n}))g( X^{n})\delta  W^{n}
,g( X^{n})\delta  W^{n}\right\rangle+\mathcal{J}^n,		
\end{align}
%   {\r why there is $A$ in the third term?}
where the remainder term 
\begin{align}\label{eq:Rn}
\mathcal{J}^n&=	-|r^{n+1} - r^n|^2+\frac{r^{n+1}-r^n}{\sqrt{E_{\textup{p}}(X^{n})}}\left\langle F'( X^{n})
,g( X^{n})\delta  W^{n}\right\rangle\\\notag
&\quad
-\frac{r^{n+1}}{4E_{\textup{p}}(X^{n})^{3/2}}\left\langle F'( X^{n}), g( X^{n})\delta  W^{n}\right\rangle^2,\\\notag
&\quad +\frac12\bigg(\frac{r^{n+1}}{\sqrt{E_{\textup{p}}(X^{n})}}-1\bigg)\left\langle F''( X^{n})g( X^{n})\delta  W^{n},
g( X^{n})\delta  W^{n}\right\rangle.
\end{align}	
Summing over $n=0$ through $n=m-1$ on both sides of \eqref{dis_energy_diss-new-en} and then taking expectations on both sides of the resulting equation, we obtain the required equation \eqref{Sec:energy:pro} with $\mathcal{R}_m^\tau=\sum_{n=0}^{m-1}\mathbb{E}[\mathcal{J}^n]$.

It remains to prove that $\lim_{\tau\to\infty}\mathcal{R}_m^\tau=0$, which needs to be estimate carefully. 
At first glance, each of the first three terms on the right-hand side of \eqref{eq:Rn} appears to have only strong convergence order $1$, which does not vanish as $\tau \to 0$  after summing over $n$.
Hence, a more refined decomposition of the remainder term $\mathcal{J}^n$ is necessary. In fact, the first three terms on the right-hand side of \eqref{eq:Rn} contain mutually canceling contributions. To illustrate this, we notice that 
$-|r^{n+1} - r^n|^2=J_{1}^n+J_{2}^n+J_{3}^n,$ where
\begin{align}\notag
J_{1}^n&: =-\frac{1}{4E_{\textup{p}}(X^n) }\left| \left\langle F'(X^n),
g(X^n)\delta W^n\right\rangle\right|^2, \\\notag
J_{2}^n	&:=\frac{1}{4E_{\textup{p}}(X^n) }\left(\left|  \left\langle F'(X^n),g(X^n)\delta W^n\right\rangle\right|^2 -\left| \left\langle F'(X^n),X^{n+1} - X^n\right\rangle\right|^2\right),\\\label{eq:termJ3}
J_{3}^n&:=-|r^{n+1} - r^n|^2+\frac{1}{4E_{\textup{p}}(X^n) }\left| \left\langle F'(X^n),X^{n+1} - X^n\right\rangle\right|^2.
\end{align}
In addition, using \eqref{scheme_r_sto-intro} leads to 
\begin{align*}%\label{dis_energy_diss2_J2}
&	\frac{r^{n+1} - r^n}{\sqrt{E_{\textup{p}}(X^n)}} \left\langle F'(X^n), g(X^n) \, \delta W^n \right\rangle =  \frac{1}{2E_{\textup{p}}(X^n)}  \left\langle F'(X^n), g(X^n) \, \delta W^n \right\rangle^2\\
&\quad+\frac{1}{2E_{\textup{p}}(X^n) } \left\langle F'(X^n), g(X^n) \, \delta W^n \right\rangle \left\langle F'(X^n), X^{n+1} - X^n- g(X^n)\delta W^n\right\rangle\\
&\quad -\frac{1}{8E_{\textup{p}}(X^n) ^2} \left\langle F'(X^n), g(X^n) \, \delta W^n \right\rangle^2 \left\langle F'(X^n),X^{n+1} - X^n\right\rangle\\
&\quad +   \frac{1}{4E_{\textup{p}}(X^n) } \left\langle F'(X^n), g(X^n) \, \delta W^n \right\rangle\left\langle F''(X^{n})(X^{n+1}-X^{n}),g(X^{n})\delta W^{n}\right\rangle\notag\\
&=:J_{4}^n+J_{5}^n+J_{6}^n+J_{7}^n,
\end{align*}
and 
\begin{align*}
%	\label{dis_energy_diss2_J3}
&-\frac{r^{n+1}}{4E_{\textup{p}}(X^{n})^{3/2}}\left\langle F'( X^{n}), g( X^{n})\delta  W^{n}\right\rangle^2=-\frac{1}{4E_{\textup{p}}(X^{n})}\left\langle F'( X^{n}), g( X^{n})\delta  W^{n}\right\rangle^2\\
&\quad-\frac14\bigg(\frac{r^{n+1}}{\sqrt{E_{\textup{p}}(X^{n})}}-1\bigg)\frac{1}{E_{\textup{p}}(X^{n})} \left\langle F'( X^{n}), g( X^{n})\delta  W^{n}\right\rangle^2=:J_{8}^n+J_{9}^n.
\end{align*}
Since $J_{1}^n+J_4^n+J_8^n=0$, one has 
\begin{equation}\label{eq:Jndecom}
\mathcal{J}^n=\sum_{i\in\mathcal{I}}J_i^n
\end{equation}
where $\mathcal{I}:=\{2,3,5,6,7,9,10\}$, and
\begin{align*}
J_{10}^n:=\frac12\bigg(\frac{r^{n+1}}{\sqrt{E_{\textup{p}}(X^{n})}}-1\bigg)\left\langle F''( X^{n})g( X^{n})\delta  W^{n},
g( X^{n})\delta  W^{n}\right\rangle
\end{align*}
is the last term on the right hand side of \eqref{eq:Rn}.

%
%	
%	\begin{lemma}\label{lem:Remainder}
%		Let Assumptions \ref{assum2} and \ref{assum1} hold, and let $\phi^0\in \dot{H}^{2}(\mathcal{O})$.  Then $\lim_{\tau\to 0}\sum_{n=0}^m\mathbb{E}[\mathcal{J}^n]=0$ for all $m\in\{0,1,\cdots,N-1\}$.
%	\end{lemma}
%	\begin{proof}

%	\textbf{Estimates of $J_2^n$ and $J_5^n$.}
Using H\"older's inequality, \eqref{eq:FXn}, \eqref{eq:BDG}, and Lemma~\ref{lem:time}, we deduce that 
\begin{align}\label{Term_J12}
	\mathbb{E}[|J_{2}^n|]
	% &\le C\mathbb{E}\left[\|F'(X^n)\|^2\|X^{n+1} - X^n-g(X^n)\delta W^n\|\|X^{n+1} - X^n+g(X^n)\delta W^n\|\right]\\
	&\le C
	\left(\|X^{n+1} - X^n\|_{L^4(\Omega;H)}+\|g(X^n)\delta W^n\|_{L^4(\Omega;H)}\right)\\
	&\quad \times \|X^{n+1} - X^n-g(X^n)\delta W^n\|_{L^2(\Omega;H)}\notag\\
	&\le C\tau^{\frac12}\|X^{n+1} - X^n-g(X^n)\delta W^n\|_{L^2(\Omega;H)}\notag
\end{align}
for any $n\in\{0,1,\cdots,N-1\}$.
Similarly, we also have 
\begin{align}\label{Term_J13}
	\mathbb{E}[|J_{5}^n|]
	\le C\tau^{\frac12}\|X^{n+1} - X^n-g(X^n)\delta W^n\|_{L^2(\Omega;H)}.
\end{align}

Referring to \eqref{schme_phi-intro}, one infers that
\begin{align*}
	X^{n+1}-X^n-g(X^n)\delta W^n&=(S(\tau)-I)X^n
	+ \left(I - S(\tau)\right)A^{-1}\tilde{f}^n\\
	&\quad+ (S(\tau)-I)g(X^n)\delta W^n.
\end{align*}
Recall that by \eqref{lemma6_eq3_R1}, for any $\alpha\in(0,\frac12)$,
\begin{equation}\label{eq:initial}
	\|(S(\tau)-I)X^n \|_{L^2(\Omega;H)}
	\le  C(\alpha)(1+t_n^{-\frac12\alpha})\tau^{\frac12(1+\alpha)},\qquad n=1,2,\cdots,N-1
\end{equation}
As for $n=0$, it follows from \eqref{assum1_eq2} and $\phi^0\in \dot{H}^2(\mathcal{O})$ that 		$
\|(S(\tau)-I)X^0 \|_{L^2(\Omega;H)}
\le  C\tau^{\frac12}.
$	
%		 {\r check it what upper bounds we needed here}.                                                                                                                              
Due to  \eqref{assum1_eq2} and \eqref{eq:fnH1},
\begin{equation}\label{Term_J12_T2}
	\|(I - S(\tau))A^{-1}\tilde{f}^n\|_{L^2(\Omega;H)} =\|(I - S(\tau))(-A)^{-\frac32}(-A)^{\frac 12}\tilde{f}^n\|_{L^2(\Omega;H)}
	%\le C\tau^{\frac34}\|(-A)^{\frac 12}\tilde{f}^n\|_{L^2(\Omega;H)} 
	\le C\tau^{\frac34}.
\end{equation}
The BDG inequality and \eqref{lema1_eq_R6-new}, together with \eqref{assum1_eq2} and Corollary \ref{lemma_stability}, imply that
\begin{align}\label{Term_J12_T3}
	\|(S(\tau)-I)g(X^n)\delta W^n\|_{L^2(\Omega;H)}&=\|(S(\tau)-I)(-A)^{-\frac 12}(-A)^{\frac 12}g(X^n)\delta W^n\|_{L^2(\Omega;H)}\notag\\
	&\le C \tau^{\frac 34}\|(-A)^{\frac 12}g(X^n)\|_{ L^2(\Omega;\mathcal{L}_2^0)}
	\le C \tau^{\frac 34}.
\end{align}
% {\r how you get the following, where is $\tau^{\frac 34}$}
Combining \eqref{eq:initial}, \eqref{Term_J12_T2}, and \eqref{Term_J12_T3},  it follows that {for any $\alpha\in(0,\frac12)$},
{
	\begin{align*}%\label{Term_J12_key}
		&\sum_{n=0}^{m-1}\|X^{n+1} - X^n-g(X^n)\delta W^n\|_{L^2(\Omega;H)}\\
		&\le C\tau^{-\frac14}+ C\tau^{\frac12}+C(\alpha)\sum_{n=1}^{m-1} (1+t_n^{-\frac12\alpha})\tau^{\frac12(1+\alpha)}
		\le C(\alpha)\tau^{-\frac12+\frac 12\alpha}. 
	\end{align*}
	%   {\r how $\frac 14$ comes from? Need a simple explanation.}
	Hence, it follows from \eqref{Term_J12} and \eqref{Term_J13} that for any $\alpha\in(0,\frac 12)$,
}
\begin{align*}%\label{Term_J12_final}
	\sum_{n=0}^{m-1}	\mathbb{E}\left[|J_2^n|+|J_{5}^n|\right]
	\le C\tau^{\frac12}\sum_{n=0}^{m-1}\|X^{n+1} - X^n-g(X^n)\delta W^n\|_{L^2(\Omega;H)}
	\le C(\alpha)\tau^{\frac12\alpha}.
\end{align*}
For the other terms, we claim that for any $i\in\{3,6,7,9,10\}$ (see Appendix \ref{App-J3-10} for its proof),
%{\r \color{red} where is the proof for this claim?}
\begin{equation}\label{eq:J3-10}
	\mathbb{E}[|J_{i}^n|]\le C\tau^{\frac32},\qquad n\in\{0,1,\cdots,N-1\}.
\end{equation}
Finally, we can conclude from \eqref{eq:Jndecom} that $\alpha\in(0,\frac 12)$,
\begin{align*}
	|\mathcal{R}_m^\tau|=|\sum_{n=0}^{m-1}\mathbb{E}[\mathcal{J}^n]|
	\le \sum_{i\in \mathcal{I}}\sum_{n=0}^{m-1}	\mathbb{E}\left[|J_i^n|\right]\le C(\alpha)\tau^{\frac12\alpha}.
\end{align*}
The proof is completed.
\hfill$\square$

\section{Numerical experiments}\label{section5}
In this section, we present a series of numerical experiments to illustrate the performance of the exponential Euler SSAV scheme \eqref{schme_phi-intro}.  In the implementation, we take $\mathcal{O}=(0,1)^2$ to be the unit square and approximate the driving process $W$ by its finite-rank Karhunen--Lo\`{e}ve truncation $W^J$, namely,
\[
W^J(t,x)
=\sum_{j_1=1}^{J}\sum_{j_2=1}^{J}
\sqrt{q_{j_1,j_2}}\; e_{j_1,j_2}(x)\,\beta_{j_1,j_2}(t),
\qquad x=(x_1,x_2)\in\mathcal O
\]
with $J=4$.
Here, $\{\beta_{j_1,j_2}(t)\}$ are independent real-valued Brownian motions, the spatial
modes are
\[
e_{j_1,j_2}(x)=2\sin\!\big(j_1\pi x_1\big)\,\sin\!\big(j_2\pi x_2\big),
\]
and $q_{j_1,j_2}=(j_1^2+j_2^2)^{-1}$. Note that for the regular domain $(0,1)^d$, the eigenvalues and eigenfunctions of the Dirichlet Laplacian are explicitly known.
This enables efficient spatial discretization using the spectral Galerkin method. In the following, spatial discretization is performed in the finite-dimensional subspace (see, e.g., \cite{STW11})
\[
H_M := \operatorname{span}\{ e_{j_1,j_2} : 1 \le j_1, j_2 \le M \}.
\]
For more general, possibly irregular, spatial domains, the spectral Galerkin method becomes less practical. In such cases, alternative spatial discretization methods, such as the finite difference or finite element method, may be employed. 
Without otherwise specified, we set $\Theta= 1$ in \eqref{def:ssav} and the spatial discretization parameter $M=128$. 
% The associated spatial semi-discretization reads
% \begin{align}
%     d X_M=A_M (-\epsilon A_M X_M+\frac{1}{\epsilon} P_M F'(X_M))dt+\epsilon^\gamma P_M(g(X_M) dW^J(t)),
% \end{align}
% where $P_M$ is the projection from $H$ to the finite-dimensional subspace 
% \[
% H_M := \operatorname{span}\{ e_{j_1,j_2} : 1 \le j_1, j_2 \le M \}.
% \]
% The numerical solutions are generated by the proposed scheme \eqref{schme_phi-intro} for the time discretization, combined with a spectral Galerkin discretization in space (see, e.g., \cite{}).
%{\r where is the constant in the model} 
%We employ a spectral Galerkin discretization with spatial mesh size \(h = 1/128\). {\r why spectral Galerkin has mesh size?} 
%Unless otherwise specified, we set $M=128$.
The expectation of a random variable $Y$ is always approximated by the Monte Carlo
ensemble mean, i.e., 
$\mathbb{E}[Y]\approx\frac{1}{P}\sum_{i=1}^{P}Y(\omega_i),$
where \(\{Y(\omega_i)\}_{i=1}^P\) are $P$ samples of $Y$. % We set the constant . 
%{\r better to use function instead of $\cdot$.}

\subsection{Averaged energy evolution}\label{sub:freeepsilon}
%\subsection{Numerical validation for the stochastic Cahn--Hilliard}
%We take the initial condition for Examples~\ref{Ex1} %and~\ref{Ex2} as follows:
% \[
% \phi^0(x)
% = 0.6\,\sin(\pi x_1)\sin(\pi x_2)
%   + 0.4\,\sin(2\pi x_1)\sin(3\pi x_2), \qquad x=(x_1,x_2)\in\mathcal O.
% \]
%\begin{example}\label{Ex1}
%begin{example}\label{Ex2}

In this subsection, we verify the averaged energy evolution law of the proposed scheme \eqref{schme_phi-intro}, by studying the following stochastic Cahn--Hilliard equation 
\begin{equation}\label{ex:eps=1}
\begin{split}
	d\phi(t) 
	&= - A^2 \phi(t)  +A\left(\phi(t)^3-\phi(t)\right) dt 
	+  2.5{(1 + \phi(t)^2)^{-\frac12}}\,dW(t),\\
	\phi^0(x)
	&= 0.6\,\sin(\pi x_1)\sin(\pi x_2)
	+ 0.4\,\sin(2\pi x_1)\sin(3\pi x_2).
\end{split}
\end{equation}
in $(0,1)^2\times(0,T]$. 
For equation \eqref{ex:eps=1},
Figure~\ref{Energy2} compares the averaged energy evolution produced by the exponential Euler SSAV scheme \eqref{schme_phi-intro} with that of the standard SAV scheme (i.e., \eqref{schme_phi-intro} with $\tilde{f}^n$ replaced by $\tilde{f}^n_{\mathrm{d}}$ in \eqref{fnd}). 
The expectation is approximated using $P=1000$ Monte Carlo samples, and the time step size is chosen as $\tau=5\times10^{-4}$. 
Since the exact solution of the stochastic Cahn--Hilliard equation \eqref{model} is not available, the reference averaged energy is computed using the fully implicit Euler method in time (see, e.g., \cite{QW20}) with the reference step size $\tau=5\times10^{-5}$.
\begin{figure}%[h]
\centering
% \includegraphics[width=10cm]{fig/DB/energy0103.eps}
\includegraphics[width=10cm]{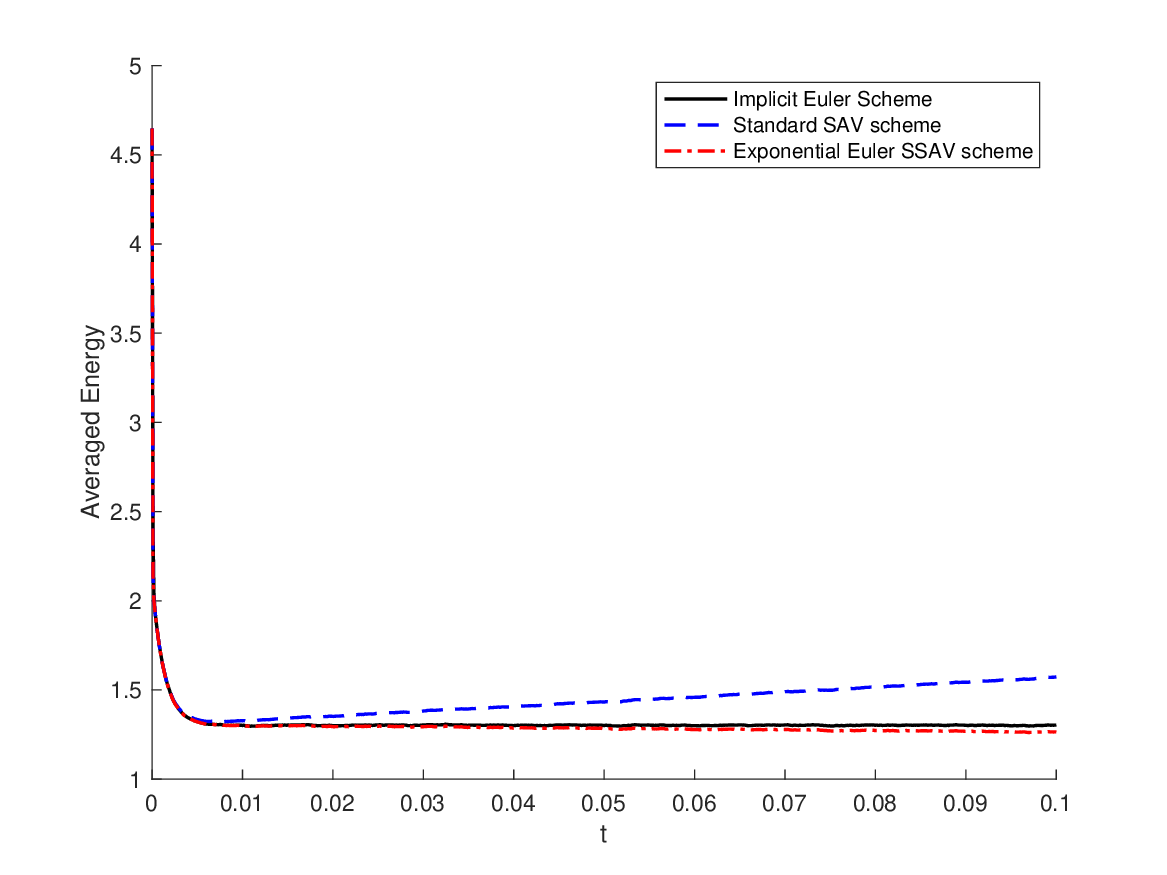}
\vspace{-0.3cm}
\caption{Comparison of averaged energies of the exponential Euler SSAV scheme and standard SAV scheme for equation \eqref{ex:eps=1}.}\label{Energy2} 
%Implicit Euler scheme {\r replace the implcit euler by referenced solu? I think that the implicit euler is using finer stepsize compared with others?}, standard SAV scheme and modified SAV scheme.  }
\end{figure}

\begin{figure}%[htbp]
\centering
%?????? ??? ??????%
\includegraphics[width=0.49\textwidth]{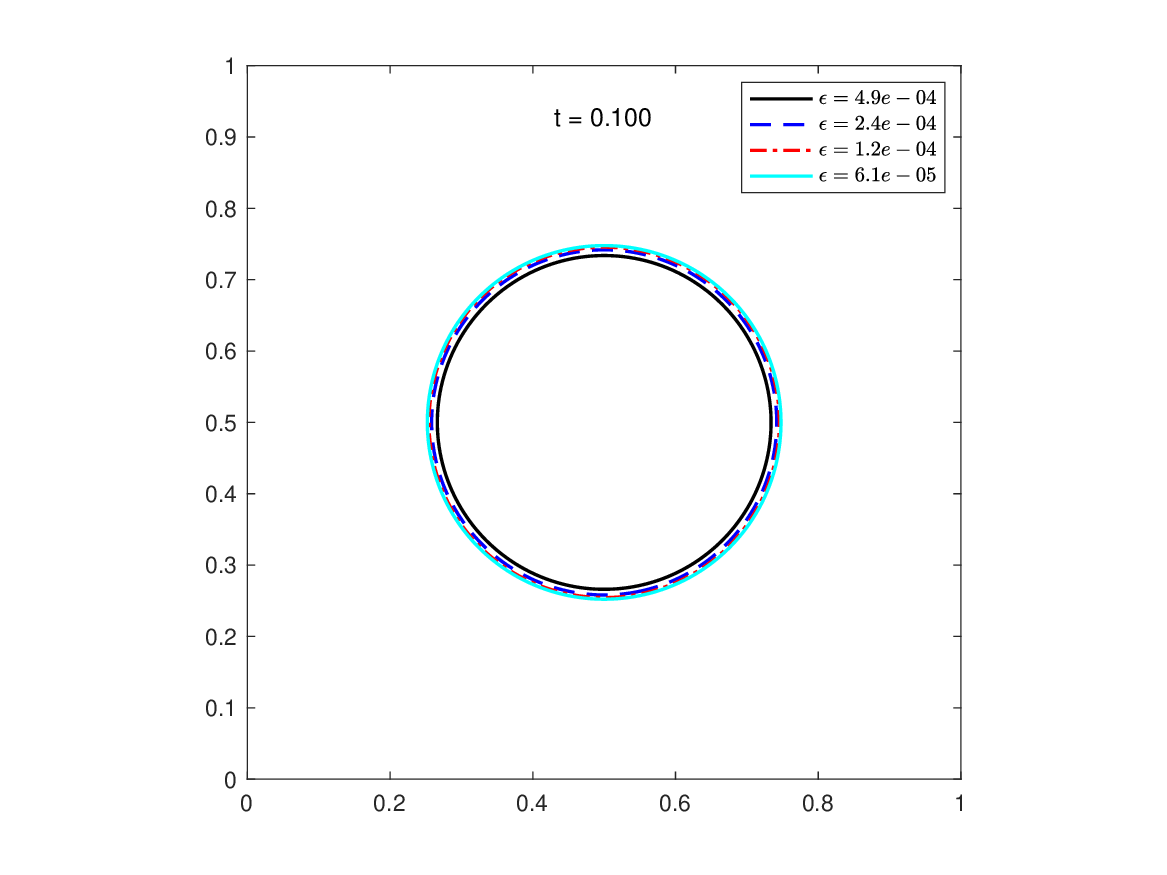}%
% \hspace{0.5pt}%
\includegraphics[width=0.49\textwidth]
{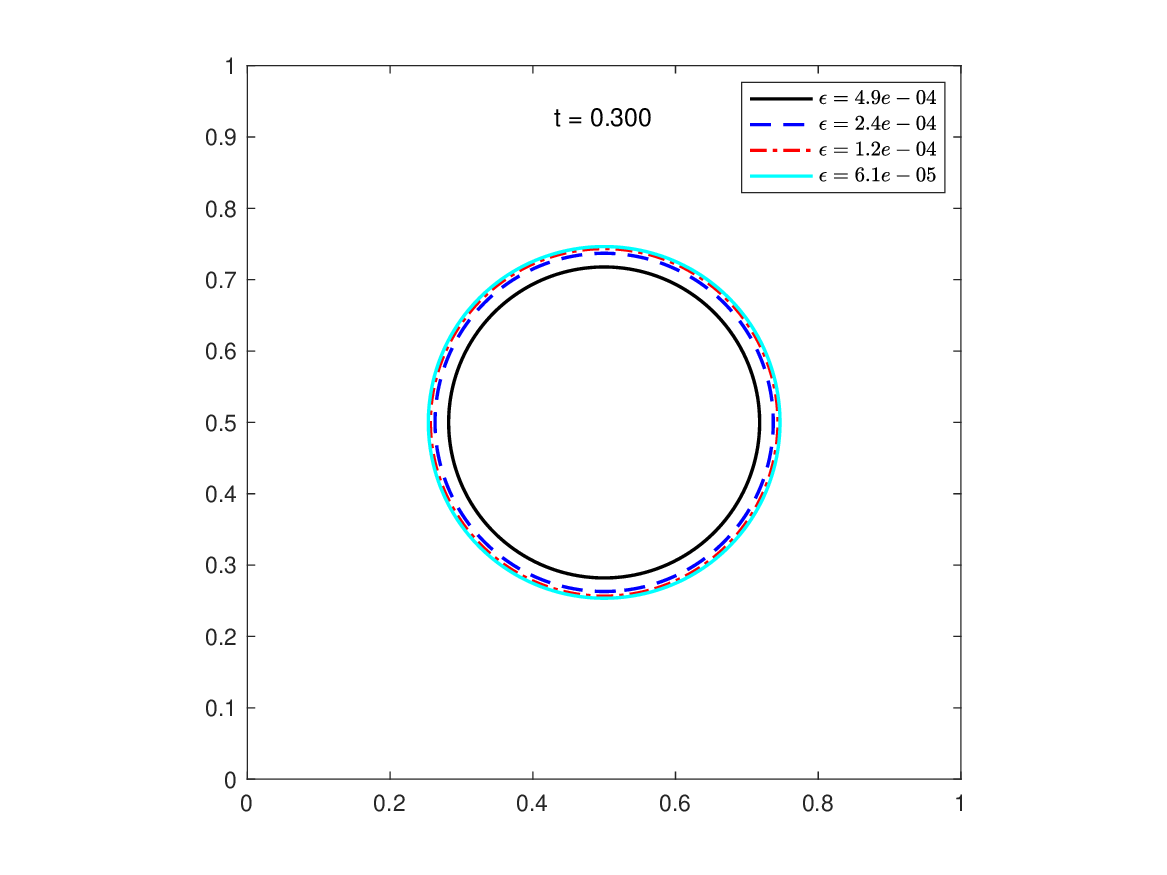}

% \vspace{1ex}  % ??????????

%?????? ??? ??????%
\includegraphics[width=0.49\textwidth]{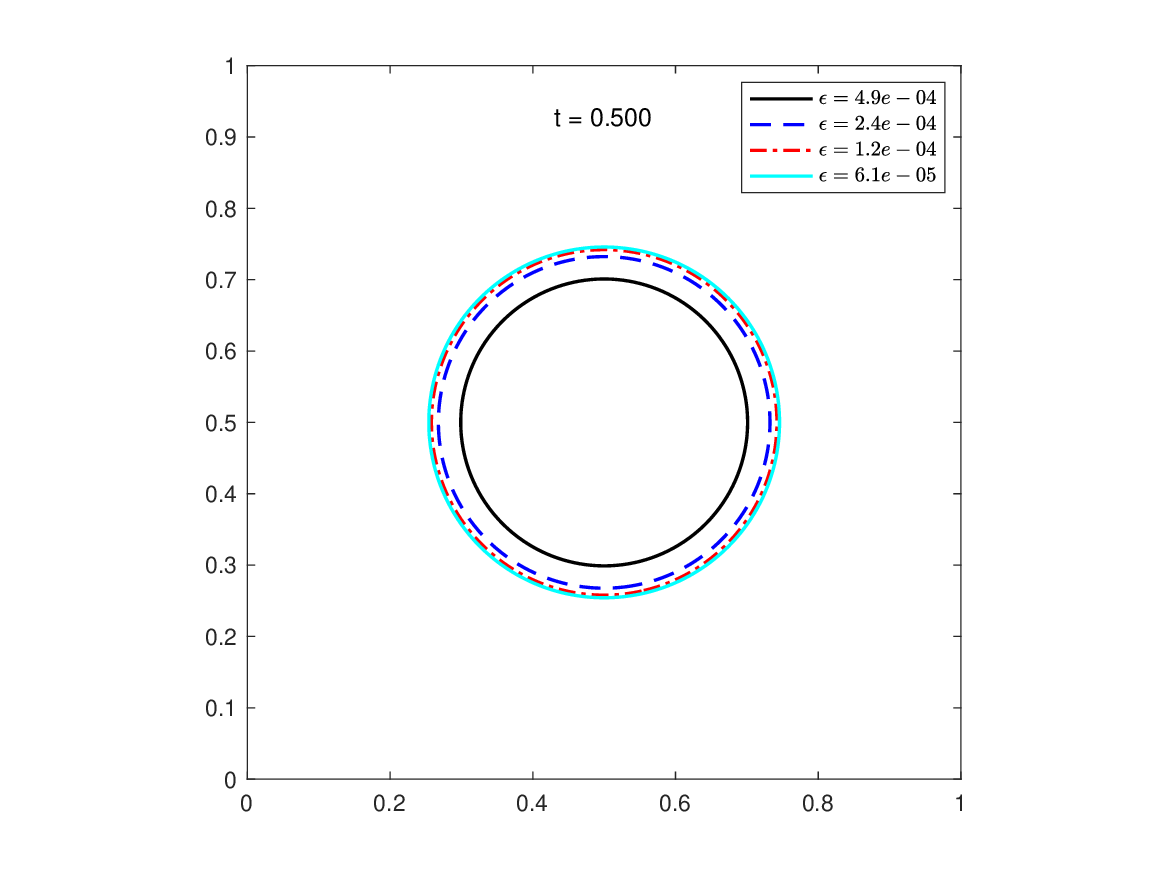}%
% \hspace{1pt}%
\includegraphics[width=0.49\textwidth]
{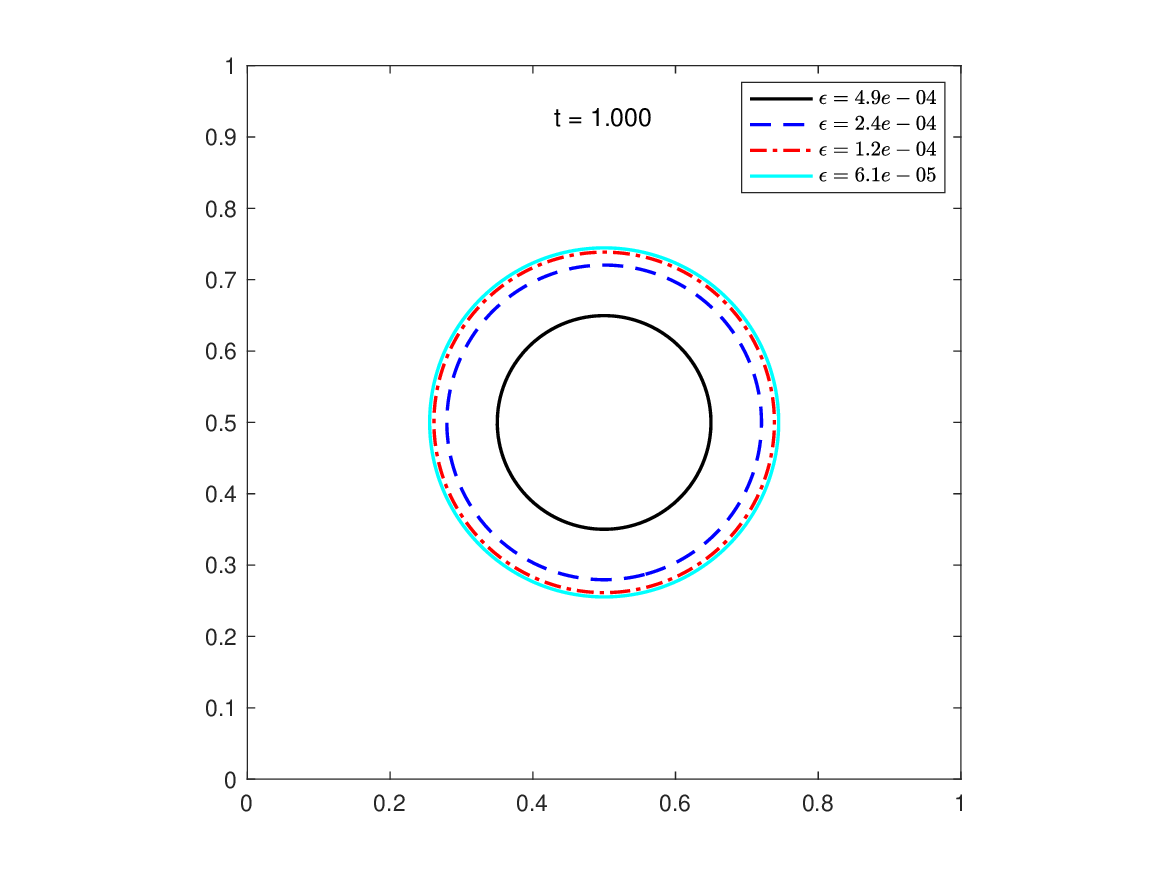}

\caption{Individual snapshots of the zero--level set of the solution at several time points for equation \eqref{model_epsilon_gamma} with $\gamma=1$.}
\label{fig:EX4_gamma1}
\end{figure}
As shown in Figure~\ref{Energy2}, the averaged modified SAV energy associated with the proposed scheme \eqref{schme_phi-intro} remains in close agreement with the reference averaged energy over the entire time interval, in consistent with Theorem~\ref{Sec:energy:pro}. 
In contrast, the averaged energy produced by the standard SAV scheme exhibits a persistent upward drift and appears to converge to a different energy level. From the perspective of the energy evolution law, Figure~\ref{Energy2} indicates that the standard SAV method is no longer suitable for the stochastic Cahn--Hilliard equation, and therefore it is necessary to consider its stochastic modifications, such as the exponential Euler SSAV scheme \eqref{schme_phi-intro}; see also Remark~\ref{rem-rt}.

\subsection{Sharp-interface dynamics}

%\begin{example}\label{ex:model_epsilon_gamma}
In this subsection, we consider the stochastic Cahn--Hilliard equation near the sharp interface limit
\begin{equation}
\begin{split}\label{model_epsilon_gamma}
	d\phi(t) 
	&= A\left(-\epsilon A \phi(t)  +\frac{1}{\epsilon}(\phi(t)^3-\phi(t))\right) dt 
	+ \epsilon^{\gamma} \frac{10}{\sqrt{1 + \phi(t)^2}}\,dW(t),\\
	\phi^0(x)& = \tanh\!\left(\frac{R - 0.25}{\sqrt{2}\,\epsilon}\right),~\text{with}~ R = \sqrt{(x_1 - 0.5)^2 + (x_2 - 0.5)^2}
\end{split}
\end{equation}
in $(0,1)^2\times(0,T]$. Introducing the interfacial width $\epsilon$ and the noise-scaling exponent $\gamma$ in \eqref{model_epsilon_gamma} is motivated by the asymptotic study in \cite{Antonopoulou2021numath} and the numerical analysis in \cite{CW24}.
%\end{example}
%\begin{example}\label{Ex3}

%In the present theoretical analysis, the parameters describing have not been explicitly involved. 

%{\r Cui, J.; Wang, F.-Y. Improving Numerical Error Bounds Near Sharp Interface Limit for Stochastic Reaction-Diffusion Equations arXiv:2412.12604} 

% {\r need reorganize this subsection}
% In contrast to the space--time white noise {\r are you sure that they consider the space-time white noise? they need the existence of strong solution} considered in \cite{Antonopoulou2021numath}, the present work focuses on a trace-class noise, which allows us to explore the effect of a small parameter $\epsilon$ of the following equation
% \begin{align}\label{model_epsilon_gamma}
%  d\phi 
%    &= A\left(-\epsilon A \phi  +\frac{1}{\epsilon}{F'}(\phi)\right) dt 
%    + \epsilon^{\gamma} g(\phi)\,dW(t),
%    && \text{in } \mathcal{O}\times(0,T].
% \end{align}
% Here, $\varepsilon>0$ controls the interfacial thickness, while the parameter $\gamma\ge0$ determines the scaling of the stochastic perturbation with respect to $\varepsilon$. 

% {\r We observe from cite figures that } 
% Different choices of $\gamma$ lead to distinct asymptotic regimes: for $\gamma=1$, the stochastic forcing vanishes in the sharp-interface limit, whereas for $\gamma=0$, the noise remains of order one and may give rise to stochastic effects at the interface.

\subsubsection{Evolution of interface}
In this part, we investigate the influence of a small interfacial
parameter~$\epsilon>0$ and the noise intensity on the evolution of the interface in the stochastic
Cahn--Hilliard equation. 
In the numerical test, we choose different values of the interfacial width parameter 
$\epsilon = 2^{-i}/64$, $i = 5,6,7,8$, and run the simulation up to the final time $T = 1$. 
We first investigate the effect of the noise intensity on the geometric evolution in the sharp-interface limit. Figure~\ref{fig:EX4_gamma1} (resp.~Figure~\ref{fig:EX4_gamma0})
plots an individual realization of the zero-level set of the numerical solution for equation \eqref{model_epsilon_gamma} at different times $t=0.1,0.3,0.5,1$ with $\gamma=1$ (resp.~$\gamma=0$). 

%Figures \ref{fig:EX4_gamma1}--\ref{fig:EX4_gamma0_mean} illustrate the evolution of the zero-level set at different times for various values of $\epsilon$ with $\gamma = 0$ and $\gamma = 1$, respectively.

\begin{figure}%[htbp]
\centering
%?????? ??? ??????%
\includegraphics[width=0.49\textwidth]{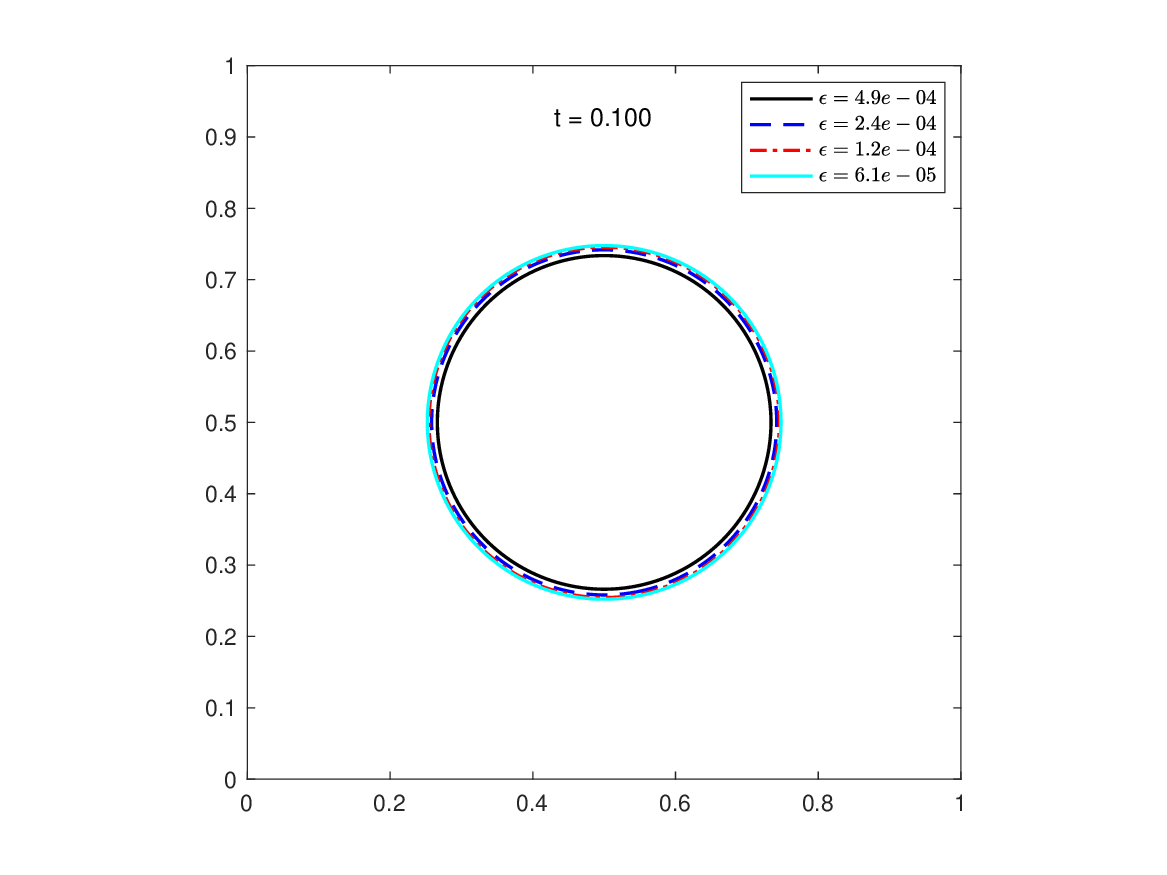}%
%\hspace{0.5pt}%
\includegraphics[width=0.49\textwidth]
{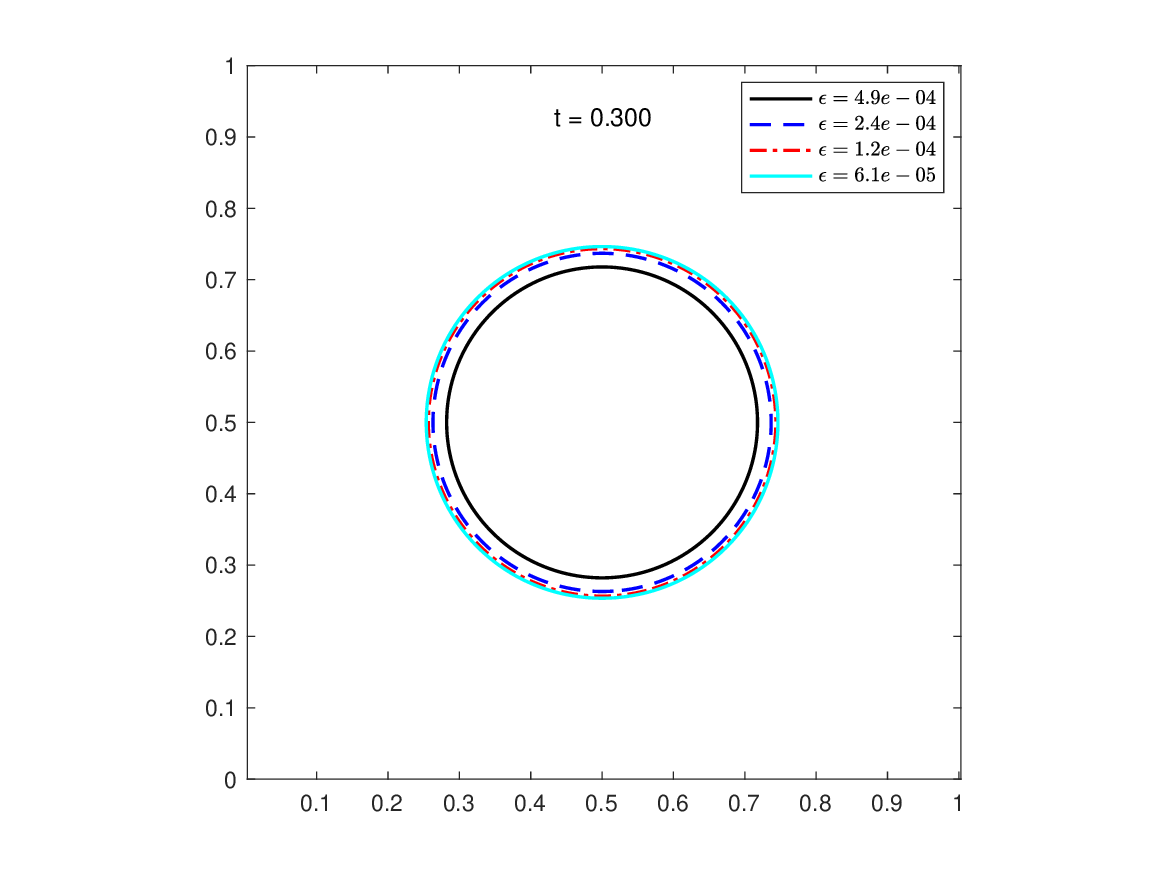}

%  \vspace{1ex}  % ??????????

%?????? ??? ??????%
\includegraphics[width=0.49\textwidth]{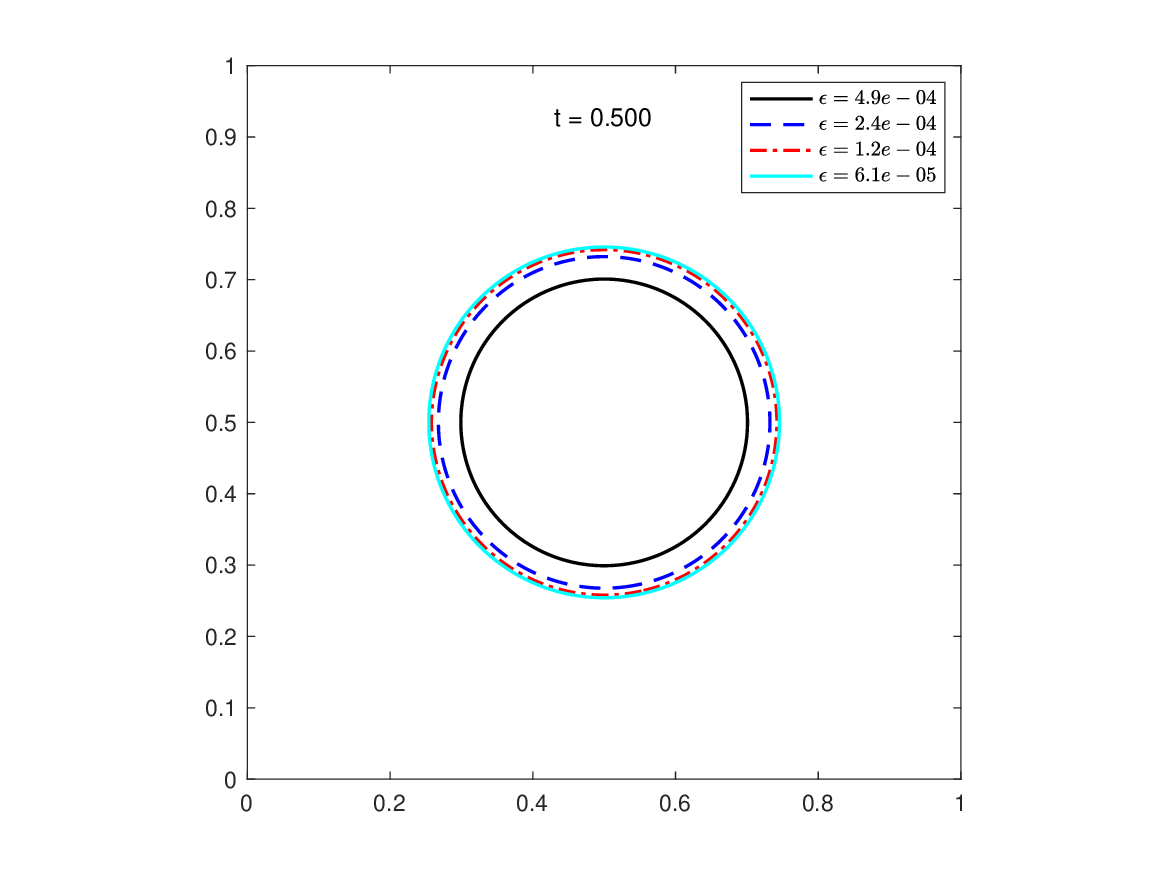}%
% \hspace{1pt}%
\includegraphics[width=0.49\textwidth]
{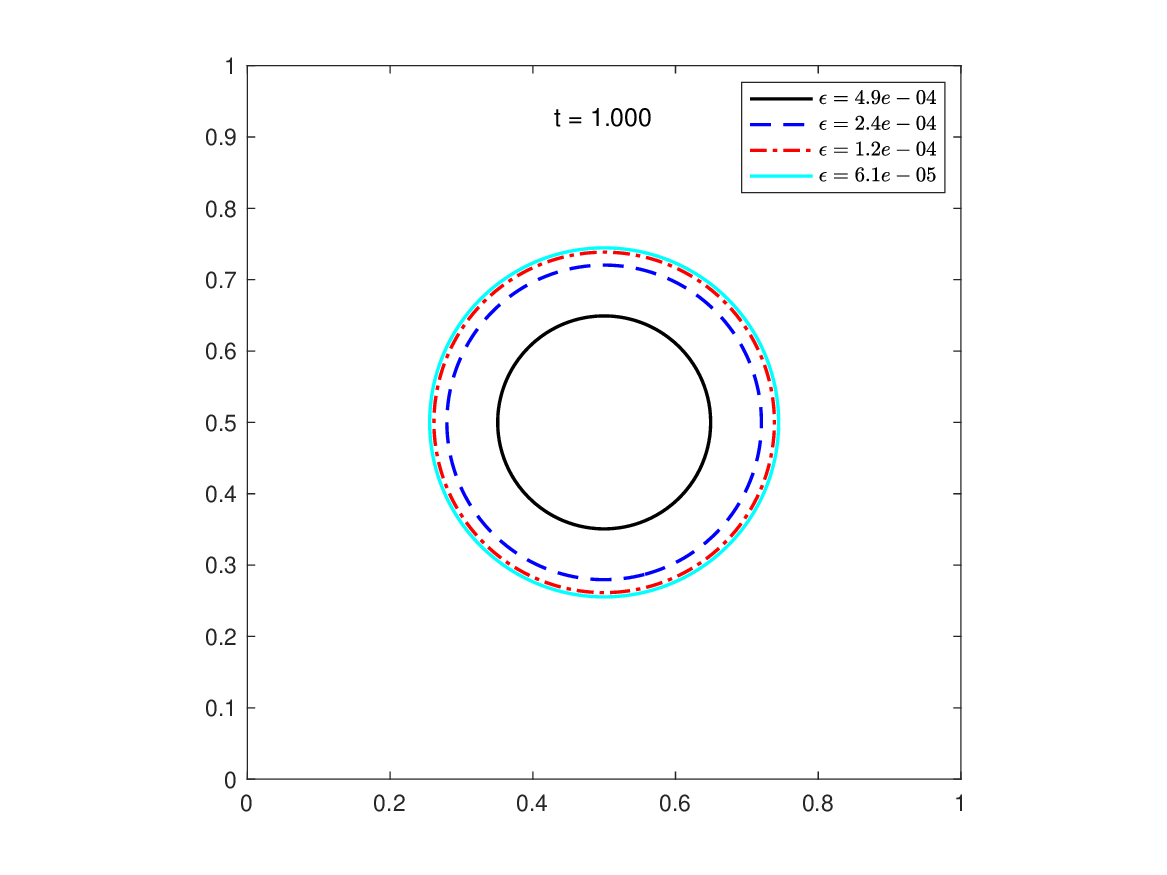}

\caption{Snapshots of the zero--level set of the solution at several time points for deterministic Cahn--Hilliard equation.}
\label{fig:De_gamma1}
\end{figure}
For $\gamma = 1$, the intensity of the space--time noise decreases as $\epsilon$ decreases, and the resulting evolution of the zero-level set (Figure~\ref{fig:EX4_gamma1}) is almost indistinguishable from that of the deterministic Cahn--Hilliard equation (Figure~\ref{fig:De_gamma1}) under the same initial configuration.
In addition, the interface remains nearly stationary, and the small random perturbations only cause mild fluctuations without altering the overall morphology.
This numerical observation is consistent with the theoretical prediction that the stochastic Cahn--Hilliard equation dynamics \eqref{model_epsilon_gamma} %with $\gamma \ge 1$ 
converges to the deterministic problem as $\epsilon \to 0$ (see \cite{Antonopoulou2021numath}).
For $\gamma = 0$, the noise amplitude remains of order $\mathcal{O}(1)$, and the zero-level set in Figure~\ref{fig:EX4_gamma0}
exhibits visible random oscillations around its deterministic profile.
The persistent randomness observed in the zero-level set suggests that the
limiting interface dynamics retain stochastic characteristics, consistent with
the conjectured stochastic sharp-interface limit %{\r in the additive noise case}
(see \cite{Antonopoulou2021numath}).
%Figure~\ref{fig:EX4_gamma0} shows zero-level-set snapshots for a single
%realization, where the influence of stochastic fluctuations is clearly visible.

\begin{figure}[htbp]
\centering
%?????? ??? ??????%
\includegraphics[width=0.49\textwidth]
{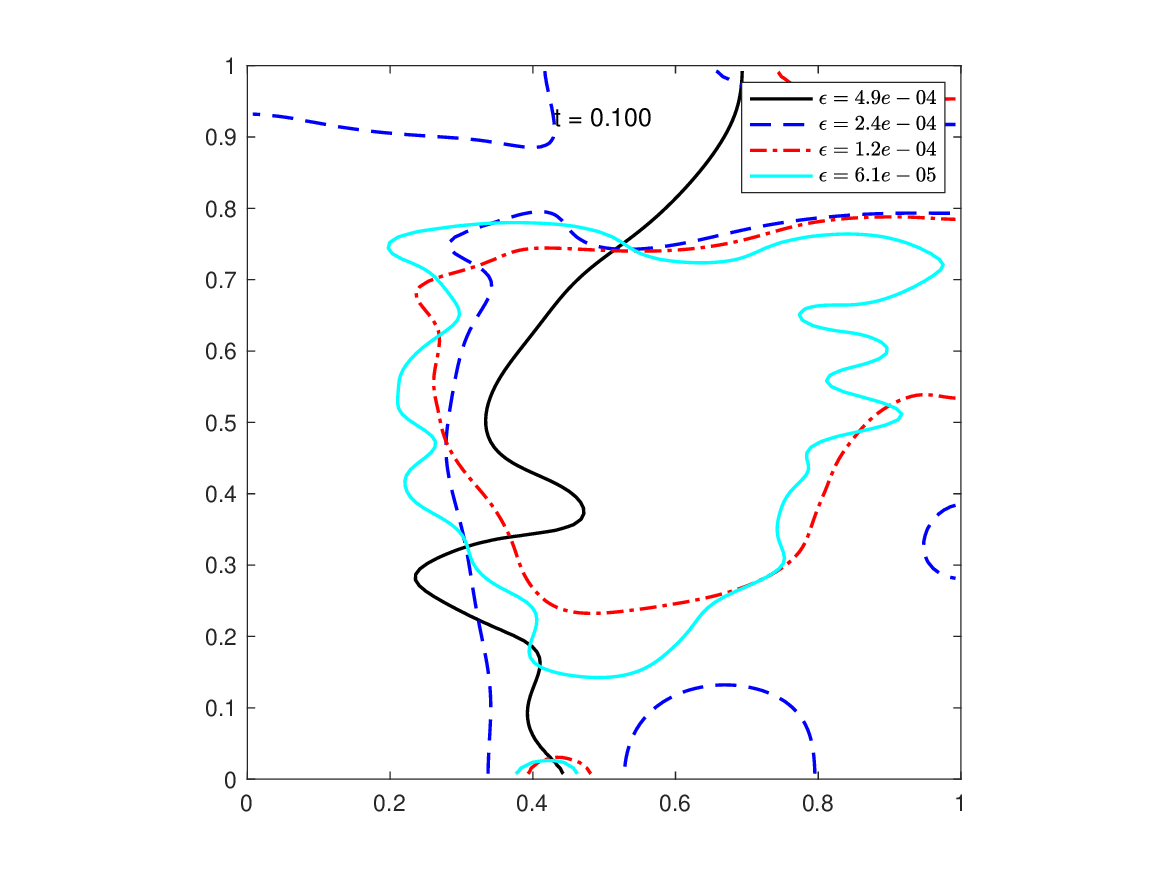}%
% \hspace{0.5pt}%
\includegraphics[width=0.49\textwidth]
{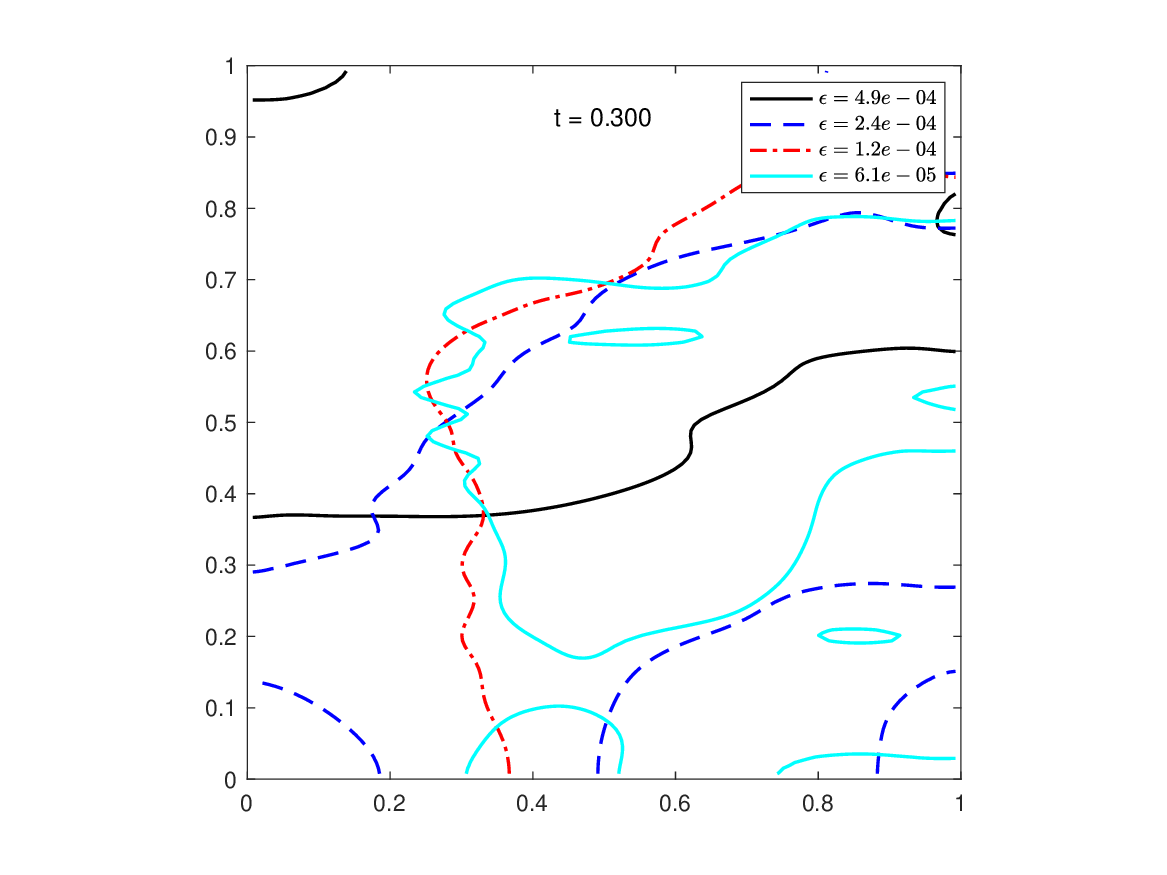}

% \vspace{1ex}  % ??????????

%?????? ??? ??????%
\includegraphics[width=0.49\textwidth]{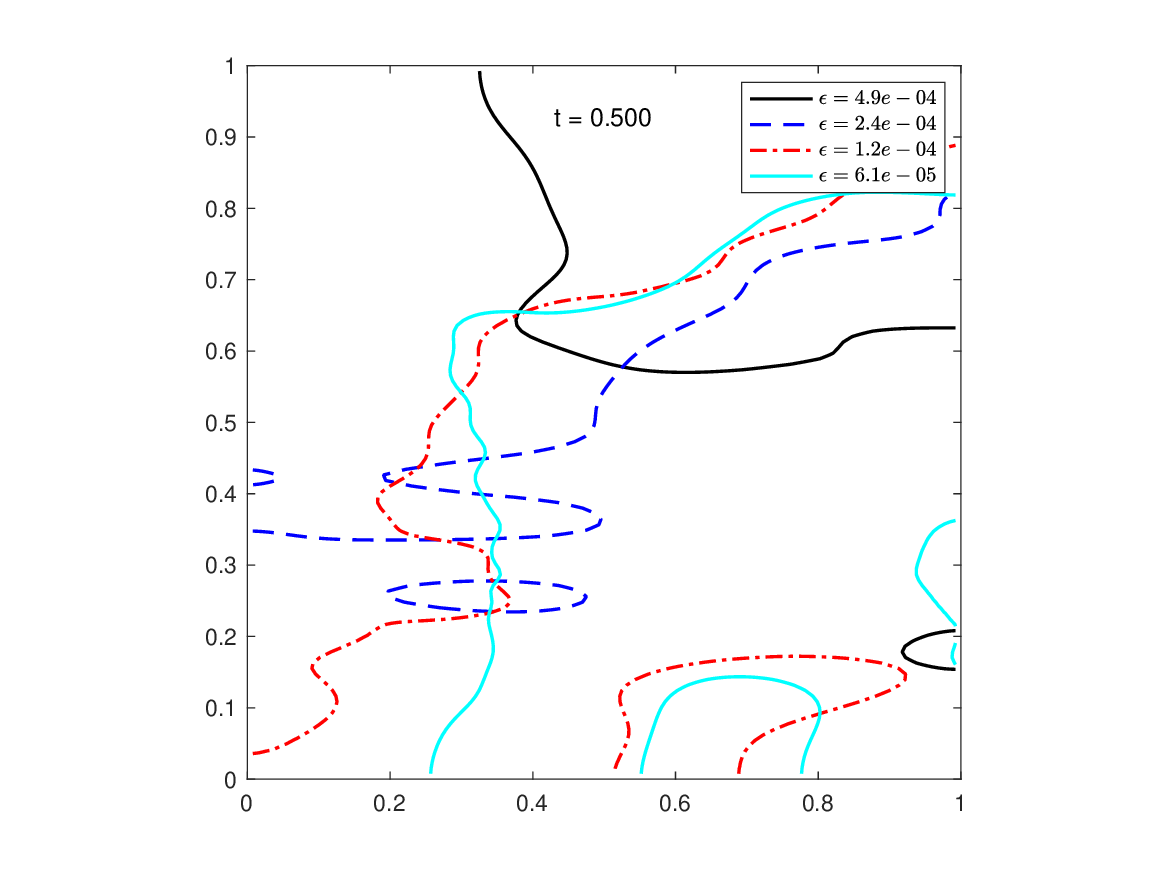}%
%  \hspace{1pt}%
\includegraphics[width=0.49\textwidth]
{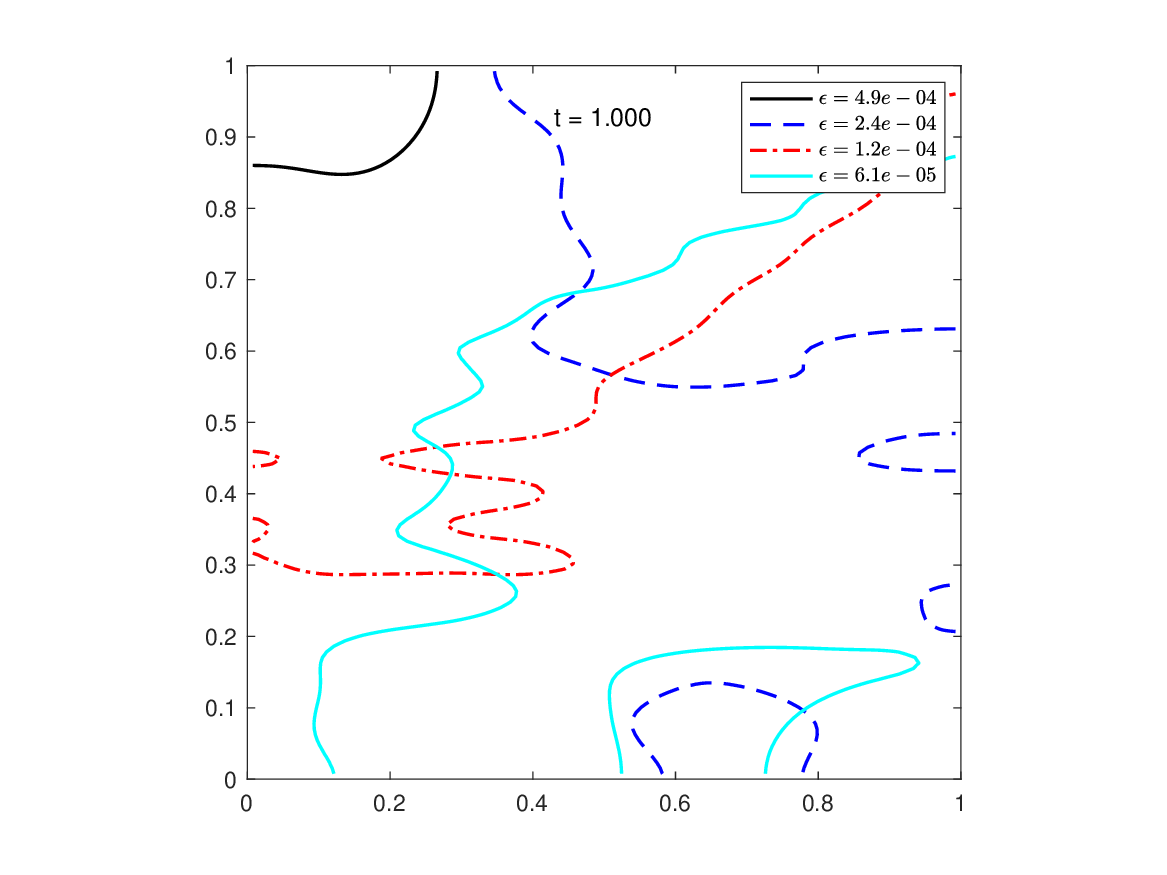}

\caption{Individual snapshots of the zero--level set of the solution at several time points for equation \eqref{model_epsilon_gamma} with $\gamma=0$.}
\label{fig:EX4_gamma0}
\end{figure}

\begin{figure}[htbp]
\centering
%?????? ??? ??????%
\includegraphics[width=0.49\textwidth]{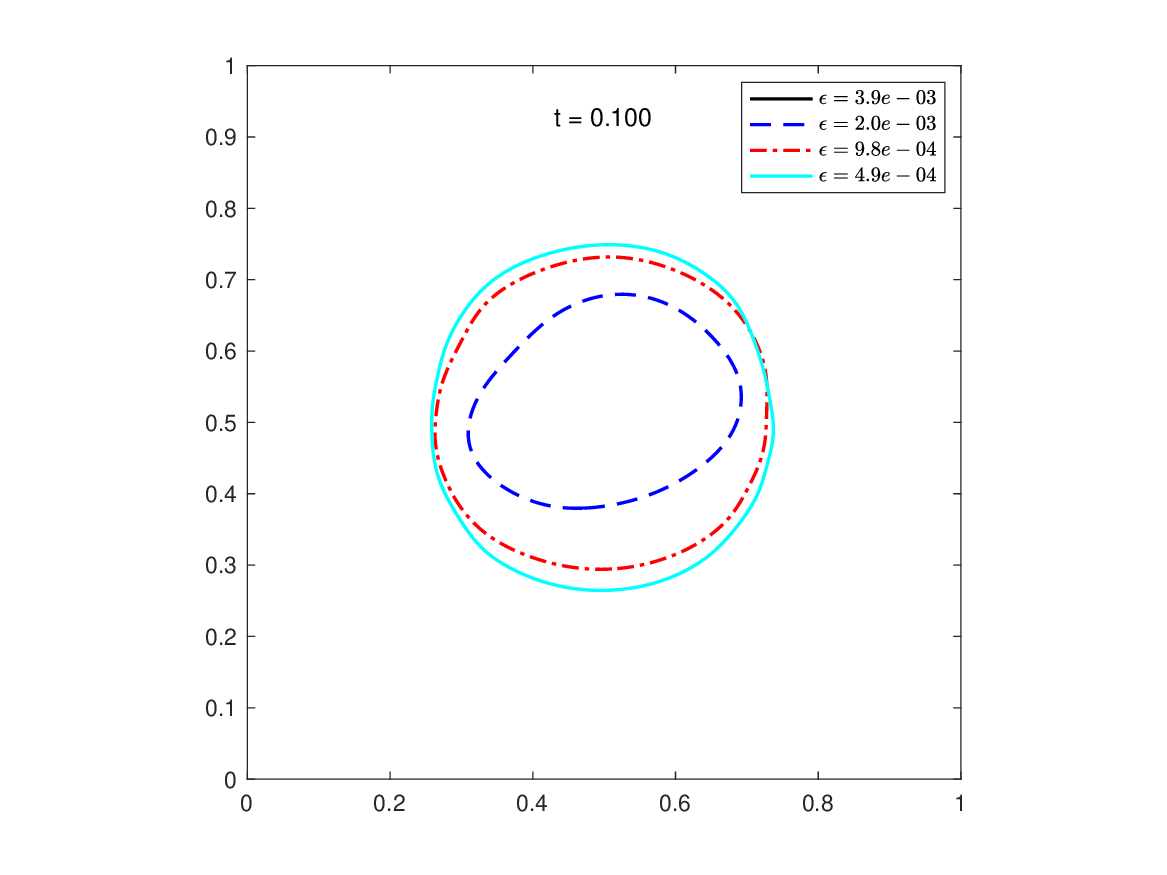}%
% \hspace{0.5pt}%
\includegraphics[width=0.49\textwidth]{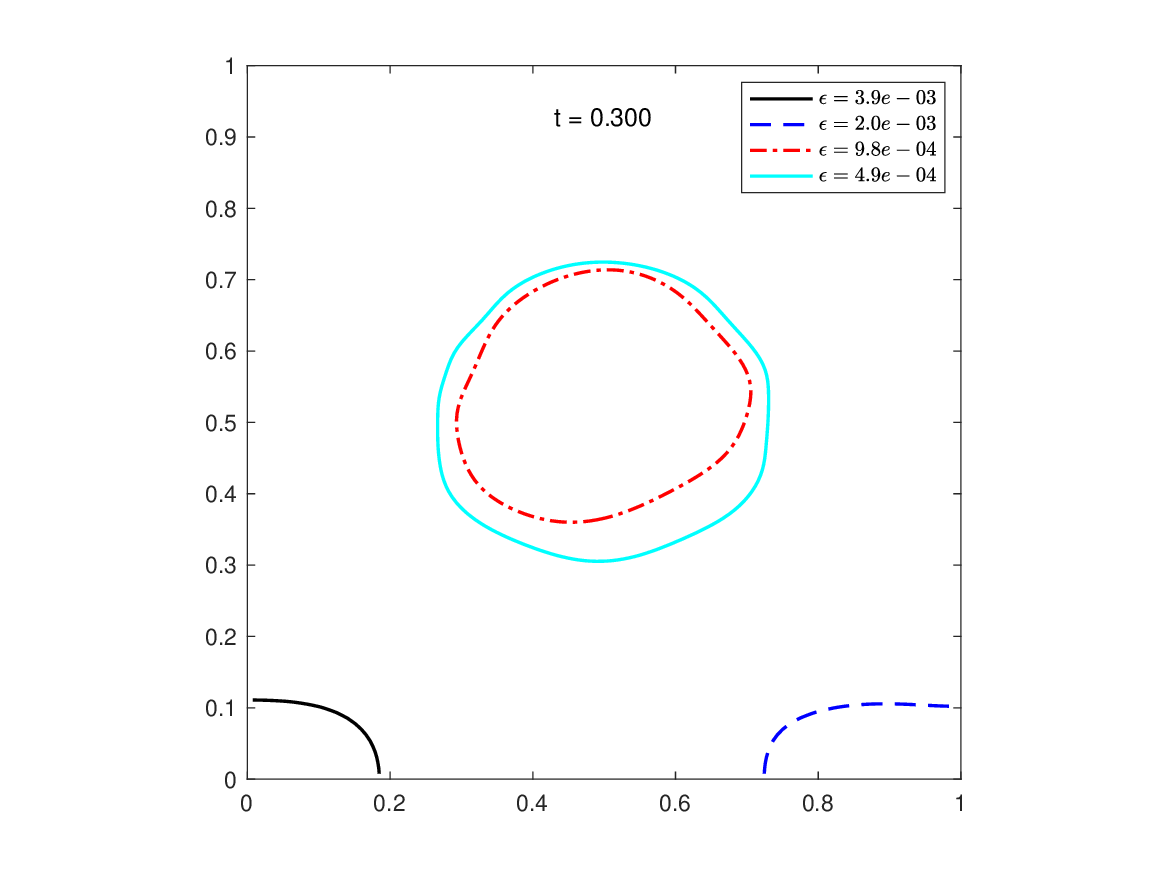}

%  \vspace{1ex}  % ??????????

%?????? ??? ??????%
\includegraphics[width=0.49\textwidth]{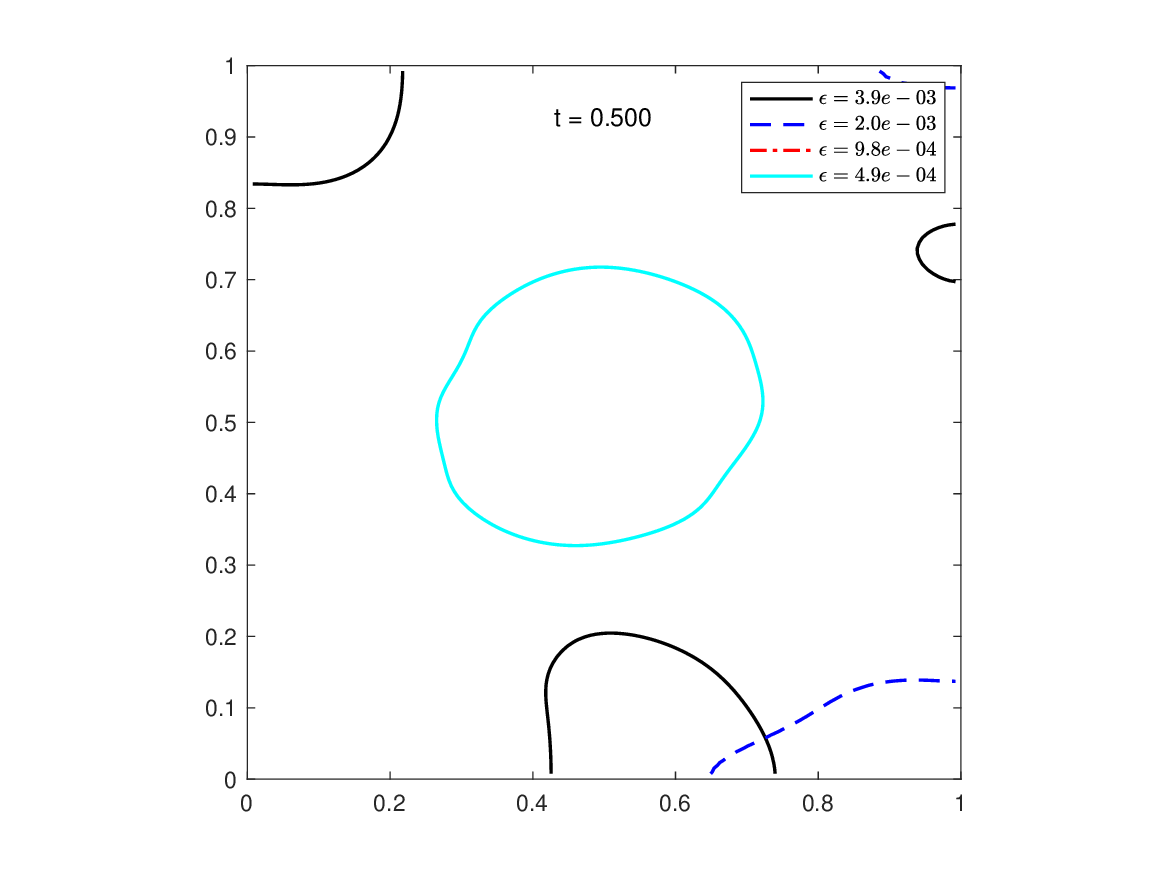}%
% \hspace{1pt}%
\includegraphics[width=0.49\textwidth]{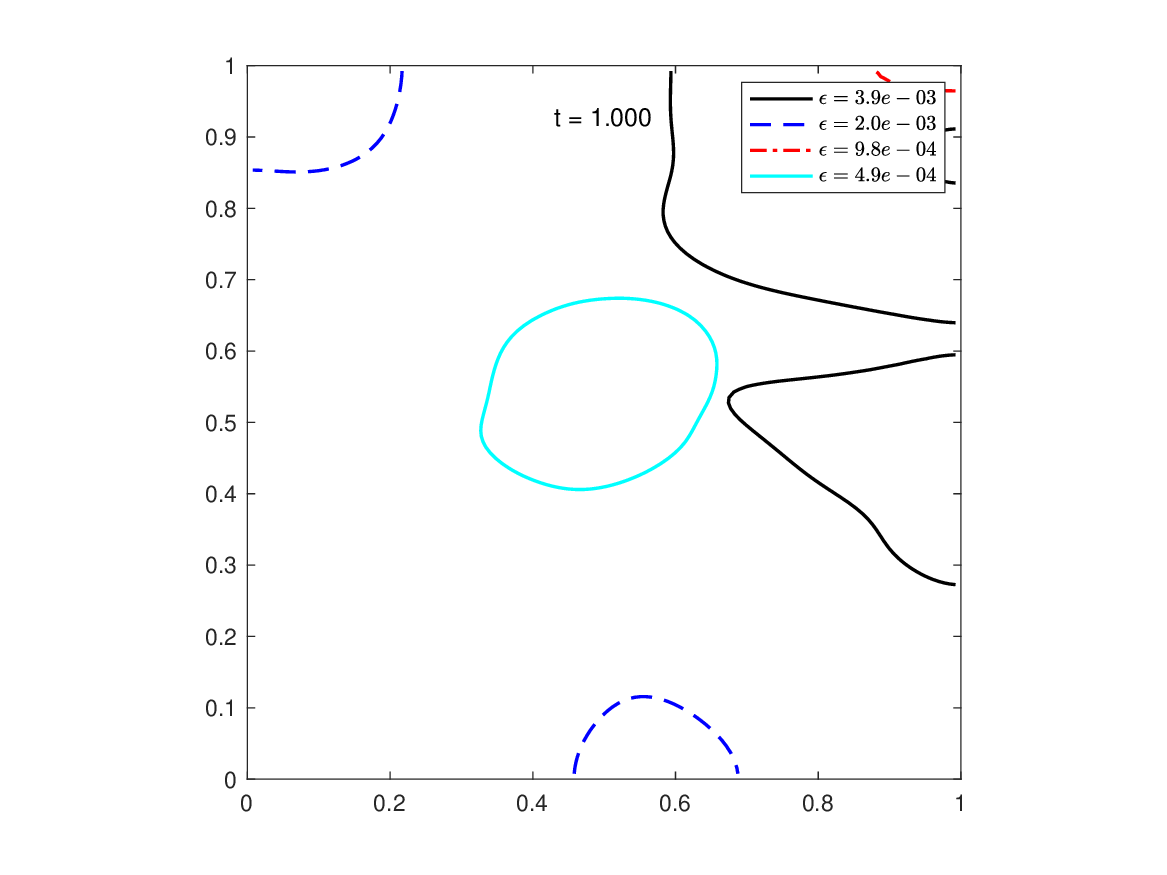}

\caption{Averaged snapshots of the zero--level set of the solution at several time points for equation \eqref{model_epsilon_gamma} with $\gamma=0$.}
\label{fig:EX4_gamma0_mean}
\end{figure}
Furthermore, we present in Figure~\ref{fig:EX4_gamma0_mean} the
averaged zero-level set of the numerical solution for equation \eqref{model_epsilon_gamma} with $\gamma=0$ over $P=100$ sample realizations. Although the averaged
interface does not fully coincide with the deterministic profile, one can still
observe the expected shrinking behavior as the parameter $\epsilon$ tends to
zero. This qualitative difference between the cases $\gamma=1$ and $\gamma=0$ highlights the distinct asymptotic regimes induced by different noise scalings in the stochastic Cahn--Hilliard equation near the sharp interface limit.

%\end{example}

%\begin{example}

% Specifically, we solve the following stochastic Cahn--Hilliard equation 
% \begin{align}\label{model_epsilon_gamma}
%  d\phi 
%    &= A\left(-\epsilon A \phi  +\frac{1}{\epsilon}{F'}(\phi)\right) dt 
%    + \epsilon^{\gamma} g(\phi)\,dW(t),
%    && \text{in } \mathcal{O}\times(0,T]
% \end{align}
% Here, $\epsilon>0$ controls the interfacial thickness, while the parameter $\gamma\ge0$ determines the scaling of the stochastic perturbation with respect to $\epsilon$ \cite{Antonopoulou2021numath}. 

\subsubsection{Discretization error and convergence order}

% First, we demonstrate the temporal convergence order of the exponential Euler SSAV scheme \eqref{schme_phi-intro} in Theorem \ref{Thm3_main}. The simulation is performed up to the final time \(T=0.01\) {\r the reviewer will criticize this short time. How about $T=0.1$ or $0.5$?}.
% % Let
% % \[
% % \textup{err}^n \;=\; \phi(t_n) - X_\tau^n,
% % % \quad
% % % \mathbb{E}\bigl[\|\textup{err}^n\|^2\bigr]
% % % =\mathbb{E}\bigl[\|\phi(t_n)-X_h^n\|^2\bigr].
% % \]
% Since the exact solution \(\phi(t_n)\) is unavailable, we approximate the mean?square error by 
% \[
% %\mathbb{E}\bigl[\|\textup{err}^n\|^2\bigr]
% %\approx
% \left(\frac{1}{P} \sum_{i=1}^P \bigl\|X_{\tau}^n(\omega_i) - X_{\text{ref}}^n(\omega_i)\bigr\|^2\right)^{\frac12},
% \]
% where $P=1000$, and $X^n_\tau$ denotes the numerical solution at time $t_n$ generated by \eqref{schme_phi-intro} with the time step size $\tau$.
% For each time step \(\tau\), the reference solution \(X_{\text{ref}}^n\) is computed using a time step size of \(\tau/2\). 
% For equation \eqref{ex:eps=1} with $g(\phi(t)) = (1+\phi(t)^2)^{-\frac12}$,
% Table~\ref{table1} lists the mean-square errors and the corresponding convergence orders of the numerical solution $X_n^\tau$. 
% The result indicates that, the proposed scheme \eqref{schme_phi-intro} achieves a temporal convergence rate of order $1/2$, which confirms the theoretical result stated in Theorem~\ref{Thm3_main}.

Continuing with equation \eqref{model_epsilon_gamma} for $\gamma=0$, we next test the convergence order of the exponential Euler SSAV scheme \eqref{schme_phi-intro}. Moreover, we also numerically investigate  the dependence of the mean square error on the interfacial parameter $\epsilon$. 
The simulations are performed up to the final time $T=0.1$, and the Monte--Carlo average is taken over $P=100$ samples.

We first fix the spatial discretization parameter at $M=128$ and use the numerical solution computed with the finest time step size $\tau_{\mathrm{ref}} = T/3200$ as the reference solution. 
A sequence of coarser time step sizes
$
\tau = 2\times 10^{-3}, 1\times 10^{-3}, 5\times 10^{-4}, 2.5\times 10^{-4}, 1.25\times 10^{-4}
$
is tested under different interfacial parameters
$
\epsilon = 0.16, 0.08, 0.04, 0.02, 0.01, 0.005.
$
For each pair $(\tau,\epsilon)$, the mean-square error at the final time $T$ is computed relative to the reference solution. The results, reported in Table~\ref{tab:tau_eps_error}, suggest that the temporal discretization error depends polynomially on $1/\epsilon$, with no indication of exponential growth.
Moreover, Table~\ref{tab:tau_eps_error} shows that, for each fixed $\epsilon$ in the tested range,
the temporal strong error exhibits an empirical rate close to $O(\tau^{1/2})$, which confirms the theoretical result stated in Theorem~\ref{Thm3_main}.

% In addition, when $\tau$ is fixed, the estimated orders become less sensitive to the refinement level as $\epsilon$ decreases, i.e., the rate appears more robust for smaller interfacial thickness. 
% This observation also clarifies why, in subsection~\ref{sub:freeepsilon}, we only present results up to $T=0.01$, which corresponds to the case $\epsilon=1$ and $\gamma=0$. In this regime, the asymptotic strong rate $O(\tau^{1/2})$ is observed only when the time step is sufficiently small; consequently, extending the simulation to larger final times would require substantially finer temporal resolutions in order to recover a comparably clean convergence behavior.

Next, we fix the time step at $\tau = T/800$ and take the numerical solution corresponding to the finest spatial resolution $M_{\mathrm{ref}} = 256$ as the reference solution. 
We consider spatial resolutions
$
M = 16, 32, 64, 128
$
for various interfacial parameters
$
\epsilon = 0.16, 0.08, 0.04, 0.02, 0.01.
$
For each pair $(M,\epsilon)$, the mean-square error at the final time $T$ is presented in Table~\ref{tab:space_eps_error}. 
This experiment similarly indicates that the spatial discretization error depends polynomially on $1/\epsilon$.

% \begin{table}[h]
% \centering
% \caption{Temporal strong $L^2$-errors for different pairs of time step $\tau$ and interface parameter $\epsilon$.}
% \label{tab:tau_eps_error}
% \setlength{\tabcolsep}{4mm}
% \begin{tabular}{c|cccccc}
	% \hline 
	% $\tau \backslash \epsilon$
	%  & 0.16 & 0.08 & 0.04 & 0.02 & 0.01 & 0.005 \\
	% \hline
	% 2e{-3}
	%  & 1.4065 & 2.1755 & 3.7003 & 4.4482 & 5.3107 & 7.4930 \\
	% 1e{-3} 
	%  & 0.8753 & 1.3335 & 2.0928 & 2.4716 & 3.0424 & 4.2077 \\
	% 5e{-4}  
	%  & 0.4988 & 0.8316 & 1.2276 & 1.3392 & 1.5817 & 2.1025 \\
	% 2.5e{-4}  
	%  & 0.2496 & 0.4640 & 0.6869 & 0.6702 & 0.8345 & 1.0379 \\
	% 1.25e{-4} 
	%  & 0.0993 & 0.2196 & 0.3491 & 0.3399 & 0.4410 & 0.5203 \\
	% \hline
	% \end{tabular}
% \end{table}
\begin{table}[h]
\centering
\caption{Temporal mean-square errors and orders for different time step sizes $\tau$ and interfacial parameters $\epsilon$.}
\label{tab:tau_eps_error}
\setlength{\tabcolsep}{3.6mm}
\begin{tabular}{c|cc|cc|cc}
	\hline
	\multirow{2}{*}{$\tau$}
	& \multicolumn{2}{c|}{$\epsilon=0.16$}
	& \multicolumn{2}{c|}{$\epsilon=0.08$}
	& \multicolumn{2}{c}{$\epsilon=0.04$}\\
	%\cline{2-7}
	& {Error} & {Order} & {Error} & {Order} & {Error} & {Order}\\
	\hline
	2e-3     & 1.1860 & {--}      & 1.4749 & {--}      & 1.9236 & {--}      \\
	1e-3     & 0.9356 & 0.3427  & 1.1546 & 0.3520  & 1.4467 & 0.4118  \\
	5e-4     & 0.7063 & 0.4053  & 0.9119 & 0.3400  & 1.1079 & 0.3847  \\
	2.5e-4   & 0.4996 & 0.4989  & 0.6812 & 0.4210  & 0.8287 & 0.4183  \\
	1.25e-4  & 0.3151 & 0.6657  & 0.4686 & 0.5397  & 0.5908 & 0.4871  \\
	\hline
	\hline
	\multirow{2}{*}{$\tau$}
	& \multicolumn{2}{c|}{$\epsilon=0.02$}
	& \multicolumn{2}{c|}{$\epsilon=0.01$}
	& \multicolumn{2}{c}{$\epsilon=0.005$}\\
%\cline{2-7}
	& {Error} & {Order} & {Error} & {Order} & {Error} & {Order}\\
	\hline
	2e-3     & 2.1091 & {--}      & 2.3045 & {--}      & 2.7373 & {--}      \\
	1e-3     & 1.5721 & 0.4249  & 1.7443 & 0.4022  & 2.0518 & 0.4158  \\
	5e-4     & 1.1574 & 0.4420  & 1.2580 & 0.4716  & 1.4500 & 0.5012  \\
	2.5e-4   & 0.8187 & 0.4990  & 0.9135 & 0.4619  & 1.0188 & 0.5099  \\
	1.25e-4  & 0.5830 & 0.4895  & 0.6641 & 0.4600  & 0.7213 & 0.4970  \\
	\hline
\end{tabular}
\end{table}

% \begin{table}[h]
% \centering
% \caption{Spatial mean-square errors for different spatial parameters $M$ and interfacial parameter $\epsilon$.}
% \label{tab:space_eps_error}
% \setlength{\tabcolsep}{4mm}
% \begin{tabular}{c|ccccc}
	% \hline
	% $M  $ 
	%  & $\epsilon = 0.16$ & $\epsilon =0.08$ & $\epsilon =0.04$ & $\epsilon =0.02$ & $\epsilon =0.01$ \\
	% \hline
	% 16  & 0.0284   & 0.1172   & 0.3203   & 0.5277   & 0.5619 \\
	% 32  & 0.0115   & 0.0231   & 0.0538   & 0.1255   & 0.2670 \\
	% 64  & 0.0034   & 0.0092   & 0.0176   & 0.0239   & 0.0596 \\
	% 128 & 5.9937e-04 & 0.0013 & 0.0031 & 0.0053 & 0.0121 \\
	% \hline
	% \end{tabular}
% \end{table}

\begin{table}[h]
\centering
\caption{Spatial strong $L^2$-errors for different spatial parameter $M$ and interface parameters $\epsilon$.}
\label{tab:space_eps_error}
\setlength{\tabcolsep}{4mm}
\begin{tabular}{c|ccccc}
	\hline
	\diagbox{$M$}{Error}{$\epsilon$}
	& $0.16$ & $0.08$ & $0.04$ & $0.02$ & $0.01$ \\
	\hline
	16  & 0.1685 & 0.3423 & 0.5660 & 0.7264 & 0.7496 \\
	32  & 0.1072 & 0.1520 & 0.2319 & 0.3543 & 0.5167 \\
	64  & 0.0583 & 0.0959 & 0.1327 & 0.1546 & 0.2441 \\
	128 & 0.0245 & 0.0361 & 0.0557 & 0.0728 & 0.1100 \\
	\hline
\end{tabular}
\end{table}

\section{Concluding remarks}	
In this work, we proposed and analyzed a semi-implicit numerical scheme for the stochastic Cahn--Hilliard equation driven by multiplicative noise, combining the SSAV approach with the exponential Euler method. The resulting exponential Euler SSAV scheme is shown to achieve the optimal strong convergence rate of order $\frac 12$
in the case of trace-class noise. Furthermore, we proved that the scheme asymptotically preserves the averaged energy evolution law of the underlying continuous system. We also conducted a series of numerical experiments to investigate the effects of the interfacial parameter and noise intensity on interface evolution, as well as the dependence of discretization error on the interfacial parameter. A theoretical investigation of these numerical phenomena will be carried out in future work. Meanwhile, we also plan to extend the proposed framework to more general noise structures and more complex phase-field models, as well as to develop higher-order structure-preserving numerical schemes.

\appendix
\section{Proof of auxiliary results}\label{Appendix}
\subsection{Proof of Corollary \ref{lemma_stability}}\label{App_Cor1}
Starting from \eqref{lema1_eq6_before}, we first take the supremum over \( m \in \{0,1, \cdots, N-1\} \) before evaluating the expectations, which gives
\begin{align}\label{eq:Lemma_Stability_Step2}\notag
&\mathbb{E}\left[\sup_{0\leq m\leq N-1}\left\|\nabla  X^{m+1}\right\|^{2p}\right]
+\mathbb{E}\left[\sup_{0\leq m\leq N-1}|r^{m+1}|^{2p}\right]
+\mathbb{E}\bigg[\bigg(\sum_{n=0}^{N-1}|r^{n+1}-r^n|^{2}\bigg)^p\bigg]  
\\
&\quad +\mathbb{E}\bigg[\bigg(\sum_{n=0}^{N-1}\|({I-S^2(\tau)})^{\frac{1}{2}}(-A)^{-\frac 12}\tilde{\mu}^n\|^2\bigg)^p\bigg]\notag \\
&\leq C + C\mathbb{E}\left[\left\|\nabla  \phi^{0}\right\|^{2p}\right] + C\mathbb{E}\left[|r^{0}|^{2p}\right] +C \tau \sum_{n=0}^{N-1}\mathbb{E}\left[|r^{n}|^{2p}\right]+\sum_{i=1}^5\mathbb{E}\left[\sup_{0\le m\le N-1}R_i^m\right].
\end{align}
Since the sequences $\{R_1^m\}_{m=0}^{N-1},\cdots,\{R_4^m\}_{m=0}^{N-1}$ 
are non-decreasing in \(m\),  they reach their maximums at \( m = N-1 \),  which enables us to  utilize the estimates derived previously for \( R_1^{N-1}, \cdots, R_4^{N-1} \) and obtain  
\begin{align*}
\mathbb{E}\left[\sup_{0\le m\le N-1}R_i^m\right]= \mathbb{E}\left[R_i^{N-1}\right]\le C,\qquad i=1,2,3,4.
\end{align*}
By Doob's martingale inequality,  
\eqref{lema1_eq_R71}, \eqref{lema1_eq_R72} together with \eqref{Lemma1_step1}, 
we have  
\begin{align*}
&\mathbb{E}\left[\sup_{0\le m\le N-1}R_5^m\right]\\
&=C\mathbb{E}\Bigg[\sup_{0\le m\le N-1}\bigg|\sum_{n=0}^m\int_{t_n}^{t_{n+1}}\bigg\langle-A X^{n}+\frac{r^{n}}{\sqrt{E_{\textup{p}}(X^{n})}}F'( X^{n}),g( X^{n})d  W(s)\bigg\rangle\bigg|^p\Bigg]\notag\\
&\le  C\mathbb{E}\Bigg[\bigg|\sum_{n=0}^{N-1}\int_{t_n}^{t_{n+1}}\sum_{k=1}^{\infty}\bigg\langle-A X^{n}+\frac{r^{n}}{\sqrt{E_{\textup{p}}(X^{n})}}F'( X^{n}),g( X^{n})Q^{\frac 12}e_k\bigg\rangle^2d  s\bigg|^{\frac p2}\Bigg]\notag\\
%&\le C\mathbb{E}\Bigg[\bigg|\sum_{n=0}^{N-1}\tau\sum_{k=1}^{\infty}\bigg\langle-A X^n+\frac{r^n}{\sqrt{E_{\textup{p}}(X^{n})}}F'( X^n),g( X^n)Q^{\frac 12}e_k\bigg\rangle^2\bigg|^{\frac p2}\Bigg]\\
&\le C(R_{51}^{N-1}+R_{52}^{N-1})\le C+C\tau\sum_{n=0}^{N-1}\mathbb{E}\left[\left|r^{n}\right|^{2p}\right]
+C\tau\sum_{n=0}^{N-1}\mathbb{E}\left[\|\nabla X^n\|^{2p}\right]\le C.\notag
\end{align*}
Finally, taking \eqref{Lemma1_step1}, \( \phi^0 \in \dot{H}^1(\mathcal{O}) \), and \eqref{lema1_eq_R1_insert_hey} into account, we conclude the proof of \eqref{lemma_re} from \eqref{eq:Lemma_Stability_Step2}. 			\hfill$\square$
\subsection{Proof of Lemma \ref{lemma_tool}}\label{App-lemma-tool}
%  \begin{proof} 
In view of H\"older's inequality, it suffices to consider $p>1$ sufficiently large.
For $\beta\in (0,2)$, 
by the mild formulation of \eqref{lemma6_eq3_R1_Zn},  the contractivity of $S(\cdot)$, and H\"older's inequality, we obtain 
\begin{align}\label{lemma6_eq3_R1_1}\notag
\mathbb{E}\bigg[\sup_{t\in[0,T]}\|(-A)^{\frac{\beta}{2}}Z(t)\|^p\bigg]&\leq 
%             C\|(-A)^{\frac{\beta}{2}}\phi^0\|^p+C\mathbb{E}\bigg[\sup_{t\in[0,T]}\Big\|\int_0^{t}S(t-s)(-A)^{\frac{\beta}{2}+1}\tilde{F}(s)ds\Big\|^p\bigg]\\\notag
% % &\quad+C\mathbb{E}\bigg[\sup_{t\in[0,T]}\Big\|{\int_0^{t}(-A)^{\frac{\beta}{2}}S(t - s) \tilde{G}(s)}dW(s)\Big\|^p\bigg]\\
%&=:
C\|(-A)^{\frac{\beta}{2}}\phi^0\|^p+R_{1,\beta}+R_{2,\beta},
\end{align}
with
\begin{align*}
R_{1,\beta}&:=\mathbb{E}\bigg[\sup_{t\in[0,T]}\Big\|\int_0^{t}S(t-s)(-A)^{\frac{\beta}{2}+1}\tilde{F}(s)ds\Big\|^p\bigg],\\\quad
R_{2,\beta}&:=\mathbb{E}\bigg[\sup_{t\in[0,T]}\Big\|{\int_0^{t}(-A)^{\frac{\beta}{2}}S(t - s) \tilde{G}(s)}dW(s)\Big\|^p\bigg].
\end{align*}
By the Minkowski inequality, \eqref{assum1_eq1}, and H\"older's inequality, we choose \( q_1 = \frac{1}{2}(1 + \frac{4}{\beta + 2}) >1\) with $(\frac{\beta}{4}+\frac12 )q_1\in (0,1)$ such that for any  $p\ge \frac{q_1}{q_1-1}$, 
\begin{align}%\label{lemma6_eq3_R12}
R_{1,\beta}&\leq C\mathbb{E}\bigg[\sup_{t\in[0,T]} \bigg(\left(\int_0^{t} (t - s)^{\left(-\frac{\beta}{4} - \frac{1}{2}\right) q_1} ds \right)^{\frac{1}{q_1}} 
\left( \int_0^{t} \|\tilde{F}(s)\|^{\frac{q_1}{q_1 - 1}} ds \right)^{\frac{q_1 - 1}{q_1}} \bigg)^p \bigg] \notag \\
&\leq C\, \mathbb{E}\Bigg[\bigg( \int_0^{T} \|\tilde{F}(s)\|^{\frac{q_1}{q_1 - 1}} ds \bigg)^{\frac{(q_1 - 1)p}{q_1}} \Bigg]\le C\mathbb{E}\int_0^{T} \|\tilde{F}(s)\|^{p} ds.
\end{align}
The factorization formula \cite[Proposition 5.9 \& Theorem 5.10]{DaPratoZabczyk2014} implies that for $p>1$ and $\rho>1/p$, 
\begin{equation}\label{eq:R3beta}
R_{2,\beta}=C\mathbb{E}\bigg[\sup_{t\in[0,T]}\bigg\|{\int_0^{t}(-A)^{\frac{\beta}{2}}S(t - s) \tilde{G}(s)}dW(s)\bigg\|^p\bigg]\le C(T,p)\mathbb{E}\int_{0}^{T}\|\tilde{Y}(u)\|^pdu,
\end{equation}
where
$
\tilde{Y}(u):=\int_{0}^{u}(u-s)^{-\varrho}S(u-s)(-A)^{\frac{\beta}{2}}\tilde {G}(s)dW(s)$ for any $u\in(0,T)$.
Next, using the BDG inequality, 
Minkowski inequality, and \eqref{assum1_eq1}, and by
choosing sufficiently small $\rho \in(0,1)$ and $q_2>1$ 
satisfying $2(\frac{\beta}{4}+\rho)q_2\in(0,1)$, 
one has that for any $p>\max\{\frac{2q_2}{q_2-1},\frac{1}{\rho}\}$ and $u\in (0,T)$,
\begin{align}\label{lemma6_eq3_R13_2-old}
\mathbb{E}\left[\|\tilde{Y}(u)\|^p\right]
&\leq C\mathbb{E}\left[\bigg(\int_{0}^{u}(u-s)^{-2(\frac{\beta}{4}+\varrho)}\|\tilde {G}(s)\|_{\mathcal{L}_2^0}^2ds\bigg)^{\frac{p}{2}}\right]\notag\\
&\le C\mathbb{E}\Bigg[\bigg(\left(\int_{0}^{u}(u-s)^{-2(\frac{\beta}{4}+\varrho)q_2}ds\right)^{\frac {1} {q_2}}\left(\int_{0}^u\|\tilde {G}(s)\|_{\mathcal{L}_2^0}^{2\frac{q_2}{q_2-1}}ds\right)^{\frac{q_2-1}{q_2}}\bigg)^{\frac{p}{2}}
\Bigg]\notag\\
&\le C\mathbb{E}\int_{0}^T\|\tilde {G}(s)\|_{\mathcal{L}_2^0}^{p}ds.
\end{align}
Gathering the inequalities \eqref{lemma6_eq3_R13_2-old} and \eqref{eq:R3beta} together gives the estimate of $R_{2,\beta}$. Then
substituting the estimates of $R_{1,\beta}$ and $R_{2,\beta}$ into \eqref{lemma6_eq3_R1_1} completes the proof of Lemma \ref{lemma_tool}.
\hfill$\square$
\subsection{Proof of \eqref{eq:chinp}}\label{App_4.24}

From the definition of $\chi^n$ in \eqref{chi} and utilizing Young's inequality, we have
$\|\chi^{n}\|^p
\le C\tau^{\frac{p}{2}}|r^{n+1}|^{2p}+C\tau^{-\frac{p}{2}}(J_1^n+J_2^n)$,
where 
\begin{align*}
J_1^n&:=
\frac{1}{E_{\textup{p}}(X^{n})^{3p}}\|F'( X^{n})\|^{2p}\left|\left\langle F'( X^n),g( X^{n})\delta  W^{n}\right\rangle\right|^{2p},\\
J_2^n&:=\frac{1}{E_{\textup{p}}(X^{n})^p}\|F''( X^{n})g( X^{n})\delta  W^{n}\|^{2p}.
\end{align*}
Therefore,
applying H\"older's inequality and using \eqref{lema1_eq_R3_insert1} as well as  \eqref{eq:FXn}, we obtain that for any $p\ge 1$,
\begin{align*}%\label{lemma2_J1_final}
\mathbb{E}[ J_1^n]&\leq C(\mathbb{E}[\|F'( X^{n})\|^{4p}] )^{\frac12}\bigg(\mathbb{E}\bigg[ \frac{1}{E_{\textup{p}}( X^{n})^{3p}}\left|\left\langle F'( X^n),g( X^{n})\delta  W^{n}\right\rangle\right|^{4p}\bigg]\bigg)^{\frac12}\le C\tau^{p}.
\end{align*}
%{\r why there is square outside?}
As for $J_2^n$, it follows from the BDG inequality, 
\eqref{lema1_eq_R2_insert}, and \eqref{lema1_eq_R1_insert_Q} that for any $p\ge 1$,
\begin{align*}%\notag\label{lemma2_J2}
\mathbb{E}\left[ J_2^n\right]&
\le 
C\mathbb{E}\bigg[\frac{1}{E_{\textup{p}}(X^{n})^p}\bigg(\tau\sum_{k=1}^{\infty}\|F''( X^{n})g( X^{n})Q^{\frac{1}{2}}e_k\|^2\bigg)^p\bigg]\\
&\leq 
C\tau^{p}\mathbb{E}\bigg[\frac{\|F''( X^{n})\|^{2p}}{E_{\textup{p}}(X^{n})^p}\bigg(\sum_{k=1}^{\infty}\|g( X^{n})Q^{\frac{1}{2}}e_k\|^2_{L^\infty}\bigg)^p\bigg]\leq C\tau^{p}.
\end{align*}
Furthermore, taking Corollary \ref{lemma_stability} into account, we conclude that for any $p\ge1$,
\begin{equation*}
\mathbb{E}\left[\|\chi^{n}\|^p\right]\le  C\tau^{\frac{p}{2}}\mathbb{E}\left[|r^{n+1}|^{2p}\right]+C\tau^{-\frac{p}{2}}(\mathbb{E}\left[J_1^n\right]+\mathbb{E}\left[J_2^n\right])\le C\tau^{\frac{p}{2}}.
\end{equation*}
\hfill$\square$

\subsection{Proof of \eqref{eq:fnH1}}\label{App-fnH1}
Under Assumptions \ref{assum2} and \ref{assum1}, the Dirichlet boundary conditions for the numerical solution and the integration by parts formula ensure
$\|\nabla\tilde{f}^n\|=\|(-A)^{\frac12}\tilde{f}^n\|$. Invoking Assumption \ref{assum2}, the chain rule, the integration by parts formula, Lemma~\ref{lemma_stability}, and \eqref{eq:Linfty}, for any $q\ge1$,
% {\r how do you apply it}{\r why the first equality hold?}
\begin{align}\label{lemma6_beta=2_1_insert}
&\mathbb{E}\left[\|(-A)^{\frac12}F'(X^n)\|^{2q}\right]= \mathbb{E}\left[\|\nabla F'(X^n)\|^{2q}\right]=
\mathbb{E}\left[\left\|F''(X^{n})\nabla X^n\right\|^{2q}\right]\\\notag
&
\leq C\mathbb{E}\left[\left\|F''(X^{n})\right\|_{L^{\infty}}^{4q}\right]
+C\mathbb{E}\left[\|\nabla X^n\|^{4q}\right]\\\notag
&\leq C+C\mathbb{E}\left[\left\|X^{n}\right\|_{L^{\infty}}^{8q}\right]
+C\mathbb{E}\left[\|\nabla X^n\|^{4q}\right]
\leq C.
\end{align}
Hence,
by  \eqref{fn}, Corollary \ref{lemma_stability}, and H\"older's inequality, for any $q\ge1$,
\begin{align}\label{eq:fnH1-old}\notag
\mathbb{E}[\|(-A)^{\frac12}\tilde{f}^{\,n}\|^{2q}]
&\le 
C(\mathbb{E}[|r^{n+1}|^{4q}])^{\frac12}	(\mathbb{E}[\|(-A)^{\frac12}F'(X^n)\|^{4q}])^{\frac12}
+C\mathbb{E}[\|(-A)^{\frac12}\chi^n\|^{2q}]\\
&\le C+C\mathbb{E}[\|(-A)^{\frac12}\chi^n\|^{2q}].
\end{align}
By the definition of $\chi^n$ in \eqref{chi}, we have 
\begin{align*}
\mathbb{E}[\|(-A)^{\frac12}\chi^n\|^{2q}]&\le C\mathbb{E}\left[\Big\|\frac{1}{4E_{\textup{p}}(X^{n})^{3/2}}\langle F'(X^n),g(X^{n})\delta  W^{n}\rangle(-A)^{\frac12} F'(X^{n})\Big\|^{2q}\right]\\
&\quad+C\mathbb{E}\bigg[\Big\|\frac{1}{2\sqrt{E_{\textup{p}}(X^{n})}}(-A)^{\frac12}\left(F''(X^{n})g(X^{n})\delta  W^{n}\right)\Big\|^{2q}\bigg]\\
&=:\mathrm{I}_q^n+\mathrm{II}_q^n.
\end{align*}

\textbf{Estimate of $\mathrm{I}_q^n$.}
By the BDG inequality and \eqref{lema1_eq_R1_insert_Q}, it holds that for any $q_1\ge1$,
\begin{equation}\label{eq:BDG}
\mathbb{E}[  \|g(X^n)\delta W^n\|^{2q_1} ]\le C\mathbb{E}[ (\tau \|g(X^n)\|^2_{\mathcal{L}_2^0})^{q_1} ]\le C\tau^{q_1}.
\end{equation} 
Then applying H\"older inequality, \eqref{eq:BDG}, \eqref{lemma6_beta=2_1_insert}, and \eqref{eq:FXn} gives that $\mathrm{I}_q^n\le C(q)\tau^{q}$ for any $q\ge1$.

\textbf{Estimate of $\mathrm{II}_q^n$.} The BDG inequality results in
\begin{align*}%\label{lemma6_beta=2_3}
\mathrm{II}_q^n&\le C\mathbb{E}\bigg[\Big\|(-A)^{\frac{1}{2}}\int_{t_n}^{t_{n+1}}F''(X^{n})g(X^{n})d W(s)\Big\|^{2q}\bigg]\\
&\le C\mathbb{E}\bigg[\bigg(\tau\sum_{k=1}^\infty\|(-A)^{\frac{1}{2}}(F''(X^{n})g(X^{n})Q^{\frac{1}{2}}{e}_k)\|^2\bigg)^{q}\bigg].
\end{align*}
Integrating by parts and by the chain rule, we obtain 
\begin{align*}
&\mathbb{E}\bigg[\bigg(\sum_{k=1}^\infty\|(-A)^{\frac{1}{2}}(F''(X^{n})g(X^{n})Q^{\frac{1}{2}}{e}_k)\|^2\bigg)^{q}\bigg]\\
&=\mathbb{E}\bigg[\bigg(\sum_{k=1}^\infty\|\nabla(F''(X^{n})g(X^{n})Q^{\frac{1}{2}}{e}_k)\|^2\bigg)^{q}\bigg]\\
&\le C\mathbb{E}\bigg[\bigg(\sum_{k=1}^\infty\|F'''(X^{n}) \nabla X^n g(X^{n})Q^{\frac{1}{2}}{e}_k\|^2\bigg)^{q}\bigg]\\
&\quad+C\mathbb{E}\bigg[\bigg(\sum_{k=1}^\infty\|F''(X^{n})\nabla(g(X^{n})Q^{\frac{1}{2}}{e}_k)\|^2\bigg)^{q}\bigg].
\end{align*}
It follows from Young's inequality,
\eqref{lema1_eq_R1_insert_Q}, \eqref{eq:Linfty}, and Corollary \ref{lemma_stability} that
\begin{align*}%\label{lemma6_beta=2_31}
&	\mathbb{E}\bigg[\bigg(\sum_{k=1}^\infty\|F'''(X^{n}) \nabla X^n g(X^{n})Q^{\frac{1}{2}}{e}_k\|^2\bigg)^{q}\bigg]\\
&\leq C\mathbb{E}\left[\left\|F'''(X^{n})\nabla X^n\right\|^{4q}\right] 
+C\mathbb{E}\bigg[\bigg(\sum_{k=1}^{\infty}\|g(X^n)Q^{\frac{1}{2}}e_{k}\|_{L^\infty}^2\bigg)^{2q}\bigg]\\
&\leq C\left(1+\mathbb{E}\left[\|X^n\|_{L^\infty}^{8q}\right]+\mathbb{E}\left[\|\nabla X^n\|^{8q}\right]\right)\leq C.
\end{align*}
Similarly, by Young's inequality, \eqref{lema1_eq_R6-new}, Lemma~\ref{lemma_stability}, and \eqref{eq:Linfty}, we also have
\begin{align*}%\label{lemma6_beta=2_32}
&\mathbb{E}\bigg[\bigg(\sum_{k=1}^\infty\|F''(X^{n})\nabla(g(X^{n})Q^{\frac{1}{2}}{e}_k)\|^2\bigg)^{q}\bigg]\\
&\leq C\mathbb{E}\left[\left\|F''(X^{n})\right\|_{L^{\infty}}^{4q}\right]
+C\mathbb{E}\bigg[\bigg(\sum_{k=1}^\infty\|\nabla (g(X^{n})Q^{\frac{1}{2}}e_{k})\|^{2}\bigg)^{2q}\bigg]\\
&\leq C\left(1+\mathbb{E}\left[\| X^n\|_{L^\infty}^{8q}\right]+\mathbb{E}\left[\|\nabla X^n\|^{4q}\right]\right)
\leq C.
\end{align*}
Consequently, we can infer that $\mathrm{II}_q^n\le C(q)\tau^q$ for any $q\ge1$.

Combining the estimates of $\mathrm{I}_q^n$ and $\mathrm{II}_q^n$, it follows that for any $q\ge1$, 
\begin{equation}\label{eq:chiH1norm}
\mathbb{E}[\|(-A)^{\frac12}\chi^n\|^{2q}]\le C(q)\tau^q.
\end{equation}
Finally, inserting \eqref{eq:chiH1norm} into \eqref{eq:fnH1-old} results in \eqref{eq:fnH1}.
\hfill$\square$
%\end{proof}


\subsection{Proof of \eqref{eq:J3-10}}\label{App-J3-10}
We begin with the estimate of $J_3^n$. By \eqref{eq:termJ3},
%{\r the estimate does not match the def of $J_3^n$?}
\begin{equation*}
\mathbb{E}[|J_{3}^n|]
=\mathbb{E}\left[\Big|-|r^{n+1} - r^n|^2+\frac{1}{4E_{\textup{p}}(X^n)}\left| \left\langle F'(X^n),X^{n+1} - X^n\right\rangle\right|^2\Big|\right].
%&=\mathbb{E}\left[\bigg|\bigg(r^{n+1} - r^n+\frac{ \left\langle F'(X^n),X^{n+1} - X^n\right\rangle}{2\sqrt{E_{\textup{p}}(X^n)}}\bigg)\bigg(r^{n+1} - r^n-\frac{\left\langle F'(X^n),X^{n+1} - X^n\right\rangle}{2\sqrt{E_{\textup{p}}(X^n)}}\bigg)\bigg|\right]
\end{equation*}
Using H\"older's inequality, \eqref{eq:FXn} and Lemma \ref{lem:time}, we have
\begin{align}\label{eq:Low}
\mathbb{E}[|\langle F'(X^n),X^{n+1} - X^n\rangle|^2]\le C\mathbb{E}[\|F'(X^n)\|^2\|X^{n+1} - X^n\|^2]\le C\tau.
\end{align}
Applying H\"older's inequality, Assumption~\ref{assum2}, the Sobolev embedding \(L^\infty(\mathcal{O}) \hookrightarrow L^6(\mathcal{O})\), Lemma~\ref{lem:time},   \eqref{eq:Linfty}, and \eqref{eq:BDG}, we obtain that for any $q\ge1$, %{\r why estimate the following}
\begin{align}\label{eq:High}
&\mathbb{E}\left[ \left|
\langle F'(X^n),g(X^n)\delta W^n\rangle 
\langle F'(X^n),X^{n+1} - X^n\rangle\right|^{2q} \right] \\\notag
&\le 
C \mathbb{E}\left[\|F'(X^n)\|^{4q}\|g(X^n)\delta W^n\|^{2q}
\|X^{n+1} - X^n\|^{2q}\right]\notag\\\notag
&\le  C\left(\mathbb{E}\left[(1+\| X^n\|_{L^\infty}^{24q})
\|X^{n+1} - X^n\|^{4q} \right]\right)^{\frac 12} \left(\mathbb{E}\left[  \|g(X^n)\delta W^n\|^{4q}\right]\right)^{\frac 12}\\\notag
&\le C\tau^{2q}.
\end{align}
In the similar manner, it can be verified that for any $q\ge1$,
\begin{align}\label{eq:High1}
\mathbb{E}\left[ \left| \langle F''(X^n)(X^{n+1} - X^n),g(X^n)\delta W^n\rangle\right|^{2q}\right]	\le C\tau^{2q}.
\end{align}
As a result of \eqref{scheme_r_sto-intro}, \eqref{eq:Low}, \eqref{eq:High}, and \eqref{eq:High1}, it holds that for any $q\ge1$,
\begin{gather}\label{eq:r-r}
\mathbb{E}\left[|r^{n+1} - r^n|^{2q}\right]+\mathbb{E}\left[\bigg|r^{n+1} - r^n+\frac{\langle F'(X^n),X^{n+1} - X^n\rangle}{2\sqrt{E_{\textup{p}}(X^n)}}\bigg|^{2q}\right]\le C\tau^q,\\\label{eq:r-r-diff}
\mathbb{E}\left[\bigg|r^{n+1} - r^n-\frac{ \langle F'(X^n),X^{n+1} - X^n\rangle}{2\sqrt{E_{\textup{p}}(X^n)}}\bigg|^{2q}\right]
\le  C\tau^{2q}.
\end{gather}
Applying H\"older's inequality and combining \eqref{eq:r-r} and \eqref{eq:r-r-diff}, we obtain
\begin{equation*}
\mathbb{E}[|J_{3}^n|]
=\mathbb{E}\left[\Big|-|r^{n+1} - r^n|^2+\frac{1}{4E_{\textup{p}}(X^n)}\left| \left\langle F'(X^n),X^{n+1} - X^n\right\rangle\right|^2\Big|\right]
%&=\mathbb{E}\left[\bigg|\bigg(r^{n+1} - r^n+\frac{ \left\langle F'(X^n),X^{n+1} - X^n\right\rangle}{2\sqrt{E_{\textup{p}}(X^n)}}\bigg)\bigg(r^{n+1} - r^n-\frac{\left\langle F'(X^n),X^{n+1} - X^n\right\rangle}{2\sqrt{E_{\textup{p}}(X^n)}}\bigg)\bigg|\right]
\le C\tau^{\frac32}.
\end{equation*}

To estimate $J_6^n$, $J_7^n$, $J_9^n$, and $J_{10}^n$, we notice that by \eqref{eq:r-r} and Lemma \ref{lemma3}, the factor $\frac{r^{n+1}}{\sqrt{E_{\textup{p}}(X^{n})}}-1$ is also of magnitude $\mathscr{O}(\tau^\frac12)$ in $L^q(\Omega;H)$ for any $q\ge1$. That is,
\begin{equation*}%\label{eq:rn+1-E}
\Big\|\frac{r^{n+1}}{\sqrt{E_{\textup{p}}(X^{n})}}-1\Big\|_{L^q(\Omega)}\le C\|r^{n+1}-r^n\|_{L^q(\Omega)}+C\|r^n-\sqrt{E_{\textup{p}}(X^{n})}\|_{L^q(\Omega)}\le C(q)\tau^{\frac12}
\end{equation*}
for any $q\ge1$.
On the other hand, due to \eqref{eq:BDG} and Lemma \ref{coro1}, the increments $g(X^n)\delta W^n$ and $X^{n+1} - X^n$ are both of magnitude $\mathscr{O}(\tau^\frac12)$ in $L^q(\Omega;H)$ for any $q\ge1$. 
Since each of the remainder terms $J_6^n$, $J_7^n$, $J_9^n$, and $J_{10}^n$ contains three factors of magnitude  $\mathcal{O}(\tau^{1/2})$ in $L^q(\Omega;H)$, it can be shown that
\begin{equation*}
\mathbb{E}[|J_6^n|+|J_7^n|+|J_9^n|+|J_{10}^n|]\le C\tau^{\frac32},\qquad n\in\{0,1,\cdots,N-1\}.
\end{equation*}
\hfill$\square$

\section{Proof of Proposition \ref{App_prop}}\label{proof-App_prop}
\begin{proof}
The proof is carried out in three steps. In the first two steps, we establish the $L^2(\mathcal{O})$- and $H^1(\mathcal{O})$-spatial regularity of the mild solution, respectively. In the final step, we apply Lemma~\ref{lemma_tool} to derive the $H^\beta(\mathcal{O})$-spatial regularity estimate of the mild solution for $\beta\in(1,2)$.
%{\r better to add logic}

\emph{Step 1.}
By  Assumption~\ref{assum2} and integrating by parts, we obtain
\begin{align*}
	- \langle\nabla \phi(t), \nabla F'(\phi(t))\rangle 
	&= -\int_{\mathcal{O}}|\nabla \phi(t,x)|^2f''(\phi(t,x))d x\le L_f\|\nabla \phi(t)\|^2\\
	&=L_f\|(-A)^{\frac12} \phi(t)\|^2\le \frac12\|A\phi(t)\|^2+\frac12L_f^2\|\phi(t)\|^2.
\end{align*}
Hence, from \eqref{model}, we can apply It\^o's formula to conclude
\begin{align*}
	d\|\phi(t)\|^2 &= 2\langle\phi(t), -A^2\phi(t)\rangle dt -2\langle\nabla \phi(t), \nabla F'(\phi(t))\rangle dt \\
	& \quad +2\langle\phi(t), g(\phi(t))dW(t)\rangle 
	+ \|g(\phi(t))\|_{\mathcal{L}_2^0}^2dt\\
	&\le -\|A\phi(t)\|^2dt 
	+ C\|\phi(t)\|^2 dt+2\langle\phi(t), g(\phi(t))dW(t)\rangle 
	+ \|g(\phi(t))\|_{\mathcal{L}_2^0}^2dt.            
\end{align*}
By the BDG inequality, H\"older's inequality, and the linear growth of $g$ (see \eqref{assume_Lip}),  we deduce that for any $t\in[0,T]$,
\begin{align*}
	&\mathbb{E}\left[\|\phi(t)\|^{2p}\right] +\mathbb{E}\bigg[\bigg(\int_{0}^t\left\|A\phi(s)\right\|^2ds\bigg)^p \bigg]
	\le C\|\phi^0\|^{2p} 
	+C\mathbb{E}\bigg[\bigg(\int_{0}^t\| \phi(s)\|^{2} ds\bigg)^p\bigg]
	\\
	&\quad +C\mathbb{E}\left[\Big|\int_{0}^t\langle\phi(s), g(\phi(s))dW(s)\rangle\Big|^p \right]
	+C\mathbb{E}\left[\bigg(\int_{0}^t\|g(\phi(t))\|_{\mathcal{L}_2^0}^2\bigg)^p\right]\notag\\     
	&\le C\|\phi^0\|^{2p}
	+C\mathbb{E}\int_{0}^t\| \phi(s)\|^{2p} ds\notag
	+C.
\end{align*}	
Using Gronwall's inequality, we have
\begin{align}\label{App_re1}
	\mathbb{E}\left[\|\phi(t)\|^{2p}\right]  +\mathbb{E}\bigg[\bigg(\int_{0}^t\|A\phi(s)\|^2ds\bigg)^p \bigg]\le C\left(\|\phi^0\|^{2p} +1\right).
\end{align}

\emph{Step 2.}
It follows from Assumption \ref{assum2} and Assumption \ref{assum1} that
\begin{align*}
	\sum_{k=1}^{\infty}\left\langle F''(\phi(s)) g(\phi(s))Q^{\frac{1}{2}}e_k,g(\phi(s))Q^{\frac{1}{2}}e_k\right\rangle&\le  \|F''(\phi(s))\|_{L^1}\sum_{k=1}^\infty
	\| g(\phi(s))Q^{\frac{1}{2}}e_k\|_{L^\infty}^2\\
	&\le C(1+\|\phi(s)\|^2).
\end{align*}
Hence, using the energy evolution law \eqref{dissi_conti} and H\"older's inequality results in 
\begin{align}\label{dissi_conti_energy}
	&E(\phi(t))  +\int_0^t \left\|\nabla\mu(s)\right\|^{2}ds
	\le E(\phi^0)+\int_0^t \left\langle \mu(s),g(\phi(s))dW(s)\right\rangle\\
	&\quad+C\int_0^t
	\|g(\phi(s))\|^2_{\mathcal{L}_2(Q^{\frac12}H,\dot{H}^1(\mathcal{O}))} ds+C\int_0^t (\| \phi(s)\|^2+1) ds.\notag
\end{align}
In view of \eqref{lema1_eq_R6-new},  for any $s\in[0,T]$,
\begin{equation}\label{eq:dissi_conti_energy_Q}
	\|g(\phi(s))\|^2_{\mathcal{L}(Q^{\frac12}H,\dot{H}^1(\mathcal{O}))}\le C(1+\|\nabla \phi(s)\|^2)\le C(1+\|A\phi(s)\|^2+\|\phi(s)\|^2).
\end{equation}
Then, inserting \eqref{pro:eq:phi0} and \eqref{eq:dissi_conti_energy_Q} into \eqref{dissi_conti_energy}, we deduce that for any \(t\in[0,T]\),
\begin{align}\label{dissi_conti_energy1}
	& \mathbb{E}\left[E(\phi(t))^p\right]  + \mathbb{E}\bigg[\bigg(\int_0^t \|\nabla\mu(s)\|^{2}ds\bigg)^p \bigg]\\\notag
	&\le C  
	+C\mathbb{E}\left[\Big|\int_0^t \left\langle \mu(s),g(\phi(s))dW(s)\right\rangle\Big|^p\right]+C\mathbb{E}\bigg[\bigg(\int_0^t \|A\phi(s) \|^2+\|\phi(s) \|^2 ds\bigg)^p\bigg].
\end{align}
Applying the BDG inequality, H\"older's inequality, Assumption \ref{assum1}, and Young's inequality, and recalling the definition $\mu := \frac{\delta E}{\delta \phi} = -\Delta \phi + F'(\phi)$ in \eqref{model}, we obtain
\begin{align}\label{dissi_conti_energy1_T1}
	&\mathbb{E}\left[\Big|\int_0^t \left\langle \mu(s),g(\phi(s))dW(s)\right\rangle\Big|^p\right]\le C\mathbb{E}\bigg[\bigg(\int_0^t\sum_{k=1}^{\infty}\left| \langle \mu(s),g(\phi(s))Q^{\frac{1}{2}}e_k\rangle\right|^2ds\bigg)^{\frac{p}{2}}\bigg]\\
	&\le C\mathbb{E}\bigg[\bigg(\int_0^t\left\|  \mu(s)\right\|^2_{L^1}\sum_{k=1}^{\infty}\|g(\phi(s))Q^{\frac{1}{2}}e_k\|_{L^{\infty}}^2ds\bigg)
	^{\frac{p}{2}}\bigg]\notag\\
	&\le C\mathbb{E}\bigg[\left(\int_0^t\left\|  \Delta\phi(s)\right\|_{L^1}^2ds\right)^{\frac{p}{2}}\bigg]+C\mathbb{E}\bigg[\left(\int_0^t\left\|  {F'(\phi(s))}\right\|_{L^1}^2ds\right)^{\frac{p}{2}}\bigg].\notag
\end{align}
Combining Assumption \ref{assum2}, the Gagliardo--Nirenberg inequality, and Young's inequality yields
\begin{align*}
	\|  F'(\phi(s))\|_{L^1}^2&\le C(1+\|  \phi(s)\|_{L^3}^6)\le C(1+\|  \phi(s)\|_{H^2}^{\frac 32}\|  \phi(s)\|^{\frac 92})\\
	&\le C(1+\|  \phi(s)\|_{H^2}^{2}+\|  \phi(s)\|^{18}),
\end{align*}
which implies that for any $t\in[0,T]$,
\begin{align}\label{dissi_conti_energy2}
	\mathbb{E}\bigg[\bigg(\int_0^t\|  F'(\phi(s))\|_{L^1}^2ds\bigg)^{\frac{p}{2}}\bigg]
	%&\le C + C\mathbb{E}\left[\left(\int_0^t\left\|  \phi(s)\right\|_{L^3}^6ds\right)^{\frac{p}{2}}\right]\notag\\
	%&\le C +C\mathbb{E}\left[\left(\int_0^t\left\|  \phi(s)\right\|_{H^2}^{\frac 32}\left\|  \phi(s)\right\|^{\frac 92}ds\right)^{\frac{p}{2}}\right]\notag\\
	&\le C +C \mathbb{E}\bigg[\bigg(\int_0^t\|  \phi(s)\|_{H^2}^{2}ds\bigg)^{\frac p2}\bigg]\\\notag
	&\quad+C\mathbb{E}\bigg[\bigg(\int_0^t\|  \phi(s)\|^{18}ds\bigg)^{\frac p2}\bigg].
\end{align}
Substituting \eqref{dissi_conti_energy2} into \eqref{dissi_conti_energy1_T1} and
applying \eqref{App_re1}, we arrive at $$\mathbb{E}\bigg[\Big|\int_0^t \langle \mu(s),g(\phi(s))dW(s)\rangle\Big|^p\bigg]\le C.$$
%\begin{align}\label{dissi_conti_energy1_T1_final}
%\mathbb{E}\left[\Big|\int_0^t \left\langle \mu(s),g(\phi(s))dW(s)\right\rangle\Big|^p\right]\le C.
%\end{align}
Plugging this into \eqref{dissi_conti_energy1} and taking \eqref{App_re1} into account, we have that for any $p\ge1$,
\begin{align}\label{dissi_conti_energy3}
	\mathbb{E}\left[E(\phi(t))^p\right]  + \mathbb{E}\bigg[\bigg(\int_0^t \|\nabla\mu(s)\|^{2}ds\bigg)^p 
	\bigg]
	\le C .
\end{align}

\emph{Step 3.}
From the energy definition \eqref{energy_def} and \eqref{dissi_conti_energy3}, we obtain
\(
\mathbb{E}\bigl[\|\nabla\phi(t)\|^{2p}\bigr]\le C(p)
\) for all $p\ge1$.
It follows from Young's inequality and the Sobolev embedding $H^1(\mathcal{O})\hookrightarrow L^6(\mathcal{O})$ that for any $t\in[0,T]$,
\begin{align}\label{eq:Fphi}
	\mathbb{E}\left[\left\|F'(\phi(t))\right\|^p\right]\leq C\mathbb{E}\left[\left\|\phi(t)\right\|_{L^6}^{3p}\right]+C\le C\mathbb{E}\left[\left\|\nabla \phi(t)\right\|^{3p}\right]+C\le C.
\end{align}
Moreover, by Assumption~\ref{assum1}, it holds that $\mathbb{E}[\|g(\phi(t))\|_{\mathcal{L}_2^0}^{p}]\le C$
for any $t\in[0,T]$, which, in combination with  \eqref{eq:Fphi}, allows us to apply Lemma~\ref{lemma_tool} with $\tilde F(s)=F'(\phi(s))$ and $\tilde G(s)=g(\phi(s))$ to conclude \eqref{eq:phi(t)beta}.
\end{proof}

\backmatter

%\bmhead{Supplementary information}
%
%If your article has accompanying supplementary file/s please state so here. 
%
%Authors reporting data from electrophoretic gels and blots should supply the full unprocessed scans for key as part of their Supplementary information. This may be requested by the editorial team/s if it is missing.
%
%Please refer to Journal-level guidance for any specific requirements.

\section*{Declarations}

\noindent \textbf{Funding} This work is supported by the Hong Kong Research Grant Council GRF grants 15302823 and 15301025, NSFC/RGC Joint Research Scheme N$\_$PolyU5141/24, NSFC grant 12522119, NSFC grant 12301526, internal funds (P0041274,P0045336) from Hong Kong Polytechnic University,  and the
CAS AMSS-PolyU Joint Laboratory of Applied Mathematics.

\vskip 1em

\noindent \textbf{Data Availability} Data sharing not applicable to this article as no datasets were generated or analyzed during the current study.

\vskip 1em

\noindent \textbf{Conflict of interest} The authors declare that they have no conflict of interest.
\bibliography{sn-bibliography}% common bib file

\end{document}